\documentclass[11pt,a4paper,twoside]{amsart}

\usepackage{hyperref}
\usepackage{amssymb}
\usepackage{amsfonts}
\usepackage{mathrsfs}
\usepackage{amsmath,amscd}
\usepackage{amssymb}
\usepackage{amsthm}
\usepackage{enumerate}
\usepackage[english,francais]{babel}
\usepackage{hyperref}

\usepackage{tikz}
\usetikzlibrary{calc}
\usetikzlibrary{fadings,arrows}

\usepackage[all]{xy}

\usepackage{t1enc}

\newtheorem{e-proposition}[theorem]{Proposition}

\newtheorem{e-definition}[theorem]{Definition\rm}

\newtheorem{example}{\it Example\/}
\newtheorem{theoreme}{Th\'eor\`eme}[section]
\newtheorem{lemme}[theoreme]{Lemme}

\newtheorem{proposition}[theoreme]{Proposition}
\newtheorem{corollaire}[theoreme]{Corollaire}
\newtheorem{definition}[theoreme]{D\'efinition\rm}
\newtheorem{remarque}{\it Remarque}
\newtheorem{exemple}{\it Exemple\/}

\newcommand{\bm}[1]{\mbox{\boldmath{$#1$}}}

\newcommand{\s}[0]{\ensuremath{P}}

\newcommand{\R}[0]{\ensuremath{R}}
\newcommand{\Q}[0]{\ensuremath{Q}}

\newcommand{\C}[0]{\ensuremath{C}}
\newcommand{\CS}[0]{\ensuremath{CS}}

\newcommand{\RR}[0]{\ensuremath{\mathbb{R}}}
\newcommand{\AF}[0]{\ensuremath{\mathbb{A}}}
\newcommand{\ZZ}[0]{\ensuremath{\mathbb{Z}}}
\newcommand{\QQ}[0]{\ensuremath{\mathbb{Q}}}

\newcommand{\KK}[0]{\ensuremath{\mathbb{K}}}
\newcommand{\CC}[0]{\ensuremath{\mathbb{C}}}

\newcommand{\EN}[1]{\ensuremath{\operatorname{\Gamma^{\star}(#1)}}}

\newcommand{\spec}[0]{\ensuremath{\operatorname{Spec}}}
\newcommand{\specrel}[0]{\ensuremath{\operatorname{\bf Spec}}}

\newcommand{\Bo}[0]{\ensuremath{\operatorname{\mathcal{B}}}}

\newcommand{\Bl}[0]{\ensuremath{\operatorname{B}}}
\newcommand{\N}[0]{\ensuremath{\operatorname{N}}}
\newcommand{\M}[0]{\ensuremath{\operatorname{M}}}

\newcommand{\FI}[0]{\ensuremath{\operatorname{FI}}}
\newcommand{\sing}[0]{\ensuremath{\operatorname{Sing}}}
\newcommand{\Ess}[0]{\ensuremath{\operatorname{Ess}}}
\newcommand{\Hom}[0]{\ensuremath{\operatorname{Hom}}}

\newcommand{\ord}[0]{\ensuremath{\operatorname{Ord}}}

\newcommand{\degt}[1]{\ensuremath{\operatorname{Deg}_{t} #1}}
\newcommand{\p}[0]{\ensuremath{\operatorname{p}}}

\newcommand{\f}[0]{\ensuremath{\operatorname{f}}}
\newcommand{\F}[0]{\ensuremath{\operatorname{F}}}
\newcommand{\e}[0]{\ensuremath{\operatorname{e}}}
\newcommand{\g}[0]{\ensuremath{\operatorname{g}}}

\newcommand{\sym}[0]{\ensuremath{\operatorname{Sym}}}
\newcommand{\E}[0]{\ensuremath{\operatorname{E}}}
\newcommand{\D}[0]{\ensuremath{\operatorname{D}}}

\newcommand{\ens}{\mathcal{E}ns}
\newcommand{\salg}{\s\!-\!\mathcal{A}lg}

\newcommand{\HS}[1]{\operatorname{HS}^{#1}_{\R\slash\s}}
\newcommand{\mJet}[3]{\operatorname{#1(#2)}_{#3}}
\newcommand{\mjp}[1]{\operatorname{\p}_{#1}}
\newcommand{\mts}[1]{\operatorname{T(S)}_{#1}}
\newcommand{\mt}[1]{\operatorname{T}_{#1}}
\newcommand{\codim}[0]{\operatorname{codim}}
\newcommand{\arcSpace}[2]{\operatorname{#1(#2)}_{\infty}}
\newcommand{\HSF}[3]{\operatorname{HS}^{#1}_{#2 \slash #3}}
\newcommand{\DerR}[2]{\operatorname{Der}^{#1}_{\s}(\R,#2)}
\newcommand{\com}[1]{\ensuremath{#1^{\star}}}

\newcommand{\h}[0]{\ensuremath{\operatorname{h}}}
\newcommand{\hk}[0]{\ensuremath{\operatorname{k}}}

\newcommand{\appui}[1]{\h_{#1}}

 \usetikzlibrary{arrows,chains,matrix,positioning,scopes}
\makeatletter
\tikzset{join/.code=\tikzset{after node path={%
\ifx\tikzchainprevious\pgfutil@empty\else(\tikzchainprevious)%
edge[every join]#1(\tikzchaincurrent)\fi}}}
\makeatother
\tikzset{>=stealth',every on chain/.append style={join},
         every join/.style={->}}

\begin{document}
\selectlanguage{francais}

\title[D\'eformation des espaces de m-jets et d'arcs ]{Familles d'espaces de m-jets et d'espaces d'arcs}

\author{Maximiliano LEYTON-\'ALVAREZ}
\address{Instituto Matem\'atica y F\'isica, Universidad de Talca,
Camino Lircay S$\backslash$N, Campus Norte, Talca, Chile}

\email{leyton@inst-mat.utalca.cl}

\date{\today}

\thanks{Soutien Financier: Projet FONDECYT $\rm{N}^\circ: 11121163$ }

\begin{abstract}
\selectlanguage{francais}
Soit $V$ une vari\'et\'e alg\'ebrique d\'efinie sur un corps $\KK$ alg\'ebriquement clos et de
caract\'eristique nulle. Comme les espaces de $m$-jets $V_m$ et l'espace d'arcs $V_{\infty}$ fournissent des informations sur la
g\'eom\'etrie  de la vari\'et\'e $V$, il est donc naturel de se poser les questions suivantes:
Quand est-ce qu'une d\'eformation de $V$ induit une d\'eformation des espaces $V_m$, $1\leq m\leq \infty$ ? 
Si l'on consid\`ere une d\'eformation  de $V$ qui admet une r\'esolution simultan\'ee 
\`a plat, comment variera l'image de l'application de Nash $\mathcal{N}_{V}$? Dans la section
\ref{sec:Fam-mJ-arc}, on donne quelques r\'eponses partielles \`a ces questions.

Dans la section \ref{sec:Deux-fam-exemp-hyp}, on montre deux familles d'hypersufaces de $\AF_{\KK}^4$ o\`u l'application 
de Nash est bijective. De plus, dans chaque cas,  on donne explicitement une d\'esingularization o\`u toutes
les composantes irr\'eductibles de la fibre exceptionnelle sont des diviseurs essentiels. 
Il faut remarquer que  dans la litt\'erature on ne trouve pas  beaucoup d'exemples de ce type.


\vskip 0.5\baselineskip

\selectlanguage{english}
\noindent{A{\tiny BSTRACT}.}
\noindent {\bf  Families of m-jet spaces and arc spaces.}

Let $V$ be an algebraic variety defined over an algebraically closed field $\KK$ of characteristic zero.  The $m$-jet spaces  $V_m$ and  the space of arcs  $V_{\infty}$ 
provide the information on the geometry of the  variety $V$, therefore it is natural to ask the following
questions:  When will a deformation of $V$ induce a deformation of the spaces $V_m$, $1\leq m\leq \infty$?  If one considers a
deformation of $V$ which admits a flat simultaneous resolution, how will the image of the Nash
application $\mathcal{N}_V$ vary? 
In section \ref{sec:Fam-mJ-arc} some partial answers to these questions can be found.

 In section \ref{sec:Deux-fam-exemp-hyp} two families of  hypersurfaces of $\AF_{\KK}^4$ are shown in which the Nash application is
bijective.  What's more, in each case a desingularization in which all the irreducible components of
the exceptional fiber are essential divisors is explicitly given.  It is important to note that very few
examples of this type are found in the literature.
\end{abstract}
\maketitle
\selectlanguage{francais}
\section{Introduction}
\label{sec:int}

Soit $V$ une vari\'et\'e alg\'ebrique sur un corps $
\KK$ alg\'ebriquement clos  ou une vari\'et\'e analytique complexe ($\KK:=\CC$). 
Dans la fin des ann\'ees 60, John Nash a introduit l'espace d'arcs 
$V_{\infty}$   dans le but d'obtenir  des informations sur la 
g\'eom\'etrie locale du lieu singulier $\sing V$ de $V$ (voir \cite{Nas95}).
 L'espace d'arcs $V_{\infty}$ peut \^etre interpr\'et\'e comme
 l'ensemble de tous les arcs $\spec \KK[[t]]\rightarrow V$,
 muni d'une structure ``naturelle'' de sch\'ema sur $\KK$.  
 Une {\it composante de Nash} associ\'ee \`a $V$ est  une famille d'arcs sur  $V$  passant par  
 le lieu singulier de $V$.  Un diviseur essentiel $\E$ sur $V$ est grosso modo un 
 diviseur exceptionnel dans 
une d\'esingularisation de $V$ qui appara\^it comme une composante irr\'eductible de
la fibre exceptionnelle  de toute  d\'esingularisation possible de $V$.
Nash  a d\'efini une application de l'ensemble des composantes de Nash associ\'ees \`a $V$
dans l'ensemble des diviseurs  essentiels sur $V$ et il  a d\'emontr\'e qu'elle est injective;
cette application est connue sous le nom {\it d'application de Nash}, not\'ee $\mathcal{N}_V $.
Nash pose donc la  question suivante: Est-ce que la application $\mathcal{N}_V $ est 
bijective?  (pour plus de d\'etails voir la section
\ref{ssec:nash}).
\'Etant donn\'e explicitement une vari\'et\'e $V$ et une d\'esingularisation $\pi$, on remarque que d\'eterminer si une composante irr\'eductible de la fibre 
exceptionnelle de $\pi$ est ou non un diviseur essentiel reste un probl\`eme difficile.
L'approche de Nash est donc tr\`es int\'eressante.\\

De nombreux math\'ematiciens, qu'on ne peut pas tous citer ici, 
ont apport\'e de nouvelles contributions originales \`a l'\'etude du probl\`eme de Nash, 
surtout dans le cas de vari\'et\'es
de dimension $2$ et $3$ (voir par exemple \cite{Lej80}, \cite{Reg95}, 
\cite{Gole97}, \cite{LeRe99}, \cite{IsKo03}, \cite{Ish05}, 
\cite{Ish06},\cite{PlPo06}, \cite{Gon07},\cite{Mor08}, \cite{Ple08},\cite{PlPo08}, \cite{Pet09}, 
\cite{Ley11a}, \cite{Bob12}, \cite{PlSp12}, \cite{dFDo14}).

Dans le cas des surfaces,  le probl\`eme est rest\'e ouvert jusqu'\`a l'ann\'ee 2011, quand les auteurs de l'article
\cite{BoPe12a} ont donn\'e une r\'eponse positive \`a la question de Nash  pour toutes les singularit\'es de surfaces sur $\CC$.

Dans l'article \cite{IsKo03} de l'ann\'ee $2003$, les auteurs 
ont d\'ecouvert le premier exemple d'une vari\'et\'e $V$  telle que l'application
de Nash  $\mathcal{N}_V $ n'est pas bijective; cette vari\'et\'e est une 
hypersurface de $\AF^5_{\KK}$ ayant une unique singularit\'e isol\'ee.
Les exemples de vari\'et\'es de dimension trois o\`u l'application de Nash n'est pas bijective ont apparu
au cours de l'ann\'ee $2012$, (voir les articles \cite{dFe13} et \cite{Kol12});
ces exemples sont  des hypersurfaces de $\AF^4_{\KK}$ ayant une unique singularit\'e 
isol\'ee. Vous trouverez plus d'exemples, pour dimensions sup\'erieures ou \'egales \`a trois, dans l'article
\cite{JoKo13}. Comme il est indiqu\'e dans l'article ant\'erieurement cit\'e, tous les  exemples connus de 
vari\'et\'es o\`u la r\'eponse \`a la question de Nash est n\'egative, 
sont obtenus en utilisant la m\^eme m\'ethode centrale. Les auteurs d\'efinissent donc
les diviseurs {\it tr\`es essentiels} sur une 
vari\'et\'e $V$, qui sont grosso modo les diviseurs
essentiels sur $V$ qui ne satisfont pas les hypoth\`eses de la m\'ethode  centrale.  
La question qui se pose est alors la suivante: 
Est-ce que l'image de l'application de Nash $\mathcal{N}_V$ correspond l'ensemble des 
diviseurs tr\`es essentiels sur $V$?
Cette question est tr\`es int\'eressante, mais c'est un probl\`eme que nous 
n'aborderons pas ici.\\

Mais au-del\`a du probl\`eme de Nash, c'est important d'\'etudier les espaces de m-jets et l'espace de arcs,
parce qu'ils fournissent des informations sur la g\'eom\'etrie  de la vari\'et\'e sous-jacente, par exemple, voir les articles suivants: 
\cite{Mus01}, \cite{LaMu09}, 
\cite{DeLo99} y \cite{DelO02}.
Les espaces d'arcs et les espaces de m-jets ont \'et\'e  utilis\'es  aussi pour solutionner certains probl\`emes 
qui ont \'et\'e consid\'er\'es comme des probl\`emes difficiles \`a r\'esoudre,  par exemple,  
la Conjecture de Batyrev sur les vari\'et\'es de Calabi-Yau (voir \cite{Kon95}).

Dans le but de mieux comprendre la relation entre la vari\'et\'e $V$ et les espaces $V_{m}$, $1\leq m \leq \infty$,
il est naturel de  se poser la question suivante: Est-ce qu'une d\'eformation de la vari\'et\'e $V$ induit une d\'eformation 
(compatible) des espaces de m-jets $V_m$ et de l'espace d'arcs $V_{\infty}$?
L'examen de cette question est le sujet de la premi\`ere moiti\'e de la section \ref{sec:Fam-mJ-arc}.

Dans la deuxi\`eme moiti\'e de cette section, on aborde et on donne quelques r\'esultats partiels au probl\`eme suivant:
si on consid\`ere une d\'eformation d'une d\'esingularization  
$\pi:X\rightarrow V$ (en pr\'eservant la condition d'\^etre une d\'esingularization)  qui induit une d\'eformations de la fibre exceptionnelle
$\pi^{-1}(\sing V)$, comment variera l'image de l'application de 
Nash $\mathcal{N}_{V}$?
 Plus pr\'ecis\'ement, si $\E$ est un diviseur  essentiel (ou tr\`es essentiel) sur $V$, est-ce que 
 la ``d\'eformation'' de $\E$ sera un diviseur  essentiel (ou tr\`es essentiel) 
 sur la ``d\'eformation'' de $V$? 
 (pour plus de pr\'ecisions et de d\'etails voir la section \ref{ssec:Fam-apl_nash}).\\

Dans la section \ref{sec:Deux-fam-exemp-hyp}, on consid\`ere deux familles d'hypersurfaces normales de $\AF_{\KK}^4$ et
on montre que pour chaque hypersurface consid\'er\'ee, l'application de Nash qui y est associ\'ee est bijective.
De plus, dans chaque cas on donne explicitement  une  d\'esingularizations o\`u tous les composantes irr\'eductibles de la fibre exceptionnelle
sont des diviseurs essentiels. \\

Dans la Section \ref{se:Pre-Rap} on fait quelques rappels sur les d\'erivations de Hasse-Schimidt, l'espaces de $m$-jets, l'espace d'arcs,
le probl\`eme de Nash, les vari\'et\'es toriques, et les r\'esolutions de singularit\'es d'hypersurfaces.

\section*{Remerciements} Je tiens \`a remercier Mark Spivakovsky pour les nombreuses discussions et son constant encouragement.
Je remercie \'egalement Antonio Campillo pour l'int\'er\^et qu'il a port\'e \`a mon travail. 

\section{Pr\'eliminaires et Rappels}
\label{se:Pre-Rap}
\subsection{D\'erivations de Hasse-Schmidt}
\label{ssec:h-s}
Dans cette section, on donne quelques r\'esultats et notions de base sur les d\'erivations de  Hasse-Schmidt, 
l'espace de $m$-jet relatif et l'espace d'arcs relatif. 
Notre but ici n'est pas de donner un expos\'e exhaustif de la th\'eorie, mais plut\^ot de donner un aper\c{c}u de celle-ci dans le cadre qui nous int\'eresse.
Pour plus d\'etails, voir \cite{Voj07}.\\

Soit $\KK$ un corps alg\'ebriquement clos de caract\'eristique nulle, 
$\s$ une $\KK$-alg\`ebre et  $\R$, $\Q$ deux $\s$-alg\`ebre.\\

Une {\it d\'erivation de  Hasse-Schmidt   d'ordre}  $m\in \ZZ_{\geq 0}$ de $\R$ vers $\Q$ est un $m\!+\!1$-uple  $(D_0,\cdots, D_m)$, o\`u   $D_0:R\rightarrow Q$  est un homomorphisme  
de $\s$-alg\`ebre et  $D_i:R\rightarrow Q$, $1\leq i\leq m$, est  un homomorphisme de groupes ab\'eliens, qui satisfait les  propri\'et\'es suivantes:
\begin{enumerate}

 \item  $D_i(p)=0$ pour tout $p\in \s$ et pour tout $1\leq i\leq m$;
 \item  Pour tous les \'el\'ements $x$ et  $y$  de $\R$ et pour tout entier $1\leq k \leq m$, on a:
 \begin{center}
  $D_k(xy)=\sum_{i+j=k}D_i(x)D_j(y)$.
   \end{center}
 \end{enumerate}
 \vspace{0.1cm}
 
 On remarque que si $\phi:Q\rightarrow Q'$ est un homomorphisme  de $\s$-alg\`ebre et $(D_0,\cdots, D_m)$ 
 une d\'erivation de Hasse-Schmidt de $\R$ vers $\Q$, alors $(\phi\circ D_0,\cdots, \phi \circ D_m)$
 est une  d\'erivation de Hasse-Schmidt de $\R$ vers $\Q'$. Par cons\'equent, on peut d\'efinir le foncteur suivant:
  \vspace{0.3cm}
 \begin{center}
  $\DerR{m}{\;\cdot\;}:\salg\rightarrow \ens;$ $\Q\mapsto \DerR{m}{\Q}$
 \end{center}
  \vspace{0.3cm}
  o\`u $\salg$  (resp. $\ens$) est la cat\'egorie des $\s$-alg\`ebres (resp. des ensembles) et $\DerR{m}{\Q}$ est l'ensemble des 
  d\'erivations de Hasse-Schmidt d'ordre $m$ de $\R$ vers $\Q$.\\

 De la m\^eme fa\c{c}on, on peut d\'efinir la d\'erivation de Hasse-Schmidt   
 d'ordre $\infty$ et le foncteur $\DerR{\infty}{\;\cdot\;}$ en utilisant, 
 \`a la place du  $m\!+\!1$-uple  $(D_0,\cdots, D_m)$,
 une suite infinie $D_0,D_1,\cdots $.\\
 
On remarque que le foncteur des  d\'erivations de Hasse-Schmidt   d'ordre $1$ $\DerR{1}{\;\cdot\;}$  est le foncteur de 
d\'erivations de K\"ahler $\operatorname{Der}_{\s}(\R,\;\cdot\;)$. Il est bien connu que ce foncteur 
est repr\'esentable par l'alg\`ebre sym\'etrique  du $\R$-module des diff\'erentiels de K\"ahler , 
c'est-\`a-dire:

\begin{center}
 $\operatorname{Der}_{\s}(\R,\;\cdot\;)\cong \Hom_{\s}(\sym(\Omega_{\R/\s}), \;\cdot\;)$,
\end{center}
La proposition suivante est une g\'en\'eralisation du r\'esultat ci-dessus.\\

Soit $\HS{m}$ la $\R$-alg\`ebre quotient de l'alg\`ebre 
$\sym( \displaystyle\mathop{\oplus}\limits_ {i=1}^{m} \displaystyle\mathop{\oplus}\limits_{ x\in R} \R d^{i}x)$ par l'id\'eal $I$ 
engendr\'e par la r\'eunion des ensembles $r_1$, $r_2$ et $r_3$, o\`u
$r_1:=\{d^i(x+y)-d^i(x)-d^i(y)\mid x,y\in \R; \; 1\leq  i\leq  m\},$
 $r_2:=\{d^i(s)\mid s \in \s; \; 1\leq i\leq  m\}$ et
 $r_3:=\{d^k(xy)-\sum_{i+j=k}d^i(x)d^i(y)\mid x,y\in \R; \; 1\leq k\leq  m\}$.\\

 On remarque que  $\HS{0}= \R$ et que  $\HS{1}= \sym(\Omega_{\R/\s})$, par contre, dans le cas $m\geq 2$,  l'alg\`ebre $\HS{m}$ devient 
 beaucoup plus difficile \`a comprendre.

On remarque aussi que $\HS{m}$ est une alg\`ebre gradu\'ee par la graduation $d^ix\mapsto i$. D'apr\`es la d\'efinition de 
l'alg\`ebre $\HS{m}$, on obtient, de fa\c{c}on naturelle, une suite d'homomorphismes d'alg\`ebres gradu\'ees:
\vspace{0.3cm}
\begin{center}
 
$\HS{0}\rightarrow \HS{1}\cdots\rightarrow\HS{m}\rightarrow \HS{m+1}\rightarrow \cdots$
\end{center}
\vspace{0.3cm}

On note  $\HS{\infty }$, la limite inductive du syst\`eme inductif $\HS{m}$, c'est-\`a-dire:
\vspace{0.3cm}
\begin{center}
 $\HS{\infty }:=\displaystyle\mathop{\lim}\limits_{m\rightarrow}\HS{m}$.
\end{center}
 
\vspace{0.2cm}
\begin{proposition}[\cite{Voj07}] \label{pro:HS-repre} L'alg\`ebre de Hasse-Schmidt $\HS{m}$, $0\leq m\leq \infty$, repr\'esente le foncteur  de  d\'erivations de Hasse-Schmidt   $\DerR{m}{\;\cdot\;}$, c'est \`a dire
\vspace{0.3cm}
\begin{center}
 $\DerR{m}{\Q}\cong \Hom_{\s}(\HS{m},Q)$,
\end{center}
\vspace{0.3cm}
o\`u $\Q$ est une $\s$-alg\`ebre quelconque. 
 
\end{proposition}
L'alg\`ebre de Hasse-Schimidt $\HS{m}$, $0\leq m\leq \infty$  satisfait les  propri\'et\'es de localisation suivantes:
\begin{proposition}[\cite{Voj07}]
Soient  $0\leq m\leq \infty$  et  $L$ (resp. $L_1$) un sous-ensemble multiplicatif de l'alg\`ebre $R$ (resp. $\s$). Alors, on a: 
\begin{enumerate}
 \item $\HS{m}[L^{\;-1}]\cong \HSF{m}{R[L^{\;-1}]}{\s}$;
 \item si le morphisme $\s\rightarrow R$ se factorise par le morphisme canonique $\s\rightarrow \s[L_1^{\;-1}]$, alors 
 $\HS{m}\cong \HSF{m}{R}{\s[L_1^{\;-1}]}$
\end{enumerate}
\end{proposition}

Soit $f:X\rightarrow Y$   un morphisme de sch\'emas.  En prenant un recouvrement par des ouverts affines de $X$  
compatible avec un recouvrement par des ouverts affines de $Y$ et en utilisant  les propri\'et\'es de localisation ci-dessus,
on peut d\'efinir un faisceau quasi-coh\'erent, $\HSF{m}{\mathcal O_{X}}{\mathcal{O}_{Y}}$, $1\leq m\leq \infty$,
de $\mathcal{O}_X$-Alg\`ebres qui ne d\'epend pas des recouvrements choisis.

\begin{definition}
 Pour chaque $0\leq m\leq \infty$, on pose:
 \begin{center}
 $\mJet{X}{Y}{m}:=\specrel \HSF{m}{\mathcal O_{X}}{\mathcal{O}_{Y}}$
 \end{center}
o\`u $\specrel \HSF{m}{\mathcal O_{X}}{\mathcal{O}_{Y}}$  est le spectre relatif 
du faisceau quasi-coh\'erent de  $\mathcal O_{X}$-Alg\`ebres  $\HSF{m}{\mathcal O_{X}}{\mathcal{O}_{Y}}$.
Si $Y:=\spec \KK$, on pose $X_{m}:=\mJet{X}{\spec \KK}{m}$.
\vspace{0.2cm}

Pour $0\leq m < \infty$ (resp. $m=\infty$), le sch\'ema $\mJet{X}{Y}{m}$ (resp. $\arcSpace{X}{Y}$) est appel\'e
l'espace de $m$-jets (resp. d'arcs) du morphisme $f:X\rightarrow Y$.

\end{definition}
 Les r\'esultats suivants justifient la terminologie ``Espace de $m$-jet et Espace d'arc du morphisme 
 $f:X\rightarrow Y$''.
\vspace{0.2cm}

On remarque qu'on peut d\'efinir les applications suivantes:
\vspace{0.5cm}
\begin{center}
 $\begin{array}{ccl}
 \DerR{m}{\Q}&\rightarrow & \Hom_{\s}(R,Q[t]/(t^{m+1}));\\
 
 (D_0,...,D_m)&\rightarrow & \left\{\begin{array}{l} 
 R\rightarrow Q[t]/(t^{m+1}); \\ x\mapsto D_0(x)+\cdots+D_m(x)t^m\end{array} \right.
                                                                                                      
\end{array}$
\end{center}
\vspace{0.5cm}

\begin{center}

$\begin{array}{ccl}
 \DerR{\infty}{\Q}&\rightarrow & \Hom_{\s}(R,Q[[t]]);\\
 
 D_0,D_1\cdots &\rightarrow & \left\{\begin{array}{l} 
 R\rightarrow Q[[t]]; \\ x\mapsto D_0(x)+D_1(x)t+\cdots\end{array} \right.
                                                                                                      
\end{array}$
\end{center}
En utilisant les applications ci-dessus, on peut d\'emontrer le r\'esultat  suivant
\begin{proposition}[\cite{Voj07}]
\label{pro:fonct-repre-mjet}
 Le foncteur $Q\mapsto \DerR{m}{Q}$, $0\!\leq\!  m\!< \!\infty$, (resp. $Q \mapsto \DerR{\infty}{Q}$) 
 est isomorphe au foncteur  $Q\mapsto \Hom_{\s}(R,Q[t]/(t^{m+1}))$ 
 (resp.   $Q\mapsto \Hom_{\s}(R,Q[[t]])$).
\end{proposition}

Les th\'eor\`emes suivants sont une cons\'equence des Propositions \ref{pro:HS-repre} et \ref{pro:fonct-repre-mjet}
\begin{theoreme} \label{th:prop-fonct-mjet} Soit $f:X\rightarrow Y$ un morphisme de sch\'emas. Alors,
l'espace de $m$-jet $\mJet{X}{Y}{m}$  de $f:X\rightarrow Y$ repr\'esente le foncteur 
de $m$-jets, c'est-\`a-dire

\begin{center}
$\Hom_Y(Z\times_{\KK}\spec \KK[t]/(t^{m+1}), X)\cong \Hom_{Y}(Z,\mJet{X}{Y}{m})$, 

\end{center}
o\`u $Z$ est un $Y$-Sch\'ema quelconque.
\end{theoreme}

\begin{theoreme}\label{th:prop-fonct-arc} Soit $f:X\rightarrow Y$ un morphisme de sch\'emas. Alors, l'espace d'arcs  $\arcSpace{X}{Y}$ de 
$f:X\rightarrow Y$ repr\'esente le foncteur  d'arcs, c'est-\`a-dire

\begin{center}
 $\Hom_Y(Z\widehat{\times}_{\KK}\spec \KK[[t]], X)\cong \Hom_{Y}(Z,\arcSpace{X}{Y})$
\end{center}
o\`u $Z$ est un $Y$-Sch\'ema quelconque et $Z\widehat{\times}_{\KK}\spec \KK[[t]]$ est le compl\'et\'e formel  du sch\'ema $ Z\times_{\KK}\spec \KK[[t]] $  le long du sous-sch\'ema 
 $ Z\times_{\KK}\spec{\KK} $.
\end{theoreme}

Les espaces de $m$-jets et l'espace d'arcs d'un morphisme satisfont la propri\'et\'e suivante:
\begin{proposition}[\cite{Voj07}]
 \label{pr:mjets-chang-base}
 Soient $f:X\rightarrow Y$  et $Y'\rightarrow Y$ deux morphisme de sch\'emas, et on pose $X':=X\times_Y Y'$. Alors,
 $\mJet{X'}{Y'}{m}\cong  \mJet{X}{Y}{m}\times_YY'$, sur $Y'$ pour tout $0 \leq m\leq \infty$ En particulier, 
 il existe un diagramme  commutatif:
 
 \begin{center}
\begin{tabular}{lr}

 $\shorthandoff{;:!?} 
\xymatrix @!0 @R=1.5pc @C=2pc{  
\mJet{X'}{Y'}{m}\ar@{->}[rrrr] \ar@{->}[dd]&& & 
& \mJet{X}{Y}{m} \ar@{->}[dd]\\
   &&\square&& \\
  Y' \ar@{->}[rrrr] &&& &  Y 
}$   
\end{tabular}
\end{center}
\end{proposition}

\subsection{L'espace  d'arcs et le probl\`eme de Nash}
\label{ssec:nash}
Soient $\KK$ un corps alg\'ebriquement clos de caract\'eristique nulle,  
$V$ une vari\'et\'e alg\'ebrique normale sur $\KK$  et   $\pi:X\rightarrow V$ 
une d\'esingularisation divisorielle de $V$.\\

\'Etant donn\'e une d\'esingularization quelconque $\pi':X'\rightarrow V$, l'application birationnelle
 $(\pi')^{-1}\circ\pi:X\dashrightarrow X'$ est bien d\'efinie en codimension $1$ ($X$ est une vari\'et\'e normale).
 Si $\E$ est composante irr\'eductible de la fibre exceptionnelle de $\pi$, alors il existe  un ouvert 
 $\E^{0}$ de $\E$  sur lequel l'application  $(\pi')^{-1}\circ\pi$ est bien d\'efinie.
 Le diviseur $\E$ est appel\'e  {\it diviseur essentiel  sur} $V$ si pour toute d\'esingularisation 
 $\pi':X'\rightarrow V$,  l'adh\'erence de Zariski   $\overline{(\pi')^{-1}\circ\pi(\E^0)}$ 
 est une composante irr\'eductible de $(\pi')^{-1}(\sing V)$, o\`u $\sing V$ est le lieu singulier de $V$. 
 On note $\Ess(V)$ l'ensemble des diviseurs essentiels sur $V$.\\

On rappelle que  $V_{\infty}$ est l'espace  d'arcs sur $V$ et que les $K$-points de $V_{\infty}$
sont en correspondance bijectives avec les   $K$-arcs sur $V$.
Soient $\alpha\in V_{\infty}$ et $\KK_{\alpha}$  le corps r\'esiduel du point $\alpha$.
Par abus de notation, on note aussi  $\alpha$ le $\KK_{\alpha}$-arc qui correspond au point $\alpha\in V_{\infty}$.
Soit $\p:V_{\infty}\rightarrow V$ la projection canonique $\alpha\mapsto \alpha(0)$, o\`u $0$ 
est le point ferm\'e de $\spec \KK_{\alpha}[[t]]$.  On note $\mathcal{CN}(V)$ 
l'ensemble des composantes irr\'eductibles de  $V_{\infty}^{s}:= \p^{-1}(\sing V)$. Nash a d\'emontr\'e  que l'application  
$\mathcal{N}_{V}:\mathcal{CN}(V)\rightarrow \Ess(V)$ qui associe \`a $C\in \mathcal{CN}(V)$ l'adh\'erence 
$\overline{\{\widehat{\alpha}_{C}(0)\}}$ est une application bien  d\'efinie et injective, o\`u $\widehat{\alpha}_{C}$ 
est le rel\`evement \`a $X$ du point g\'en\'erique $\alpha_{C}$ de $C$, c'est-\`a-dire 
$\pi\circ\widehat{\alpha}_{C}=\alpha_{C}$. Le probl\`eme de Nash consiste \`a \'etudier la surjectivit\'e  de  $\mathcal{N}_{V}$.\\

Soit $\E$ une composante irr\'eductible de la fibre exceptionnelle du morphisme $\pi:X\rightarrow V$.
On pose $\N'(\E):=\{\alpha\in V_{\infty}\backslash (\sing V)_{\infty}\mid \widehat{\alpha}(0)\in  \E\}$ et
on note $\N(\E)$ l'adh\'erence dans  $V_{\infty}$ de  $\N'(\E)$.
 On peut montrer que $\N(\E)$ est irr\'eductible et que $V_{\infty}^{s}=\bigcup_{\E\in \Ess(V)}\N(\E)$.
Les morphismes $\omega:\spec K[[s,t]]\rightarrow V$ sont appel\'es $K$-wedges  sur $V$. Les  $K$-wedges  sur $V$
sont en correspondance bijectives avec les   $K[[s]]$-points de $V_{\infty}$. 
L'image du point ferm\'e (resp. du point g\'en\'erique) de $\spec K[[s]]$ dans $V_{\infty}$ 
est appel\'e le centre (resp. l'arc g\'en\'erique)  du $K$-wedge $\omega$.\\

Soit $\E$ un diviseur essentiel sur $V$. 
Un $K$-wedge $\omega$ est appel\'e  {\it $K$-wedge admissible centr\'e  en} $N(\E)$ 
si le centre (resp. l'arc g\'en\'erique)  de $\omega$  est le point g\'en\'erique de $N(\E)$ 
(resp. appartient \`a $V_{\infty}^{s}$). Dans  \cite{Reg06} (voir aussi l'article \cite{Lej80}), l'auteur montre que  
$\E$ appartient \`a l'image de l'application de Nash $\mathcal{N}_V$  si et seulement si, pour tout 
 corps d'extension $K$  du corps r\'esiduel du point g\'en\'erique de $\N(\E)$, 
tout $K$-wedge admissible centr\'e en $\N(\E)$ se rel\`eve \`a $X$, c'est-\`a-dire il existe un $K$-wedge 
$\widehat{\omega}$ sur $X$ tel que $\pi\circ \ \widehat{\omega}=\omega$.
\subsection{Rappels sur les Vari\'et\'es Toriques et la R\'esolution des Singularit\'es isol\'ees d'hypersurfaces}
\label{ssec:torique}

Soient  $\KK$ un corps alg\'ebriquement clos de caract\'eristique quelconque 
et $\KK^{\star}$ son groupe multiplicatif. Un  {\it tore alg\'ebrique} 
$T$ de dimension $n$ est une vari\'et\'e alg\'ebrique affine isomorphe \`a $(\KK^{\star})^n$.  
La structure de groupe multiplicatif  de $(\KK^{\star})^n$ 
induit une structure de groupe sur  $T$.
\'Etant donn\'e un tore alg\'ebrique $T$ sur $\KK$, 
une vari\'et\'e torique $V$ est une vari\'et\'e alg\'ebrique contenant  $T$, 
comme ouvert de Zariski dense, et munie d'une action du tore alg\'ebrique, $T\times V\rightarrow V$, 
prolongeant  $T\times T\rightarrow T$.
Toutes les vari\'et\'es toriques que nous consid\'erons dans la suite sont normales.\\

Les vari\'et\'es toriques ont une relation \'etroite avec la g\'eom\'etrie convexe \'el\'ementaire.
Dans la suite on donnera quelques d\'etails de cette description.\\

Soit $\N:=\ZZ^{n}$ muni de sa base standard $\{\e_1,\e_2,...,\e_{n}\}$ 
Le dual de $\N$, not\'e $\M$, est identifi\'e avec $\ZZ^{n}$  
au moyen de  la forme bilin\'eaire standard 
$\langle \e_i,\e_j\rangle =\delta_{ij}$, $1\leq i,j\leq n$.

Un  {\it c\^one convexe  poly\'edral} ou plus simplement  {\it c\^one} est tout ensemble 
$\sigma$  de  $\N_{\RR}:=\N\otimes_{\ZZ}\RR$ (resp. $M_{\RR}:=\M\otimes_{\ZZ}\RR$) 
tel qu'il existe un entier $s$ et des points $v_i$, $1\leq i\leq s$ de $\N_{\RR}$ 
(resp. $M_{\RR}$)  tels  
que $\sigma$ soit l'ensemble $\langle v_1,...,v_s\rangle$ des combinaisons lin\'eaires 
\`a coefficients r\'eels non n\'egatifs des $v_i$, $1\leq i\leq s$. 
On note $\Delta:=\langle \e_1,\e_2,...,\e_{n}\rangle$ le c\^one standard.

Le c\^one $\sigma$ est fortement convexe s'il admet $0$ pour sommet. 
Son int\'erieur  $\sigma^0$ est, par d\'efinition, l'ensemble 
des \'el\'ements de $\sigma$ 
qui n'appartient \`a aucune de ses faces strictes.

Un c\^one $\sigma$  de $\N_{\RR}$ d\'etermine un c\^one $\sigma^{\vee}$ 
(resp. un sous-espace vectoriel $\sigma^{\perp}$) de $M_{\RR}$  en posant $\sigma^{\vee} (\mbox{resp.}\;\;   \sigma^{\perp})=\{m\in M_{\RR}\mid m_{\mid \sigma}\geq 0\;(\mbox{resp.}\;=0)\}$. La correspondance $\sigma\rightarrow \sigma^{\vee}$ \'etablit une bijection des c\^ones de 
$\N_{\RR}$ dans ceux de $M_{\RR}$ puisque $\sigma=\{n\in \N_{\RR}\mid n_{\mid \sigma^{\vee}}\geq 0\}$.
Soit  $\KK[\sigma^{\vee}\cap \M]$ la $\KK$-alg\`ebre du semigroupe $\sigma^{\vee}\cap \M$. 
En vertu du Lemme de Gordan, la $\KK$-alg\`ebre  $\KK[\sigma^{\vee}\cap \M]$ est  
de type fini.

On note $T$ le tore $\N\otimes_{\ZZ}\KK$,  
$x_j:=\chi^{\e_j} \in \KK[\Delta^{\vee}\cap \M]$, $1 \leq j\leq 4$, 
et (par abus de notation) $0$ la $T$-orbite de dimension z\'ero de $\AF_{\KK}^n:=\spec \KK[x_1,x_2,\cdots,x_n]$.
On remarque que $T=\spec \KK[\M]$.

On pose $U_{\sigma}=\spec \KK[\sigma^{\vee}\cap \M]$. 
Le plongement canonique  $\KK[\sigma^{\vee}\cap \M]\hookrightarrow \KK[\M]$ montre que $U_{\sigma}$ 
est une vari\'et\'e torique affine normale sur $\KK$.\\

Un \'eventail $\Sigma$ est un ensemble fini  de c\^ones fortement convexes dans $\N_{\RR}$  tel que:
\begin{itemize}
 \item[-] toute  face d'un c\^one de  $\Sigma$ est dans $\Sigma$;
\item[-] l'intersection de deux c\^ones de  $\Sigma$ est une face de chacun de ces deux c\^ones. 
\end{itemize}

\'Etant donn\'e un \'eventail $\Sigma$, on d\'efinit une vari\'et\'e torique $X(\Sigma)$
en recollant les vari\'et\'es affines $U_{\sigma}$, quand $\sigma$ parcourt $\Sigma$ 
le long des ouverts d\'efinis par les faces communes, c'est-\`a-dire si $\sigma_1$ 
et $\sigma_2$ appartiennent \`a $\Sigma$, on a $U_{\sigma_1}\cap U_{\sigma_2}=U_{\sigma_1\cap \sigma_2}$.

Il est connu que 
toute vari\'et\'e torique normale est obtenue comme la vari\'et\'e torique 
 associ\'ee \`a un \'eventail.\\

On dit que $v$ est un  {\it vecteur primitif}  du r\'eseau $N$ (resp.  $M$)  si $v$ engendre le $\ZZ$-module $\RR v\cap N$ (resp.  $\RR v\cap M$). Les vecteurs primitifs des faces de dimension $1$ d'un c\^one sont appel\'es les {\it vecteurs extr\'emaux} du c\^one.
Si $\sigma$ est un c\^one fortement convexe tel que $\sigma\neq\{0\}$, 
alors il est engendr\'e par ses vecteurs extr\'emaux.

On appelle c\^one r\'egulier  tout c\^one fortement convexe dont 
les vecteurs extr\'emaux sont une partie d'une base du r\'eseau. Un \'eventail $\Sigma$ est 
appel\'e  \'eventail r\'egulier si  tout c\^one $\sigma\in \Sigma$ est r\'egulier. 
La vari\'et\'e torique $X(\Sigma)$ est non singuli\`ere si et seulement si l'\'eventail $\Sigma$ est  r\'egulier.\\

Un th\'eor\`eme tr\`es int\'eressant de la th\'eorie de vari\'et\'es
toriques est l'existence d'une d\'esingularisation en caract\'eristique quelconque 
(Une d\'esingularisation ou r\'esolution des singularit\'es d'une vari\'et\'e $V$ est 
un morphisme propre et birationnel $\pi: X\rightarrow V$ tel que $\pi\mid_{X \backslash \pi^{-1}(\sing V)}: X \backslash \pi^{-1}(\sing V)\rightarrow V\backslash \sing V$ est un 
isomorphisme, o\`u $X$ est une vari\'et\'e lisse et $\sing V$ est le lieu singulier de $V$).
Plus pr\'ecis\'ement, soit $X(\Sigma)$ 
la vari\'et\'e torique associ\'ee \`a un \'eventail $\Sigma$ dans $\N_{\RR}$. Alors, il existe une d\'esingularisation \'equivariante $f:X(\Sigma')\rightarrow X(\Sigma)$, o\`u $\Sigma'$ est une subdivision r\'eguli\`ere de $\Sigma$, c'est-\`a-dire tous les c\^ones  de $\Sigma$ sont r\'eguliers.
Pour plus de d\'etails, voir \cite{KKMS73} ou \cite{CLS11}.\\

Maintenant, on fait quelques rappels sur la
r\'esolution des singularit\'es isol\'ees d'hypersurfaces.\\

Soit $\f=\sum c_{e}x^e\in \KK[x_1,x_2,...,x_n]$, $n\geq 3$, un polyn\^ome, o\`u $e=(e_1,e_2,...,e_n)\in
\ZZ^n_{\geq 0}$, $x^e:=x_1^{e_1}x_2^{e_2}\cdots x_4^{e_4}$ et $c_e\in \KK$. 
On note $\mathcal{E}(\f)$ l'ensemble des exposants $e\in \ZZ^{n}_{\geq 0}$, 
dont le coefficient $c_e$ est  non nul, c'est-\`a-dire 
$\mathcal{E}(\f):=\{e\in \ZZ^{n}_{\geq 0}\mid c_e\neq 0\}$.\\

Le poly\`edre de Newton  $\Gamma_{+}(\f)$ associ\'e \`a $\f$ est l'enveloppe convexe de 
l'ensemble $\{e+\RR^{n}_{\geq 0}\mid e\in \mathcal{E}(\f)\}$ et la fronti\`ere de Newton $\Gamma(\f)$
est la r\'eunion des faces compactes de $\Gamma_{+}(\f)$. 

On note $\mathcal{I}(\f)$ l'id\'eal de $\KK[x_1,x_2,...x_n]$ engendr\'e par les mon\^omes 
$x^e$, $e\in \Gamma(\f)\cap\ZZ^{n}$, c'est-\`a-dire $\mathcal{I}(\f):=(\{x^e\mid e\in \Gamma(\f)
\cap \ZZ^{n}\})$.

L'\'eventail de Newton $\EN{\f}$ associ\'e \`a $\f$ est la subdivision du 
c\^one standard $\Delta$ correspondant \`a l'\'eclatement normalis\'e de 
$\AF^n_{\KK}$ de centre l'id\'eal $\mathcal{I}(\f)$.
En particulier  $\mathcal{I}(\f)\mathcal{O}_{X(\EN{\f})}$ est un faisceau inversible.\\

Le polyn\^ome $\f$ est appel\'e {\it non}-{\it d\'eg\'en\'er\'e par rapport \`a la 
fronti\`ere de Newton} si pour toute face compacte $\gamma$ de 
$\Gamma_{+}(\f)$, le polyn\^ome $\f_{\gamma}:=\sum_{e\in \gamma}c_ex^{e}$ 
est non singulier sur le tore $T:=\N\oplus_{\ZZ}\KK$, 
c'est-\`a-dire les polyn\^omes $\f_{\gamma}$, $\partial_{x_1}\f_{\gamma}$,..., $\partial_{x_n}
\f_{\gamma}$ n'ont pas de z\'ero commun en dehors  de l'ensemble $x_1x_2\cdots x_n=0$.\\

On d\'efinit sa fonction d'appui associ\'ee au poly\`edre de Newton $\Gamma_{+}(\f)$), $p\in \Delta$, par

\begin{center} $ \appui{\f}(p)=\inf \{\langle r,p\rangle\mid r\in \Gamma_{+}(\f)\}.$\end{center}
o\`u $\langle\mbox{ },\mbox{ }\rangle:\RR^n\times\RR^n\rightarrow \RR$ est la forme bilin\'eaire
standard.\\

Il est connu que l'\'eventail de Newton $\EN{\f}$ satisfait la propri\'et\'e
suivante:

\begin{itemize}
 \item[($\star$)] Soient $J\subset \{1,2,...,n\}$ et  $\sigma_{J}=\{(p_1,p_2,...,p_n)\in 
 \Delta\mid p_i= 0\; \mbox{ssi}\; i\not\in J\}$. S'il existe $p\in \sigma_J$ 
 tel que $\appui{\f}(p)=0$,  alors l'adh\'erence de $\sigma_J$ dans $\Delta$ est un c\^one de 
 $\EN{\f}$. 
\end{itemize}
\vspace{0.2cm}

Une subdivision r\'eguli\`ere $\Sigma$ de $\EN{\f}$
est appel\'ee {\it subdivision r\'eguli\`ere admissible},  
si l'\'eventail  $\Sigma$ satisfait la propri\'et\'e $(\star$), 
c'est-\`a-dire  s'il existe $n\in \sigma_J$ tel que $\appui{\g}(n)=0$,
alors $\overline{\sigma}_J\in \Sigma$.\\
   Pour plus de d\'etails, voir \cite{KKMS73} et \cite{Var76}.\\

Dans la proposition suivante, on suppose que $V$ est une hypersurface normale de $\AF^{n}_{\KK}$,
donn\'ee par l'\'equation $\f=0$, o\`u $\f$ est un polyn\^ome irr\'eductible 
non-d\'eg\'en\'er\'e par rapport \`a la fronti\`ere de Newton $\Gamma(\f)$. 
De plus, on suppose  que $0\in \AF_{\KK}^n$ est l'unique point singulier de $V$.\\

Soit $\Sigma$ une subdivision r\'eguli\`ere admissible de l'\'eventail de Newton $\EN{\f}$. 
On note $\pi:X(\Sigma)\rightarrow \AF^{n}_{\KK}$ 
le morphisme torique induit par la subdivision $\Sigma$ du c\^one $\Delta$ 
et  $\widetilde{V}$ le transform\'e strict de $V$ dans $X(\Sigma)$. 
Par abus de notation, on note $\pi:\widetilde{V}\rightarrow V$
la restriction du morphisme $\pi:X(\Sigma)\rightarrow \AF_{\KK}^n$ \`a $\widetilde{V}$.\\ 

 La proposition suivante   r\'esulte  des  Lemmes $10.2$ et $10.3$ de \cite{Var76} 
 (pour plus de d\'etails, voir \cite{Mer80}).
\begin{proposition}

\label{pr:cr_nor-f}  Le morphisme $\pi:\widetilde{V}\rightarrow V$ est une d\'esingularisation de $V$. 

 Si l'hypersurface $V$ ne contient aucune de $T$-orbite de $\AF_{\KK}^n$ de 
 dimension strictement plus grande que z\'ero, 
 alors le morphisme $\pi:X(\Sigma)\rightarrow \AF^n_{\KK}$ est une r\'esolution plong\'ee de $V$, 
 c'est-\`a-dire $\pi:X(\Sigma)\rightarrow \AF_{\KK}^n$ est un morphisme propre et birationnel, 
 $\pi':X(\Sigma)\backslash(\pi')^{-1}(0)\rightarrow \AF_{\KK}^n\backslash \{0\}$
 est un isomorphisme et $(\pi')^{-1}(V)$ est un diviseur \`a croisements normaux.
\end{proposition}

Soient $\Sigma'$ un \'eventail dans $\N_{\RR}$ et $\sigma$ un c\^one de $\Sigma'$.
 Nous appellerons $G$-{\it subdivision r\'eguli\`ere} de $\sigma$, une subdivision r\'eguli\`ere de $\sigma$, dont les vecteurs extr\'emaux
sont exactement les \'el\'ements irr\'eductibles du semi-groupe $\sigma\cap \N$.

 On dit qu'un \'eventail $\Sigma'_{\mathcal{G}}$ est une 
 $G$-{\it subdivision r\'eguli\`ere} de $\Sigma'$ si chaque 
 c\^one de $\Sigma'_{\mathcal{G}}$ est obtenu par une  
 $G$-subdivision r\'eguli\`ere d'un c\^one de $\Sigma'$.
 Le morphisme \'equivariant  associ\'e \`a une $G$-subdivision
 r\'eguli\`ere de $\Sigma'$  est appel\'e $G$-{\it d\'esingularisation} de $X(\Sigma')$.

 Il existe des exemples o\`u  un \'eventail $\Sigma$ n'admet pas une 
 $G$-{\it subdivision r\'eguli\`ere} (voir \cite{BoGo92}).
 On remarque que si $\EN{\f}$ admet une $G$-{\it subdivision r\'eguli\`ere} $\EN{\f}_{G}$, alors 
 $\EN{\f}_G$ est une subdivision r\'eguli\`ere admissibe.

\section{Familles d'espaces de $m$-Jets et d'espaces d'arcs}
\label{sec:Fam-mJ-arc}
Il est connu que les espaces de $m$-jets et l'espace d'arcs fournissent des 
informations sur la g\'eom\'etrie  de la vari\'et\'e sous-jacente. Il est donc naturel de se poser la question suivante:
Est-ce qu'une d\'eformation de la vari\'et\'e $V$ induit une d\'eformation 
(compatible) des espaces de $m$-jets $V_m$ et de l'espace d'arcs $V_{\infty}$?
L'examen de cette question est le sujet de la sous-section \ref{ssec:Fam-mJ-arc},
o\`u on pr\'esente quelques r\'esultats et quelques probl\`emes ouverts.

Dans la sous-section \ref{ssec:Fam-apl_nash}, on consid\`ere certaines d\'eformations d'une vari\'et\'e ayant une singularit\'e isol\'ee
qui admet certaines r\'esolutions simultan\'ees \`a plat (Autrement dit, de fa\c con intuitive, la fibre exceptionnelle
se d\'eforme aussi) et on se pose la question suivante:
Est-ce que la condition qu'un diviseur appartient \`a l'image de l'application de Nash est pr\'eserv\'ee
par cette type de d\'eformations?
On donne quelques r\'eponses partielles \`a cette question.

\subsection{Familles d'espaces de $m$-Jets et d'espaces d'arcs}
\label{ssec:Fam-mJ-arc}
Soient $V$ une vari\'et\'e alg\'ebrique sur $\KK$, $S$ un sch\'ema, $0\in S$ un point ferm\'e et $W$ 
une d\'eformation de $V$ sur $S$, c'est-\`a-dire il existe un diagramme  commutatif:
 
 \begin{center}
\begin{tabular}{lr}

 $\shorthandoff{;:!?} 
\xymatrix @!0 @R=1.5pc @C=2pc{  
V\ar@{^{(}->}[rr] \ar@{->}[dd] & 
& W \ar@{->}[dd]\\
   &\square& \\
  0 \ar@{^{(}->}[rr] & &  S
}$   
\end{tabular}  
\end{center} o\`u $W\rightarrow S$ est plat et $V\cong W\times_S\{0\}$.
Par abus de notation, on note $\{0\}$ le sous-sch\'ema de $S$ associ\'e au point $0$.\\  

Comme on a dit ci-dessus, on se pose la question suivante:
Est-ce qu'une  d\'eformation $W$ de la vari\'et\'e $V$ induit une d\'eformation 
des espaces de m-jets $V_m$ et de l'espace d'arcs $V_{\infty}$?\\

En vertu des Th\'eor\`emes \ref{th:prop-fonct-mjet} et \ref{th:prop-fonct-arc}, et la Proposition \ref{pr:mjets-chang-base},
l'espace $\mJet{W}{S}{m}$ est  le candidat naturel pour \^etre l'espace d\'eformation de $V_{m}$, 
$1\leq m\leq \infty$. La question est donc de savoir si  le morphisme $\mJet{W}{S}{m}\rightarrow S$ est plat.\\

En g\'en\'eral, le fait que le morphisme $W\rightarrow S$ soit plat n'implique pas que le morphisme $\mJet{W}{S}{m}\rightarrow S$
le soit.\\

\begin{exemple}
 Soient $V\subset \AF_{\KK}^2$ une courbe donn\'ee par le polyn\^ome $x^3+y^4=0$, $W\rightarrow \AF_{\KK}:(x,y,s)\mapsto s$
 la d\'eformation donn\'ee par le polyn\^ome $x^3+y^4+s=0$ et $W_s$ la fibre sur $s\in \AF_{\KK}$.
 Pour tout $s\neq 0$ la vari\'et\'e $W_s$ est lisse, on a donc que l'espace de $2$-jet $(W_s)_2$ est une vari\'et\'e
 de dimension $3$. Par contre, l'espace de $2$-jet $V_2$ a une composante de dimension $4$. Par cons\'equence
 le morphisme $\mJet{W}{\AF_{\KK}}{m}\rightarrow \AF_{\KK}$ n'est pas plat.
\end{exemple}

La question est donc: Quand est-ce qu'une d\'eformation de $V$ induit une d\'eformation 
de l'espace $V_{m}$, $1\leq m\leq\infty$?\\

Dans la suite,  on consid\`ere le cas de vari\'et\'es localement intersection compl\`ete.\\

Soient $n$, $p$ deux entiers positifs tels que $0<p<n$, et  $V\subset (\AF_{\KK}^{n},0)$ un germe d'un sch\'ema intersection compl\`ete,
d\'efinie par l'id\'eal $(\f_1(x),...,\f_p(x))\subset \mathcal{O}_{\AF^n,0}$, o\`u $x:=(x_1,...,x_n)$ et
$\mathcal{O}_{\AF^{n},0}$ est la 
fibre en $0$ du faisceau $\mathcal{O}_{\AF^{n}}$\\

On fixe un entier $l\geq 1$ et soient $\F_1(x,s), ...,\F_p(x,s)\in \mathcal{O}_{\AF^{n+l},0}$, o\`u $s:=(s_1,...,s_l)$,  
tels que $\F_i(x,0)=\f_i(x)$, $1\leq i\leq p$, et $W$ le germe du sch\'ema d\'efini par l'id\'eal  
$I:=(F_1(x,s),...,F_p(x,s))\subset \mathcal{O}_{\AF^{n+l},0}$.\\

Dans les Sections $3$ et $4$ de \cite{Art76}, l'auteur montre que 
$\mathrm{Tor}_1^{\mathcal{O}_{\AF^{l},0}}(\KK_{\AF^{l}_{0}}, \mathcal{O}_{\AF^{n+l},0}/I)=0$,  o\`u $\KK_{\AF^{l}_0}$ est 
le corps r\'esiduel de $\mathcal{O}_{\AF^{l},0}$. Ce qui implique, en utilisant le crit\`ere local 
de platitude (voir le Th\'eor\`eme $6.8$ de \cite{Eis95}), que le  morphisme  $\pi: W\rightarrow S; (x,s)\mapsto s$, 
o\`u $S:=\AF_{\KK, 0}^l$, est 
plat. En particulier $W$ est une d\'eformation de $V$.\\

Si  $V$ est un germe d'une vari\'et\'e intersection compl\`ete ayant une   singularit\'e
isol\'ee \`a l'origine $0$,
alors  il existe un entier $l$ et $\F_1,...,\F_p\in \mathcal{O}_{\AF^{n+l},0}$ tel que la d\'eformation $\pi:W\rightarrow S$
est une d\'eformation Verselle 
(voir \cite{KaSc72}),
c'est-\`a-dire s'il existe une autre d\'eformation $\pi':W'\rightarrow S'$ de $V$, alors il existe un morphisme
$S'\rightarrow S$ tel que la d\'eformation $W'$ est isomorphe \`a la d\'eformation $W\times_S S'$.\\

Dans les r\'esultats suivants, on suppose que $V$ est un germe d'une vari\'et\'e intersection 
compl\`ete ayant une   singularit\'e isol\'ee \`a l'origine $0$ et que $W$ est une d\'eformation Versalle de $V$.

Le r\'esultat suivant est une cons\'equence directe du crit\`ere local de platitude et de la Proposition \ref{pr:mjets-chang-base}.
\begin{proposition}
\label{pr:def-mjets-lic}
 On suppose que $V_{m}$ est localement intersection compl\`ete, 
 alors le morphisme $\mJet{W}{S}{m} \rightarrow S$ est une d\'eformation
 de  $V_m$ pour tout $1 \leq m\leq \infty$.
\end{proposition}

L'espace de $m$-jets $V_m$ est d\'efini par $p(m+1)$ \'equations, mais en g\'en\'eral cet espace n'est pas  
localement une intersection compl\`ete. 
En fait, si l'espace de $m$-jets $V_m$ est localement une intersection compl\`ete (resp. irr\'eductible) si et seulement
si $V$  est une singularit\'e log-canonique (resp. rationnelle). Voir les articles \cite{Mus01} et \cite{EiMu04}.\\

Le r\'esultat suivant suit imm\'ediatement  de la Proposition \ref{pr:def-mjets-lic} 
et du Th\'eor\`eme  $0.1$ de l'article ant\'erieurement cit\'e.

\begin{theoreme}

On conserve les notations ci-dessus et on suppose que la singularit\'e de $V$ est log-canonique. Alors l'espace  $\mJet{W}{S}{m}$ est une
une d\'eformation de l'espace de $m$-jet $V_m $, $1\leq m\leq \infty$.
\end{theoreme}

Dans le r\'esultat suivant, on ne suppose pas que la singularit\'e de $V$ est de log-canonique.

\begin{proposition}
 Si $pm\leq n-p$, alors  l'espace de $m$-jets $V_m$ est localement une intersection compl\`ete. En particulier le 
 morphisme $\mJet{W}{S}{m}\rightarrow S$ est une d\'eformation de $V_m$.
\end{proposition}

\begin{proof}
Soit $\pi_m:V_m\rightarrow V$ la projection canonique. On a donc la d\'ecomposition suivante: 

\begin{center}
 $V_m=\pi_m^{-1}(\{0\})\cup \overline{\pi_m^{-1}(V\backslash \{0\})}$
\end{center}

On remarque que $\overline{\pi_m^{-1}(V\backslash \{0\})}$ est une composante irr\'eductible de $V_m$
de dimension $(n-p)(m+1)$ et $\dim \pi_m^{-1}(0)\leq nm$.
En vertu de la proposition $1.4$ de \cite{Mus01}, $V_m$ est une intersection compl\`ete si et seulement 
si $\dim V_m\leq (n-p)(m+1)$. Ce qui ach\`eve la preuve de la proposition, car par hypoth\`ese on a  
$nm \leq (n-p)(m+1)$.
\end{proof}

Maintenant, on  suppose que  $S=\spec \KK[[s]]$ (ou $S:=\spec \KK[s]_{(s)}$), et soit  $V$ une vari\'et\'e quelconque
et $W\rightarrow S$ une d\'eformati\'on 
de $V$. Le morphisme $\mjp{m}:\mJet{W}{S}{m}\rightarrow S$, $m\geq 0$, est plat si et seulement si pour tout
{\it point associ\'e} $p\in \mJet{W}{S}{m}$ (c'est-\`a-dire,  tous les \'el\'ements de l'id\'eal maximal du point $p$ sont des
diviseurs du z\'ero), le point $\mjp{m}(p)$ est le point g\'en\'erique de $S$ (voir la page $257$ de \cite{Har77}).
Une partie de ce r\'esultat peut s'\'ecrire de la fa\c{c}on suivante:

\begin{lemme} \label{le:plat-arcs} On suppose que pour tout point $p\in  V_m\subset  \mJet{W}{S}{m}$, il existe un arc 
$\alpha:\spec K[[s]] \rightarrow \mJet{W}{S}{m}$, o\`u $K$ est  une extension du corps $\KK$, tel que l'image par 
$\alpha$ (resp. par $\mjp{m}\circ \alpha$) du point ferm\'e $0$  (resp. point g\'en\'erique $\eta$ ) de
$ \spec K[[s]]$ est le point $p$ (resp.  le point g\'en\'erique de $S$). Alors le morphisme  
$\mjp{m}:\mJet{W}{S}{m}\rightarrow S$ est plat.
\end{lemme}
\begin{proof}
On suppose qu'il existe un point associ\'e $p\in  \mJet{W}{S}{m}$ tel que $\mjp{m}(p)$ soit le point ferm\'e de $S$, alors
$p\in V_{m}$. Comme $p$ est un point  associ\'e, on a  $\alpha(0)=\alpha(\eta)=p$. Ce qui est  une contradiction, 
car $\mjp{m}(\alpha(\eta))$ est le point g\'en\'erique de $S$.
\end{proof}

Le lemme  \ref{le:plat-arcs} et la propri\'et\'e fonctorielle de l'espace $\mJet{W}{S}{m}$
motive la d\'efinition suivante:

\begin{definition}
Soit $m$ (resp. $n$)   un entier  sup\'erieur ou \'egal \`a $0$ (resp. $3$).
Un morphisme   $\phi:\spec K[t]/(t^{m+1})\rightarrow \AF_{\KK}^n$ est appel\'e $(K,m)$-jet sur  $\AF_{\KK}^n$.

On dira que un $S$-morphisme  $\Phi:\spec K[[s]][t]/(t^{m+1})\rightarrow \AF_{S}^n$ est une d\'eformation d'un 
 $(K,m)$-jet $\phi$, si la sp\'ecialisation en $s=0$ de
  $\Phi$   est $\phi$.
 \end{definition}
\vspace{0.2cm}

D'abord, on fixe quelques  notations. \'Etant donn\'e  un polyn\^ome  $\varphi\in K[t]$,  on note $\ord_t\varphi$ l'ordre en $t$ du polyn\^ome  $\varphi$ et on d\'efinit le $m$-ordre de $\varphi$, not\'e $\ord_t^m\varphi$, comme suit:

\begin{center}
 $\ord_t^m\varphi:=\left \{ \begin{array}{cl}
                       \infty &\mbox{si}\; \varphi\equiv 0\mod (t^{m+1})\\
                        \ord_t\varphi    & \mbox{sinon}
                      \end{array}\right.$

\end{center}
On remarque qu'on peut d\'efinir $\ord_t^m$ sur l'anneau  $K[t]/(t^{m+1})$.

\'Etant donn\'e le $(K,m)-jet$, $\phi:\spec K[t]/(t^{m+1})\rightarrow \AF_{\KK}^n$, on d\'efinit  le $m$-ordre, $\ord_t^m \phi$, de $\phi$ de la fa\c con suivante:
\begin{center}
$\ord_t^m \phi :=(\ord_{t}^m\; \phi^{\star }(x_1),...,\ord_{t}^m\; \phi^{\star }(x_{n}))$,
\end{center}
 o\`u $\phi^{\star}$ est le co-morphisme  de $\phi$.
 
Avec la notation standard des multi-indices (c'est-\`a-dire  on note $x^e=x_1^{e_1}\cdots x_n^{e_n}$ pour  $e=(e_1,...,e_n)\in \ZZ_{\geq 0}^n$), on consid\`ere la s\'erie formelle suivante:\\
\begin{equation*} \tau:=\sum\limits_{\mbox{\small{{\it e $\in \ZZ_{\geq 0}^n$}}}}c_ex^e,
\end{equation*}
o\`u $c_e \in K$  pour tout $e\in \ZZ_{\geq 0}^l$.
L'ensemble $\mathcal{E} (\tau):=\{e\in \ZZ_{\geq 0}^n\mid c_e\neq 0\}$ est l'ensemble des exposants de $\tau$. Si la s\'erie formelle $\tau$ n'est pas nulle, alors l'ensemble $\mathcal{E} (\varphi)$ n'est pas vide.\\
Soient $v\in \RR_{> 0}^n$,  on d\'efinit  le $v\mbox{-}ordre$ de $\tau$, not\'e $\nu_v\tau$, et le  $(v,m)\mbox{-}ordre$ de $\tau$, not\'e $\nu_v^m\tau$,  de la fa\c con suivante:\\
\begin{center}
$ \nu_v\tau:=\min\{v\cdot e\mid e  \in\mathcal{E}(\tau)\}$;
$\nu_v^m\tau:=\left \{ \begin{array}{cl}
                       \infty &\mbox{si}\; \nu_v\tau>m\\
                        \nu_v\tau    & \mbox{sinon}
                      \end{array}\right.$

\end{center}
\vspace{1cm}

On consid\`ere une suite de  polyn\^omes  $\f, \g_1, \g_2,...$ $\in \KK[x_1,...,x_n]$, $n\geq 3$, et on pose 
\vspace{0.3cm}
\begin{center}
$\F(x,s):=\f(x)+s\g_1(x)+s^2\g_2(x)+\cdots$,
\end{center}
o\`u $x:=(x_1,...,x_n)$ et $\f$ un polyn\^ome irr\'eductible non nul.  
Dans toute la suite, $V$ est la vari\'et\'e d\'efinie par le polyn\^ome $\f$ et $W$ la d\'eformation de $V$
d\'efinie par $\F\in\KK[x_1,x_2,...,x_m][[s]]$ (voir page 79 de \cite{EiHa00}).\\

\begin{lemme}
\label{le:def-mjet}  Soit $\phi:\spec K[t]/(t^{m+1})\rightarrow V$ un $(K,m)$-jet et  $v=\ord_t\phi$. On suppose que
\begin{center}
  $\nu_v^m\f=\min \{\ord_t^m \phi^{\star}(x_i\partial_i\f(x))\mid 1\leq i \leq n \}$ et  que
  $\nu_v^m\f\leq \nu_v^m\g_j $, pour tout $1\leq j$, 
\end{center}  
   Alors, il existe une d\'eformation $\Phi$  du  $(K,m)$-jet $\phi$ tel que $ \Phi^{\star}(\F(x,s)) \equiv 0\!\!\mod\! (t^{m+1}) $.

 \end{lemme}
 \begin{remarque}[Polyn\^omes de Pham-Brieskorn]
  \label{re:pham-brieskorn}
 On remarque que dans le cas o\`u  \begin{center} $\f:=x_1^{a_1}+x_2^{a_2}+\cdots+x_2^{a_k}+\cdots+ x_n^{a_n}$, $a_k>1$, \end{center}
et les polyn\^omes  $\g_i$, $1\leq i$,  
appartiennent \`a la cl\^oture int\'egrale de l'id\'eal engendr\'e par 
$x_1^{a_1},x_2^{a_2},\cdots, x_n^{a_n}$, les hypoth\`eses du Lemme \ref{le:def-mjet} sont toujours v\'erifi\'ees.
 \end{remarque}

\begin{proof}
Soit  $\phi_i:=\phi^{\star}(x_i)$, $1\leq i\leq n$,  et on  pose $\Phi_i:= \phi_i+ s\psi_{1 i}+s^2\psi_{2 i}+
s^3\psi_{3 i}+\cdots$, $1\leq i\leq n$, o\`u les $\psi_{j i}$ appartient \`a 
$K[t]/(t^{m+1})$. On peut donc d\'efinir une {\rm d\'eformation} $\Phi$ du  $(K,m)$-jet $\phi$ tel que 
$\Phi^{\star}(x_i):=\Phi_i$, $1\leq i\leq n$.\\

L'id\'ee de la d\'emonstration est de   montrer qu'on peut choisir les  $\psi_{j i}\in K[t]/(t^{m+1})$  
de fa\c con que $\F(\Phi_1,...,\Phi_n,s)\equiv 0 \mod (t^{m+1})$.\\

 On pose  $\overline{\Phi}:=(\Phi_1,...,\Phi_n)$,  $\overline{\phi}:=(\phi_1,...,\phi_n)$, $\overline{\Psi}:=\overline{\Phi}-\overline{\phi}$
 et $\overline{\psi}_j:=(\psi_{j 1},...,\psi_{j n})$.
On remarque qu'on a toujours la formule suivante:

\begin{center}
 $ \F(\overline{\Phi},s)=\f(\overline{\phi})+s(\sum \partial_i\f(\overline{\phi})\psi_{1 i}+
 T_1))+ s^2(\sum \partial_i\f(\overline{\phi})\psi_{2 i}
 + T_{2})+\cdots +s^j(\sum \partial_i\f(\overline{\phi})\psi_{j i}
 + T_{j})+\cdots$
\end{center}
o\`u  $T_j$, $j\geq 1$,  est d\'efini de la mani\`ere suivante:   
 $$  T_{j}:= \g_j(\overline{\phi})+  \sum_{l=1}^{j-1}\sum_{1\leq \mid \alpha \mid \leq j-l}\tfrac{1}{\alpha! }
 \partial^{\alpha}\g_{l}(\overline{\phi})I_{j-l}(\alpha)+ 
\sum_{2\leq \mid \alpha \mid \leq j}\tfrac{1}{\alpha! }\partial^{\alpha}\f(\overline{\phi})I_{j}(\alpha),$$  
o\`u   $\partial_{\alpha}:=\partial_1^{\alpha_1}\cdots \partial_n^{\alpha_n}$, 
$\mid \alpha\mid :=\sum \alpha_i$, $\alpha !=\alpha_1!\cdots \alpha_n!$,
et $I_k(\alpha):=\frac{1}{k!}\partial^{k}_{s}\overline{\Psi}^{\alpha}\mid_{ s=0}$, pour  $1\leq k \leq j$ et $(\alpha_1,...,\alpha_n)\in \ZZ_{\geq 0}^n$.\\

On remarque  que si $1\leq \mid\alpha\mid$ 
(resp.  $2\leq \mid\alpha\mid$) et $k'>k$ (resp. $k'\geq k$), alors $I_k(\alpha)$ ne d\'epend pas de $\overline{\psi}_{k'}$.\\
 
On va proc\'eder par r\'ecurrence sur $j$ pour  montrer la propri\'et\'e suivante:\\

($\star$) Il existe  $\overline{\psi}_j$,  
tel que 
 $\ord_t^m(\sum \partial_i\f(\overline{\phi})\psi_{j i}
 + T_{j})=\infty$ et  $\ord_t^m\phi_i\leq \ord_t^m \psi_{j i}$, pour tout  $1\leq i\leq n$.\\

Mais d'abord, on fait une remarque que l'on utilisera fr\'equemment dans cette
d\'emonstration:\\
\begin{remarque}
 \label{re:lemma-def-mjet}
Soient $p,q,r\in K[t]/(t^{m+1})$ tel que $\ord_t^m pq\leq \ord_t^mr$, on remarque qu'il existe  $w\in K[t]/(t^{m+1})$ tel que 
$\ord_t^m (pw-r)=\infty$ et $\ord_t^m q\leq \ord_t^m w$.\\
\end{remarque}

Par hypoth\`ese, on a   
$\nu_v^m\f=\min \{\ord_t^m \phi^{\star}(x_i\partial_i\f(x))\mid 1\leq i \leq n \}$ et  $\nu_v^m\f\leq \nu_v^m\g_1 $, 
$v=\ord_t\phi$.  Comme  $\nu_v^m\g_1 \leq \ord_t^m \g_1(\overline{\phi})$, on obtient qu'il existe un entier
$1\leq i_0\leq n$ tel que
$\ord_t^m\partial_{i_0}\f(\overline{\phi})\phi_{i_0}\leq \ord_t^m\g_1(\overline{\phi})$. 
En utilisant la Remarque \ref{re:lemma-def-mjet}, on obtient qu'il existe  $\overline{\psi}_{1}$ tel que  $\ord_t^m(\sum \partial_i
\g_0(\overline{\phi})\psi_{1 i}
 + T_{1})=\infty$, o\`u $\ord_t^m\phi_i\leq \ord_t^m \psi_{1 i}$, pour tour $1\leq i\leq n$.\\

 On peut donc supposer  qu'il existe un entier  $j_0$ tel que la propri\'et\'e ($\star$) est v\'erifi\'ee  pour tout $j<j_0$.
 On va montrer que la propri\'et\'e ($\star$)  est v\'erifi\'ee pour $j_0$.\\

Comme  $\ord_t^m \phi_i\leq \ord_t^m \psi_{j i}$, pour tout $1\leq i\leq n$ et $1\leq j< j_0$, on a 
$\nu_v^m\f\leq \ord_t^m\partial^{\alpha}\f(\overline{\phi})I_{j_0}(\alpha)$  (resp.
$\nu_v^m\g_l\leq \ord_t^m\partial^{\alpha}\g_l(\overline{\phi})I_{j_0-l}(\alpha)$, $1\leq l<j_0$,) 
pour tout  $\alpha\in \ZZ^n_{\geq 0}$ tel que $2\leq \mid \alpha \mid \leq j_0$  (resp.  $1\leq \mid \alpha \mid \leq j_0-l$). Par cons\'equent, on a: $\nu_v^m \f \leq \ord_t^mT_{j_0}$.

Comme  $\nu_v^m \f \leq \ord_t^m T_{j_0}$  et 
$\nu_v^m\f=\min \{\ord_t^m \phi^{\star}(x_i\partial_i\f(x))\mid 1\leq i \leq n \}$, on obtient qu'il existe 
$\overline{\psi}_{j_0}$ tel que  $\ord_t^m(\sum \partial_i\f(\overline{\phi})\psi_{j_0 i}
 + T_{j_0})=\infty $ et  $\ord_t^m \phi_{i}\leq \ord_t^m\psi_{j_0 i}$, pour tout $1\leq i\leq n$ 
 (voir la remarque \ref{re:lemma-def-mjet}).
Ceci ach\`eve la d\'emonstration du lemme  
\end{proof}

Le th\'eor\`eme suivant  est une cons\'equence imm\'ediate des lemmes \ref{le:plat-arcs} et \ref{le:def-mjet} 
\begin{theoreme}
 
 Soit 
 $\f:=x_1^{a_1}+x_2^{a_2}+\cdots+x_2^{a_k}+\cdots+ x_n^{a_n}$, $a_k>1$, 
et on suppose que les polyn\^omes  $\g_i$, $i\geq 1$,  
appartiennent \`a la cl\^oture int\'egrale de l'id\'eal engendr\'e par 
$x_1^{a_1},x_2^{a_2},\cdots, x_n^{a_n}$. 
Alors, le morphisme $\mjp{m}  :  \mJet{W}{S}{m}\rightarrow S$ est plat pour tout $0\leq m\leq \infty$.
Autrement dit,  $\mJet{W}{S}{m}$ est une d\'eformation de $V_m$.

\end{theoreme}

\begin{remarque}
 On remarque que la d\'eformation $W$ de la vari\'et\'e $V$ du th\'eor\`eme ci-dessus est une d\'eformation \`a
 nombre de Milnor constant. Est-ce que la question principale de cette section a une r\'eponse 
 affirmative dans le cas de d\'eformation \`a  nombre de Milnor constat?

 Plus pr\'ecis\'ement, soient $n\geq 3$, $\f(x_1,...,x_n)$ un polyn\^ome qui n'est pas d\'eg\'en\'er\'e par rapport \`a la 
 fronti\`ere de Newton $\Gamma(\f)$ et $\F(x_1,x_2,...,x_n,s)\in \KK[x_1,..,x_n,s]$  une d\'eformation de $\f$ \`a 
 nombre de Milnor constant. Est-ce que $\mJet{W}{\AF_{\KK}}{m}$ est une d\'eformation de $V_m$?\\

\end{remarque}

\subsection{Familles d'espaces d'arcs et l'application de Nash}

\label{ssec:Fam-apl_nash}
 Dans cette section,  $V$ est une vari\'et\'e normale ayant
une unique singularit\'e isol\'ee, not\'ee $0$, et $W$ est une d\'eformation de $V$ qui admet
une {\it r\'esolution simultan\'ee \`a plat} (voir \cite{Tei80}).
La  d\'efinition suivante est plus restrictive que celle de l'article ant\'erieurement cit\'e, mais
 elle est la plus adapt\'ee \`a nos probl\`emes.\\
 
Soit $S$ un sch\'ema et  $s\in S$ un point. Par abus de notation, on note $\{s\}$ le sous-sch\'ema de $S$ associ\'e au point $s$.   
Soit  $0\in S$ un point ferm\'e de $S$.  Une {\it d\'eformation $W$ de $V$ sur $S$ le long une section} $\sigma:S\rightarrow W$ est
un  diagramme commutatif:
\begin{center}
\begin{tikzpicture}[node distance=2cm, auto]
  \node (W) {$W$};
  \node (D)[right of=W]{$S$};
   \draw[->, bend right] (D) to node[swap] {$\sigma$} (W);
    \draw[->] (W) to node {} (D);
 
\end{tikzpicture}
\end{center}
o\`u le morphisme $W\rightarrow V$ est plat \`a fibres r\'eduits et  
$V\cong W_0:=W\times_{S}\{0\}$.
\begin{definition}
\label{def-equiv} Soit 
$\pi_0 : \widetilde{V} \rightarrow V$  une d\'esingularization telle que 
la fibre exceptionnelle $\pi_0^{-1}(0)=\E_1\cup \E_2\cdots\cup \E_l$ est un diviseur 
\`a croisements normaux et  $W$  une d\'eformation de $V$ sur $S$ le long une section 
$\sigma:S\rightarrow W$. On suppose que $\sigma(s)$, $s\in S$, est l'unique 
point singulier de la vari\'et\'e $W_s:=W\times_S\{s\}$ . 
La d\'eformation $W$  admet une r\'esolution simultan\'ee \`a plat, 
s'il existe un diagramme commutatif:

\begin{center}
\begin{tabular}{lr}

 $\shorthandoff{;:!?} 
\xymatrix @!0 @R=1.5pc @C=2pc{  &&& 
\E_1,\E_2,\cdots, \E_l \ar@{^{(}->}[rrrr] \ar@{^{(}->}[dd]&& & 
&\widetilde{\E}_1,\widetilde{\E}_2,\cdots, \widetilde{\E}_l \ar@{^{(}->}[dd]\\
 (\star) &&& &&& \\
 &&& \widetilde{V} \ar[dd] \ar[rd] \ar@{^{(}->}[rrrr] &&& &  \widetilde{W} 
 \ar[dd] \ar[rd]^{\pi} &
\\
 &&&  &V\ar[ld] \ar@{^{(}->}[rrrr] &&&& W \ar[ld]\\
&&& 0 \ar@{^{(}->}[rrrr]  & &&& S}$  &  
$\shorthandoff{;:!?}    \xymatrix @!0 @R=2pc @C=2pc{
(\widetilde{\E}_1)_s,...,(\widetilde{\E}_l)_s \ar@{^{(}->}[rrr] &&&
\widetilde{W}_s \ar[dd]^{\pi_s}&\hspace{1.2cm}:=\widetilde{W}\times_S\{s\}\\
&&&&\\
 &&& W_s \ar[d]&\hspace{1.2cm}:= W\times_S\{s\} \\
   && & s & }$   
\end{tabular}
\end{center}
o\`u  le morphisme  $\pi:\widetilde{W}\rightarrow W$ est
une modification propre de $W$  telle que  $\pi_s:\widetilde{W}_s\rightarrow W_s$
est une d\'esingularization, o\`u que la fibre exceptionnelle, $\pi^{-1}(\sigma(s))=(\widetilde{\E}_1)_s\cup 
(\widetilde{\E}_2)_s\cdots\cup (\widetilde{\E}_l)_s$, est un diviseur \`a 
croisements normaux,  et le sch\'ema $\widetilde{W}$ 
(resp. la fibre sch\'ematique  $\widetilde{W}\times_W\sigma(S)$)
est une d\'eformation de $\widetilde{V}$ (resp. de la fibre sch\'ematique 
$\widetilde{V}\times_V \{0\}$) sur $S$.
\end{definition}

\begin{exemple} Soit 
$\f \in \KK[x_1,...,x_n]$, $n\geq 4$, un polyn\^ome irr\'eductible et non d\'eg\'en\'er\'e 
par rapport \`a la fronti\`ere de Newton, et soient  $\g_1$,..,$\g_m$  
des polyn\^omes de  $\KK[x_1,...,x_n]$ tels que tous les exposants
de chaque polyn\^ome $\g_i$, $1\leq i\leq m$, 
appartient  au poly\`edre de Newton $\Gamma_{+}(\f)$.
On pose $\F(x,s):=\f(x)+\sum s_i\g_i(x)$, o\`u $s:=(s_1,...,s_m)\in S:= \AF_{\KK,0}^m$ et $x\in\AF_{\KK}^n$.

Soit $V$ (resp. $W$) la vari\'et\'e d\'efinie par $\f$ (resp. $\F$). On remarque que le morphisme naturel 
$W\rightarrow S$ est une d\'eformation de $V$. Si $V$ ne contient aucune des $T$ orbites de
$\AF_{\KK}^n$, alors  la Proposition \ref{pr:cr_nor-f} montre que $W$ admet une r\'esolution simultan\'ee \`a plat.
En fait, \`a un changement de base pr\`es,  toute d\'eformation d'une  hypersurface quasi-homog\`ene  $V$ qui admet une r\'esolution simultan\'ee 
\`a plat \'equisingulier est obtenue de cette fa\c con. Voir
\cite{Kou76}, \cite{Osh87},  \cite{SaTo04} et \cite{Var82}.
\end{exemple}
  
 Pour chaque diviseur $\widetilde{\E}_i$, on note $\N(\widetilde{\E}_i)$ 
 l'adh\'erence dans $\arcSpace{W}{S}$ de l'ensemble suivant:
 
  $$\{\alpha\in \arcSpace{W}{S}\backslash \arcSpace{(\sing W) }{S}\mid 
  \widehat{\alpha}(0)\in \widetilde{\E}_i \}$$
  o\`u $\widehat{\alpha}$ est le rel\`evement \`a $\widetilde{W}$  de l'arc $\alpha$. 
  \'Etant donn\'e $s\in S$, on pose $\N(\widetilde{\E}_i)_s:= \N(\widetilde{\E}_i)\times_S\{s\}$.
  On remarque que $\N(\E_i)\subset \N(\widetilde{\E}_i)_0\subset V_{\infty}^s $ 
\begin{lemme}  \label{le:NE-irr} On conserve la notation employ\'ee auparavant  et, en plus,
on suppose que $S$ est une vari\'et\'e sur $\KK$. Alors 
 $\N(\widetilde{\E}_k)$ est irr\'eductible  pour tout $k \geq 1$.
\end{lemme}
\begin{remarque}
\label{re:Ns-dense}
 D'apr\`es le lemme ci-dessus, si $s$ est le point g\'en\'erique de $S$, alors l'ensemble $\N(\widetilde{\E}_k)_s$ est dense dans 
 $\N(\widetilde{\E}_k)$ pour tout $k\geq 1$.
\end{remarque}

  \begin{proof} 
   Soit $\pi_{\infty}:\arcSpace{\widetilde{W}}{S}\rightarrow \arcSpace{W}{S}$ le morphisme induit 
 par $\pi$. On pose  $\M(\widetilde{\E}_k):= \{\beta\in \arcSpace{\widetilde{W}}{S} \mid 
  \beta(0)\in \widetilde{\E}_k \}$ et  on remarque que 
  $\N(\widetilde{\E}_k)=\overline{\pi_{\infty}(\M(\widetilde{\E}_k))}$. Il suffit donc de montrer que $\M(\widetilde{\E}_k)$
  est irr\'eductible.
  
  \'Etant donn\'e  un ouvert $U\subset \widetilde{W}$, on peut supposer que  $\mJet{U}{S}{m}$ est un ouvert de  
  $\mJet{\widetilde{W}}{S}{m}$,  $0\leq m\leq \infty$ (voir le Th\'eor\`eme 4.3 de  \cite{Voj07}). 
  
  Comme $\widetilde{W}$ et $S$ sont de type fini sur $\KK$, le morphisme $\widetilde{W}\rightarrow S$ 
  est un morphisme lisse de dimension relative $d=\dim \widetilde{W}-\dim{S}$ (voir le Th\'eor\`eme $10.2$ de \cite{Har77}).
 Alors, en vertu de la Proposition $5.10$ de  \cite{Voj07}, pour chaque $x\in \widetilde{\E}_k$  il existe un ouvert 
 $x\in U_x\subset \widetilde{W}$ et un  $U_x$-isomorphisme,  $\mJet{U_{\it x}}{S}{m}\cong \AF^{dm}_{U_x}$, 
 $0\leq m \leq \infty$. En particulier $\M(\widetilde{\E}_k)\cap \mJet{U_{\it x}}{S}{\infty}$ est un cylindre sur 
 $\widetilde{\E}_k\cap U$.
 Comme $\widetilde{\E}_k\cap U$  est irr\'eductible, on obtient que 
 $\M(\widetilde{\E}_k)\cap \mJet{U_{\it x}}{S}{\infty}$  
 est irr\'eductible pour chaque $x\in \widetilde{\E}_k$. 
 Ce qui implique que   $\M(\widetilde{\E}_k)$ est irr\'eductible.
  \end{proof}

\begin{proposition} 
\label{pro:nash-equiv-def}
Soient $S=\AF_{\KK,0}$ et  $s$  le point g\'en\'erique de $S$. Si le diviseur $\E_i$ appartient \`a l'image de l'application de
Nash $\mathcal{N}_{V}$, alors $\N(\widetilde{\E}_i)_s\not \subset \N(\widetilde{\E}_j)_s$ pour tout $j\neq i$.
\end{proposition}

\begin{proof}
Pour chaque $m\geq 1$, soit  $\mt{m}:V_{\infty}\rightarrow V_m$ et $\mts{m}:\arcSpace{W}{S}  \rightarrow \mJet{W}{S}{m}$ les morphismes
de troncature canoniques. Pour $m\geq 0$ et $k\geq 1$, on pose $\CS_{k}^{m}:=\overline{\mts{m}(\N(\widetilde{\E}_k))}$ et
$\C_{k}^{m}:=\overline{\mt{m}(\N(\E_k))}$. Comme  $\N(\widetilde{E}_k)$,  $k\geq 1$,  est irr\'eductible (voir le lemme pr\'ec\`edent),
on obtient que $\CS_{k}^{m}$  est irr\'eductible.\\

Sans perte de g\'en\'eralit\'e, on peut supposer que $i=1$.   Comme $\E_1$ appartient \`a l'image de
 l'application de Nash $\mathcal{N}_{V}$, on a $\N(\E_1)\not\subset \N(\E_k)$ pour tout $k>1$. \\

On suppose  qu'il  existe $j>1$ tel que  $\N(\widetilde{\E}_1)_s \subset \N(\widetilde{\E}_j)_s$. 
Ce qui implique que   $\N(\widetilde{\E}_1) 
 \subset \N(\widetilde{\E}_j)$, car $\N(\widetilde{\E}_k)_s$ est dense dans 
 $\N(\widetilde{\E}_k)$ pour tout $k\geq 1$ (voir la Remarque \ref{re:Ns-dense}).\\

 Comme    $\N(\widetilde{\E}_1)\neq \N(\widetilde{\E}_j)$, il existe $m>>0$ assez grand tel que 
 $\CS_{1}^{m} \varsubsetneq  \CS_{j}^{m}$ (en particulier 
 $\dim \CS_{1}^{m}< \dim \CS_{j}^{m}$).
 De plus, on peut supposer que   $\C_{1}^{m}\not\subset \C_{k}^{m}$ et 
 $\codim(\N(\E_k))=\dim \overline{\mt{m}(V_{\infty})}-\dim \C_{k}^{m}$  pour tout $k>1$ 
 (quitte \`a remplacer $m$ par un entier plus grand).\\
 
 On rappelle que  $\CS_{k}^{m}$,  $k\geq 1$,  est irr\'eductible. Comme le morphisme induit par restriction
$\CS_{k}^{m}\rightarrow S$ est dominant et $S=\AF_{\KK,0}$, on obtient que $\CS_{k}^{m}\rightarrow S$, 
$k\geq 1$, est plat (voir la Proposition $9.7$ de \cite{Har77}).\\

Comme $\dim \CS_{1}^{m}< \dim \CS_{j}^{m}$, on obtient que
 $\dim \C_{1}^{m}<\dim( \CS_{j}^{m}\times_S\{0\})$. Ce  qui implique qu'il existe  au moins une composante irr\'eductible
de $\CS_{j}^{m}\times_S\{0\}$ qui contient strictement l'ensemble 
$\C_{1}^{m}$.
Comme $\CS_{j}^{m}\times_S\{0\} \subset \overline{\mt{m}(V_{\infty}^s)}$, il existe $j'>1$ tel que 
$\C_{1}^{m}\varsubsetneq \C_{j'}^{m}$, d'o\`u la contradiction. 
\end{proof}

Soient $v:=(v_1,...,v_n)\in \ZZ^n_{>0}$, $d\in \ZZ_{>0}$ et $\f\in \KK[x_1,...,x_n]$ un polyn\^ome quasi-homog\`ene de type 
$(s,v)$, c'est-\`a-dire $\f$ est homog\`ene de degr\'e $d$  par rapport \`a la graduation  $\nu_v x_i=v_i$. 
Soit $\h\in \KK[x_1,...,x_n]$ tel que $\nu_v\h>d$.  On note $V$ (resp. $V'$) 
l'hypersurface donn\'ee par l'\'equation 
 $\f(x_1,...,x_n)=0$ (resp. $\f(x_1,...,x_n)+\h(x_1,...,x_n)=0$).
On suppose que $V$ (resp. $V'$) a une unique singularit\'e isol\'ee \`a l'origine de $\AF_{\KK}^n$,
  $V$ (resp. $V'$)  ne contient aucune de $T$-orbite de $\AF_{\KK}^n$  et $\f$ (resp $\f_1:=\f+\h$)  est
un polyn\^ome non d\'eg\'en\'er\'e par rapport \`a \'la fronti\`ere de Newton $\Gamma(\f)$. 
On remarque  que le morphisme  $\pi:X(\Sigma)\rightarrow \AF^{n}$  est une r\'esolution plong\'ee 
des vari\'et\'es $V$ et $V'$  et que les fibres 
 exceptionnelles des d\'esingularitations   $\pi_0:X\rightarrow V$,  $\pi_1:X'\rightarrow V'$ 
 induites pour la r\'esolution plong\'ee $\pi$ sont trivialement  hom\'eomorphes.
 Dans le th\'eor\`eme suivant, par abus de notation, on utilise la m\^eme notation pour d\'esigner les composantes irr\'eductibles des fibres exceptionnelles de 
 $\pi_0$ et $\pi_1$.
 
\begin{theoreme}

 \label{th:def-appl-nash}
 
 Si le diviseur $\E$ appartient \`a l'image de l'application de Nash $\mathcal{N}_{V}$, alors $\E$ apparient l'application de 
 Nash $\mathcal{N}_{V'}$

\end{theoreme}

\begin{proof} Soient $S:=\AF_{\KK,0}$ ,
 $\F(x_1,...,x_n,s):=\f(x_1,...,x_n)+\frac{1}{s^d}\h(s^{v_1}x_1,...,s^{v_n}x_n)\in \mathcal{O}_{\AF,0}[x_1,...,x_n]$ 
 et $W$ la d\'eformation de $V$ induite par $F$. 
 Soit $s$ le point g\'en\'erique de $S$. On remarque que $W_s\cong V_1\times \{s\} $, ce qui implique que 
 $\arcSpace{W_s}{\{s\}}\cong V'_{\infty}\times \{s\}$. On suppose que $E$ 
 n'appartient pas \`a l'image de l'application de Nash $\mathcal{N}_{V'}$, alors 
 il existe une composante irr\'eductible $\E'$ de fibre exceptionnelle $\pi_1^{-1}(0)$ tel que $\N(\E)\subset \N(\E')$. D'o\`u
 $\N(\E)\times\{s\}\subset \N(\E')\times \{s\}$. Par cons\'equence, on a $\N(\widetilde{E})_s\subset \N(\widetilde{E}')_s$. Ce qui rentre en contradiction 
 avec la proposition  \ref{pro:nash-equiv-def}.
 \end{proof}

\section{Deux familles d'exemples d'hypersufaces de $\bm{\AF^4_{\KK}}$ avec l'application de Nash bijective} 
\label{sec:Deux-fam-exemp-hyp}

Dans cette section on consid\`ere deux familles d'hypersurfaces normales de $\AF_{\KK}^4$ et
on montre que pour chaque hypersurface consid\'er\'ee, l'application de Nash qui y est associ\'ee est bijective.

De plus, dans chaque cas on donne explicitement  une  d\'esingularitations o\`u tous les composantes irr\'eductibles de sa fibre exceptionnelle
sont des diviseurs essentiels. Comme on a dit dans l'introduction,  \'etant donn\'e une vari\'et\'e $V$
singuli\`ere et une d\'esingularization $\pi$ de $V$, d\'eterminer si une composante irr\'eductible de la fibre exceptionnelle
de $\pi$ est ou non un diviseur essentiel est un probl\`eme  qui est, en g\'en\'eral, difficile \`a r\'esoudre.\\

Dans les deux cas on utilise la m\^eme m\'ethode de d\'emonstration. Cette m\'ethode est une 
adaptation de celle d\'evelopp\'ee dans l'article \cite{Ley11a}.\\

Chaque famille d'hypersurface consid\'er\'ee est le sujet d'une des deux sous-sections suivantes.\\

\subsection{La premi\`ere
famille d'exemples}

\label{se:PFEA4}

Soit $V$ l'hypersurface de $\AF_{\KK}^4$ donn\'ee par une \'equation du type:

\begin{center}
 $\f(x_1,x_2,x_3,x_4):=\h_q(x_1,x_2)+\hk_{pq}(x_3,x_4)$
\end{center}
o\`u $p\geq 2$, $q\geq 2$ sont deux entiers et $\h_q$, $\hk_{pq}$ 
sont deux polyn\^omes homog\`enes sans facteur multiple. 
De plus, $\h_q$  (resp. $\hk_{pq}$) est de degr\'e $q$  (resp $pq$) et le polyn\^ome 
$\f$ n'est pas d\'eg\'en\'er\'e par rapport \`a la fronti\`ere de Newton $\Gamma(\f)$.

\begin{example} Le polyn\^ome $\f=x_1^q+x_2^q+x_3^{pq}+x_4^{pq}$, $p\geq 2$, $q\geq 2$,  
satisfait les hypoth\`eses ci-dessus.  Dans ce cas la vari\'et\'e $V$ est une hypersurface de Pham-Brieskorn.
\end{example}
 
\begin{remarque}  
\label{re:hq-hpq_no-div}
\`A un automorphisme  lin\'eaire  de  $\AF_{\KK}^4$ pr\`es,  $x_1$ et  $x_2$  (resp. $x_3$, $x_4$) ne divisent pas $\h_q$ (resp. $\hk_{pq}$).
\end{remarque}

Le r\'esultat principal de cette section est le th\'eor\`eme suivant:

\begin{theoreme}
\label{th:premier-the-hyp-A4}
 L'application de Nash $\mathcal{N}_V$ associ\'ee \`a l'hypersurface $V$ est bijective et le nombre de diviseurs essentiels sur $V$ est \'egal \`a $(p-1)q+1$.
\end{theoreme}
Le r\'esultat suivant est une cons\'equence directe des Th\'eor\`emes  \ref{th:premier-the-hyp-A4}, \ref{th:def-appl-nash} et la
Proposition \ref{pr:des_prim_famil_hyp_A4}.
Soit $\h\in \KK[x_1,x_2,x_3,x_4]$ tel que les exposants de ses mon\^omes  appartiennent \`a l'ensemble $\Gamma_{+}(\f)-\Gamma(\f)$.
On note $V'$ l'hypersurface  de  $\AF_{\KK}^4$ donn\'ee par l'\'equation $\f(x_1,x_2,x_3,x_4)+\h(x_1,x_2,x_3,x_4)=0$.

\begin{corollaire}
 L'application de Nash $\mathcal{N}_{V'}$ associ\'ee \`a l'hypersurface $V'$ 
 est bijective et le nombre de diviseurs essentiels sur $V'$ est \'egal \`a $(p-1)q+1$.
\end{corollaire}

D'abord, on utilise la g\'eom\'etrie torique pour r\'esoudre la singularit\'e de $V$. 
Ensuite, on d\'emontre le Th\'eor\`eme \ref{th:premier-the-hyp-A4}.\\

 Par abus de notation, on note $0$ l'origine de $\AF_{\KK}^4$. La proposition suivante r\'esulte d'un calcul direct.
\begin{proposition}
\label{pr:premier-pr-sing-iso}
 L'hypersurface $V$ est normale et le point $0$ est son unique point singulier. De plus, $V$ ne contient aucune $T$-orbite de dimension strictement positive.
\end{proposition}

En vertu des Propositions \ref{pr:cr_nor-f} et \ref{pr:premier-pr-sing-iso}, pour r\'esoudre la singularit\'e
de $V$, il suffit trouver une subdivision r\'eguli\`ere admissible de  $\EN{\f}$. 
Dans la suite, on montre explicitement une subdivision r\'eguli\`ere admissible de $\EN{\f}$.\\

La Figure \ref{fi:poly_newt_pri_fami_A4} repr\'esente la 
face compacte du poly\`edre de Newton $\Gamma_{+}(\f)$, c'est-\`a-dire la fronti\`ere de Newton 
$\Gamma(\f)$. On rappelle que $\f=\h_q(x_1,x_2)+\hk_{pq}(x_3,x_4)$ et que $x_1$ et  $x_2$  (resp. $x_3$, $x_4$)
ne divisent pas $\h_q$ (resp. $\hk_{pq}$), voir la Remarque \ref{re:hq-hpq_no-div}.

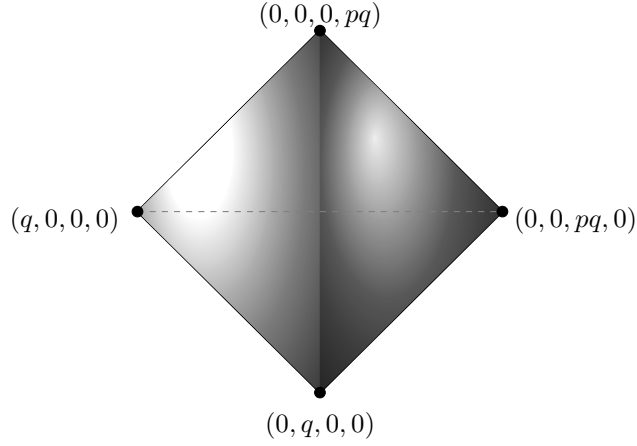
\begin{figure}[!h]

\begin{tikzpicture}[x=0.8cm,y=0.8cm ]
\tikzfading[name=fade1,left color=transparent!80,right color=transparent!80, inner color=transparent!80,outer color=transparent!60]
 \tikzfading[name=fade2,left color=transparent!10,right color=transparent!80, inner color=transparent!50,outer color=transparent!30]

   \draw[thick] (180:3)  -- (270:3) -- (0:3)  -- (90:3) -- (180:3);
   \draw[very thick] (90:3) --  (270:3);

\shade[ball color=white, path fading=fade1] (180:3) -- (90:3) -- (270:3);
\shade[ball color=gray, path fading=fade2] (90:3) -- (270:3)-- (0:3);
\foreach \i in {0,0.2,...,5.9}
\draw[gray] (-3+\i,0)--(-3+\i+0.1,0);

\draw (0:4.2) node[inner sep=-1pt,below=-1pt,rectangle,fill=white] {{\small $(0,0,pq,0)$}};
\draw (180:4.2) node[inner sep=-1pt,below=-1pt,rectangle,fill=white] {{\small $(q,0,0,0)$}};
\draw (90:3.4) node[inner sep=-1pt,below=-1pt,rectangle,fill=white] {{\small $(0,0,0,pq)$}};
\draw (270:3.4) node[inner sep=-1pt,below=-1pt,rectangle,fill=white] {{\small $(0,q, 0,0)$}};

    \draw[fill=black]  (0:3) circle(0.7mm) (180:3) circle(0.7mm)  (90:3) circle(0.7mm)  (270:3) circle(0.7mm);

  \end{tikzpicture}

\caption{La fronti\`ere de Newton $\Gamma(\f)$, o\`u $\f=\h_q(x_1,x_2)+\h_{2q}(x_3,x_4)$}
\label{fi:poly_newt_pri_fami_A4}
\end{figure}

Soit $H$ un hyperplan qui ne contient pas 
l'origine de $\RR^4$ et tel que $H\cap \RR^4_{\geq 0}$ soit un ensemble compact.
La Figure \ref{fi:EN_newt_pri_fami_A4}  repr\'esente l'intersection de $H$ avec la 
subdivision $\EN{\f}$ de $\RR^4_{\geq 0}$. Chaque sommet du diagramme est identifi\'e avec 
le vecteur primitif d'un c\^one de dimension $1$ de l'\'eventail $\EN{\f}$
({\it vecteur extr\'emal})  correspondant.

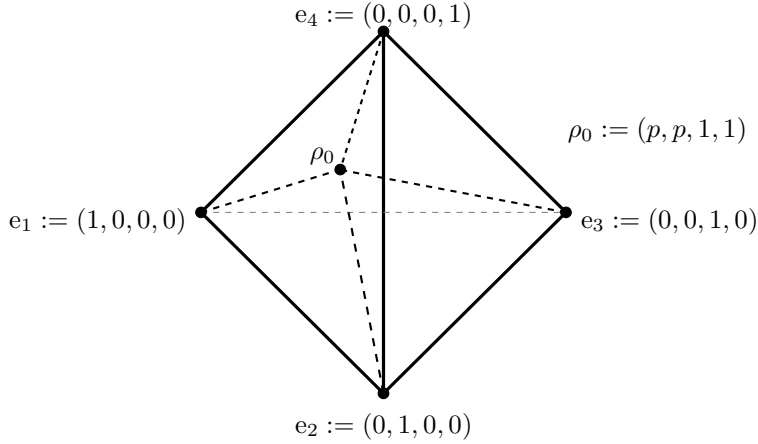
\begin{figure}[!h]

\begin{tikzpicture}[x=0.8cm,y=0.8cm ]

   \draw[very thick] (180:3)  -- (270:3) -- (0:3)  -- (90:3) -- (180:3);
   \draw[very thick] (90:3) --  (270:3);

\foreach \i in {0,0.2,...,5.9}
\draw[gray] (-3+\i,0)--(-3+\i+0.1,0);

\foreach \i in {0,0.2,...,2.28}
\draw[thick] (-3+\i,0.31*\i)--(-3+\i+0.1,0.31*\i+0.032);

\foreach \i in {0,0.05,...,0.71}
\draw[thick] (0-\i,-3+5.22*\i)--(-\i-0.025,-3+5.22*\i+5.22*0.025);

\foreach \i in {0,0.05,...,0.70}
\draw[thick] (0-\i, 3-3.22*\i)--(-\i-0.025,3-3.22*\i-3.22*0.025);

\foreach \i in {0,0.2,...,3.70}
\draw[thick] (3-\i, 0.19*\i)--(3-\i-0.1,0.19*\i +0.19*0.1);

\draw (0:4.7) node[inner sep=-1pt,below=-1pt,rectangle,fill=white] {{\small $\e_3:=(0,0,1,0)$}};
\draw (180:4.7) node[inner sep=-1pt,below=-1pt,rectangle,fill=white] {{\small $\e_1:=(1,0,0,0)$}};
\draw (90:3.4) node[inner sep=-1pt,below=-1pt,rectangle,fill=white] {{\small $\e_4:=(0,0,0,1)$}};
\draw (270:3.4) node[inner sep=-1pt,below=-1pt,rectangle,fill=white] {{\small $\e_2:=(0,1, 0,0)$}};
\draw (-1,1) node[inner sep=-1pt,below=-1pt,rectangle,fill=white] {{\small $\rho_0$}};

\draw (4.5,1.5) node[inner sep=-1pt,below=-1pt,rectangle,fill=white] {{\small $\rho_0:=(p,p,1,1)$}};

    \draw[fill=black]  (0:3) circle(0.7mm) (180:3) circle(0.7mm)  (90:3) circle(0.7mm)  (270:3) circle(0.7mm) (-0.71,0.71) circle(0.7mm);

  \end{tikzpicture}
\caption{Section de l'\'eventail de Newton $\EN{\f}$, o\`u $\f=\h_q(x_1,x_2)+\hk_{pq}(x_3,x_4)$}
\label{fi:EN_newt_pri_fami_A4}
\end{figure}

On note $\sigma_j$, $1\leq j\leq 4$, le c\^one de dimension $4$ de $\EN{\f}$ engendr\'e par le vecteur $(p,p,1,1)$ et l'ensemble 
$\{\e_i\mid 1\leq i\leq 4,\;i\neq j\}$.\\

La proposition suivante r\'esulte d'un simple calcul.

\begin{proposition}
 Les c\^ones $\sigma_3$ et $\sigma_4$ sont r\'eguliers.
\end{proposition}

De gauche \`a droite: la Figure \ref{fi:sub_sigma12_pri_fami_A4} repr\'esente l'intersection du plan $H$ avec une  subdivision du c\^one $\sigma_2$ et du c\^one $\sigma_1$. On rappelle que chaque sommet des diagrammes est identifi\'e avec le  vecteur extr\'emal correspondant.

\begin{figure}[!h]

\begin{tikzpicture}[x=0.9cm,y=0.9cm ]

   \draw[very thick] (180:3)  -- (270:3) -- (0:3)  -- (90:3) -- (180:3);
   \draw[very thick] (90:3) --  (270:3);

\foreach \i in {0,0.2,...,5.9}
\draw[gray] (3-\i,0)--(3-\i-0.1,0);

\draw[thick] (0,-3)--(2/5,-8/5);
\draw[thick] (3/5,-4.5/5)-- (1,0.5);

\draw[thick] (3,0)--(2/3,-2/3);

\draw[thick] (3,0)--(1/3,-5.5/3);

\draw[thick] (0,3)--(2/3,-2/3);

\draw[thick] (0,3)--(1/3,-5.5/3);

\foreach \i in {2,4,...,40}
\draw[thick] (0.05*\i/3 +0.075*\i -3, -0.05*\i/3)--(0.05*\i/3 +0.075*\i -3 +0.55/6, -0.05*\i/3-0.1/6);

\foreach \i in {2,4,...,40}
\draw[thick] (0.05*\i/6 +0.075*\i -3, -0.05*5.5*\i/6)--(0.05*\i/6 +0.075*\i -3 +0.55/6, -0.05*5.5*\i/6-0.55/12);

\foreach \i in {2,4,...,8}
\draw[thick] (2*0.02*\i +3*0.02*10-3*0.02*\i ,-8*0.02*\i-4.5*0.02*10+4.5*0.02*\i)--  (2*0.02*\i +2*0.02+3*0.02*10-3*0.02*\i -3*0.02 ,-8*0.02*\i -8*0.02 -4.5*0.02*10+4.5*0.02*\i +4.5*0.02);

\draw[thick] (3,0)--(1,0.5);

\draw[thick] (0,3)--(1,0.5);

\foreach \i in {0,0.2,...,3.70}
\draw[thick] (-3+\i, 0.125*\i)--(-3+\i+0.1,0.125*\i -0.125*0.1);

\draw (0:3.4) node[inner sep=-1pt,below=-1pt,rectangle,fill=white] {{\small $u_1$}};
\draw (-3.5,-0.5) node[inner sep=-1pt,below=-1pt,rectangle,fill=white] {{\small $(1,0,0,0)$}};
\draw (90:3.4) node[inner sep=-1pt,below=-1pt,rectangle,fill=white] {{\small $u_2$}};
\draw (270:3.4) node[inner sep=-1pt,below=-1pt,rectangle,fill=white] {{\small $\rho_0$}};
\draw (1.4,0.8) node[inner sep=-1pt,below=-1pt,rectangle,fill=white] {{\small $\rho_{p-1}$}};
\draw (1.2,-2.3/3) node[inner sep=-1pt,below=-1pt,rectangle,fill=white] {{\small $\rho_{p-2}$}};
\draw (1.8/3,-5.9/3) node[inner sep=-1pt,below=-1pt,rectangle,fill=white] {{\small $\rho_1$}};

\draw (-2,2) node[inner sep=-1pt,below=-1pt,rectangle,fill=white] { $\sigma_2$};
    \draw[fill=black]  (2/3,-2/3) circle(0.7mm) (1/3,-5.5/3) circle(0.7mm) (0:3) circle(0.7mm) (180:3) circle(0.7mm)  (90:3) circle(0.7mm)  (270:3) circle(0.7mm) (1,0.5) circle(0.7mm);

  \end{tikzpicture}
\begin{tikzpicture}[x=0.9cm,y=0.9cm ]

   \draw[very thick] (180:3)  -- (270:3) -- (0:3)  -- (90:3) -- (180:3);
   \draw[very thick] (90:3) --  (270:3);

\foreach \i in {0,0.2,...,5.9}
\draw[gray] (-3+\i,0)--(-3+\i+0.1,0);

\draw[thick] (-3,0)--(-1,0.5);

\draw[thick] (0,3)--(-1,0.5);

\draw[thick] (0,-3)--(-2/5,-8/5);
\draw[thick] (-3/5,-4.5/5)-- (-1,0.5);

\draw[thick] (-3,0)--(-2/3,-2/3);

\draw[thick] (-3,0)--(-1/3,-5.5/3);

\draw[thick] (0,3)--(-2/3,-2/3);

\draw[thick] (0,3)--(-1/3,-5.5/3);

\foreach \i in {2,4,...,40}
\draw[thick] (-0.05*\i/3 -0.075*\i +3, -0.05*\i/3)--(-0.05*\i/3 -0.075*\i +3 -0.55/6, -0.05*\i/3-0.1/6);

\foreach \i in {2,4,...,40}
\draw[thick] (-0.05*\i/6 -0.075*\i +3, -0.05*5.5*\i/6)--(-0.05*\i/6 -0.075*\i +3 -0.55/6, -0.05*5.5*\i/6-0.55/12);

\foreach \i in {2,4,...,8}
\draw[thick] (-2*0.02*\i -3*0.02*10+3*0.02*\i ,-8*0.02*\i-4.5*0.02*10+4.5*0.02*\i)--  (-2*0.02*\i -2*0.02-3*0.02*10+3*0.02*\i +3*0.02 ,-8*0.02*\i -8*0.02 -4.5*0.02*10+4.5*0.02*\i +4.5*0.02);

\foreach \i in {0,0.2,...,3.70}
\draw[thick] (3-\i, 0.125*\i)--(3-\i-0.1,0.125*\i +0.125*0.1);

\draw (3.5,-0.5) node[inner sep=-1pt,below=-1pt,rectangle,fill=white] {{\small $(0,1,0,0)$}};
\draw (180:3.4) node[inner sep=-1pt,below=-1pt,rectangle,fill=white] {{\small $u_1$}};
\draw (90:3.4) node[inner sep=-1pt,below=-1pt,rectangle,fill=white] {{\small $u_2$}};
\draw (270:3.4) node[inner sep=-1pt,below=-1pt,rectangle,fill=white] {{\small $\rho_0$}};
\draw (-1.3,0.8) node[inner sep=-1pt,below=-1pt,rectangle,fill=white] {{\small $\rho_{p-1}$}};
\draw (-1.1,-2.3/3) node[inner sep=-1pt,below=-1pt,rectangle,fill=white] {{\small $\rho_{p-2}$}};
\draw (-1.6/3,-5.9/3) node[inner sep=-1pt,below=-1pt,rectangle,fill=white] {{\small $\rho_1$}};
\draw (-2,2) node[inner sep=-1pt,below=-1pt,rectangle,fill=white] {$\sigma_1$};
\draw (3.5,2) node[inner sep=-1pt,below=-1pt,rectangle,fill=white] {{\small $u_1:=(0,0,1,0)$}};

\draw (3.5,1.5) node[inner sep=-1pt,below=-1pt,rectangle,fill=white] {{\small $u_2:=(0,0,0,1)$}};

\draw (4,2.5) node[inner sep=-1pt,below=-1pt,rectangle,fill=white]{{\small $0\leq j\leq p-1$}};
\draw (3.7,3) node[inner sep=-1pt,below=-1pt,rectangle,fill=white] {{\small $\rho_j:=(p-j,p-j, 1,1)$}};

    \draw[fill=black]  (-2/3,-2/3) circle(0.7mm) (-1/3,-5.5/3) circle(0.7mm) (0:3) circle(0.7mm) (180:3) circle(0.7mm)  (90:3) circle(0.7mm)  (270:3) circle(0.7mm) (-1,0.5) circle(0.7mm);

  \end{tikzpicture}
\caption{Subdivision des c\^ones $\sigma_1$ et $\sigma_2$}
\label{fi:sub_sigma12_pri_fami_A4}
\end{figure}
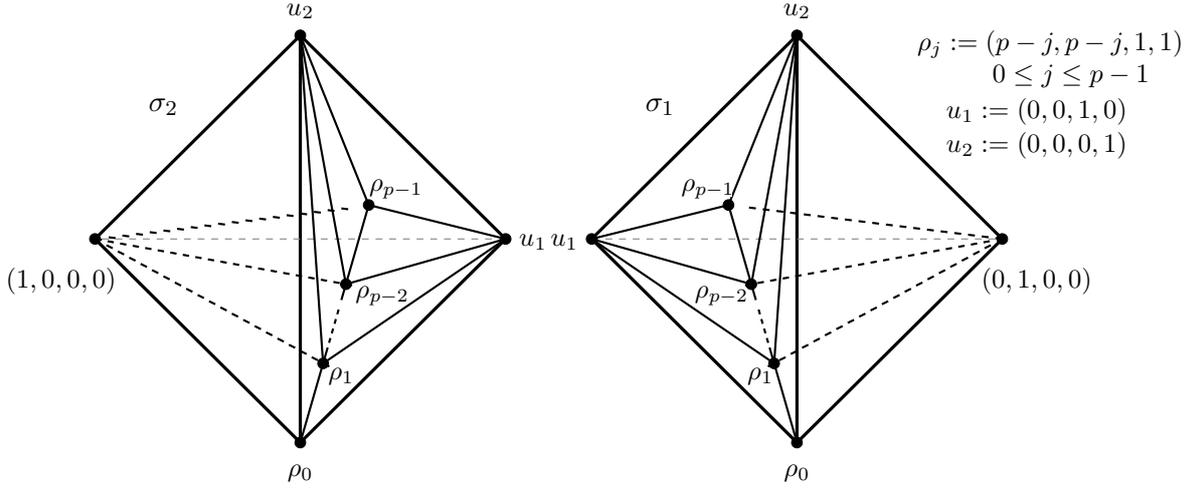

Pour chaque paire d'entiers $1\leq j<p-1$ et $1\leq k\leq 2$,
on note $\sigma_{1jk}$ (resp. $\sigma_{2jk}$) le c\^one de dimension $4$ engendr\'e par 
le vecteur  $(0,1,0,0)$ (resp. $(1,0,0,0)$) et l'ensemble $\{\rho_{j-1}, \rho_j,u_k\}$,   o\`u   
$u_1:=(0,0,1,0)$,  $u_2:=(0,0,0,1)$ et $\rho_j=(p-j,p-j,1,1)$.

On note $\sigma_{1\;p-1}$ (resp. $\sigma_{2\;p-1}$) le c\^one engendr\'e par 
le vecteur $(0,1,0,0)$ (resp. $(1,0,0,0)$) et l'ensemble $\{u_1,u_2,\rho_{p-1}\}$.
On remarque que l'ensemble  $\{ \sigma_{1 j k},\sigma_{2jk}\mid 1\leq j< p-1,1\leq k\leq 2\}\cup \{\sigma_{1\;p-1},\sigma_{2\;p-1},\sigma_3,\sigma_4\}$ engendre un \'eventail, not\'e $\Sigma$. \\

Par une simple inspection des c\^ones de l'\'eventail $\Sigma$, on obtient la proposition suivante:

\begin{proposition}
 \label{pr:GDEN-prim-ex-A4} L'\'eventail $\Sigma$ engendr\'e par l'ensemble  $\{ \sigma_{1 j k},\sigma_{2jk}\mid 1\leq j< p-1,1\leq k\leq 2\}\cup \{\sigma_{1\;p-1},\sigma_{2\;p-1},\sigma_3,\sigma_4\}$ est une $G$-subdivision r\'eguli\`ere de l'\'eventail $\EN{\f}$.
\end{proposition}

Soient $\pi: X(\Sigma)\rightarrow \AF^4_{\KK}$ le morphisme torique induit 
par l'\'eventail $\Sigma$ de la Proposition \ref{pr:GDEN-prim-ex-A4} et 
$\widetilde{V}$ le transform\'e strict de $V$ dans $X(\Sigma)$.
Par abus de notation, on note $\pi:\widetilde{V}\rightarrow V$ la
restriction de $\pi$ \`a $\widetilde{V}$.\\

En vertu de la  correspondance orbites-c\^ones,  les vecteurs $\rho_j$, $0\leq j\leq p-1$,  sont en correspondance biunivoque avec 
les composantes irr\'eductibles de  la fibre exceptionnelle du morphisme 
$\pi: X(\Sigma)\rightarrow \AF^4_{\KK}$.  Pour plus de d\'etails, voir 
\cite{KKMS73} ou \cite{CLS11}. Soit $\D_j$ le diviseur torique associ\'e au vecteur $\rho_j$.

\begin{proposition}
\label{pr:des_prim_famil_hyp_A4}
 La  fibre exceptionnelle  de la d\'esingularisation $\pi:\widetilde{V}\rightarrow V$ est  un diviseur \`a 
 croisements normaux qui est la r\'eunion de $(p-1)q+1$ composantes  irr\'eductibles. Le diagramme suivant repr\'esente la fibre exceptionnelle de $\pi$:

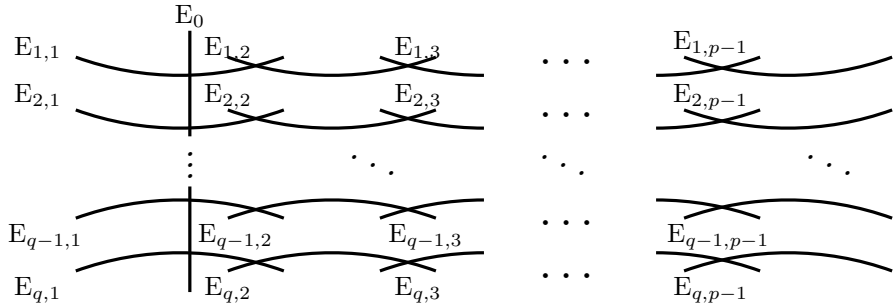
\begin{center} 
 \begin{tikzpicture}[x=1cm,y=0.7cm ]
 \draw[very thick] (0,-2.85)--(0,-0.85);

\draw (0,0)  node[inner sep=-1pt,below=-1pt,rectangle,fill=white] { {\LARGE $\vdots$}};
\draw (2.4,0)  node[inner sep=-1pt,below=-1pt,rectangle,fill=white] { {\LARGE $\ddots$}};
\draw (4.9,0)  node[inner sep=-1pt,below=-1pt,rectangle,fill=white] { {\LARGE $\ddots$}};
\draw (8.4,0)  node[inner sep=-1pt,below=-1pt,rectangle,fill=white] { {\LARGE $\ddots$}};
 \draw[very thick] (0,0.1)--(0,2.1);
\foreach \i in {1,2,3}
\draw (-1.9+2.5*\i-2.6,-2.8+0.15) node[inner sep=-1pt,below=-1pt,rectangle,fill=white] {{\small $\E_{q,\i}$}}
(-1.9+2.5*\i-2.6+0.1,-1.8+0.15) node[inner sep=-1pt,below=-1pt,rectangle,fill=white] {{\small $\E_{q-1,\i}$}}
(-1.9+2.5*\i-2.6,1) node[inner sep=-1pt,below=-1pt,rectangle,fill=white] {{\small $\E_{2,\i}$}}
(-1.9+2.5*\i-2.6,1.9) node[inner sep=-1pt,below=-1pt,rectangle,fill=white] {{\small $\E_{1,\i}$}};

\foreach \i in {4.5}
\draw (-1.9+2.5*\i-2.5,-2.8+0.15) node[inner sep=-1pt,below=-1pt,rectangle,fill=white] {{\small $\E_{q,p-1}$}}
(-1.9+2.5*\i-2.5+0.1,-1.8+0.15) node[inner sep=-1pt,below=-1pt,rectangle,fill=white] {{\small $\E_{q-1,p-1}$}}
(-1.9+2.5*\i-2.5,1) node[inner sep=-1pt,below=-1pt,rectangle,fill=white] {{\small $\E_{2,p-1}$}}
(-1.9+2.5*\i-2.5,2) node[inner sep=-1pt,below=-1pt,rectangle,fill=white] {{\small $\E_{1,p-1}$}};

\draw (0,2.5) node[inner sep=-1pt,below=-1pt,rectangle,fill=white] {{\small $\E_0$}};

\foreach \i in {0,2,8}
\draw[very thick] (-1.5+\i,-1.6+0.15) arc (110:70:4cm) (-1.5+\i,-2.6+0.15) arc (110:70:4cm) (-1.5+\i,0.6) arc (250:290:4cm)  (-1.5+\i,1.6) arc (250:290:4cm);

\draw[very thick] (-1.5+4,-1.6+0.15) arc (110:90:4cm) (-1.5+4,-2.6+0.15) arc (110:90:4cm) (-1.5+4,0.6) arc (250:270:4cm)  (-1.5+4,1.6) arc (250:270:4cm);

\draw[very thick](-1.5+9,-1.6+0.15) arc (70:90:4cm) (-1.5+9,-2.6+0.15) arc (70:90:4cm) (-1.5+9,0.6) arc (290:270:4cm)  (-1.5+9,1.6) arc (290:270:4cm);

\draw (-1.5+6.5,-1.6+0.15) node[inner sep=-1pt,below=-1pt,rectangle,fill=white] { {\LARGE $\cdots$}};

\draw (-1.5+6.5,-2.6+0.15) node[inner sep=-1pt,below=-1pt,rectangle,fill=white] { {\LARGE $\cdots$}};

\draw  (-1.5+6.5,0.6) node[inner sep=-1pt,below=-1pt,rectangle,fill=white] { {\LARGE $\cdots$}};

\draw (-1.5+6.5,1.6)  node[inner sep=-1pt,below=-1pt,rectangle,fill=white] { {\LARGE $\cdots$}};

\end{tikzpicture}
\end{center}
o\`u $\E_0:=\D_0\cap \widetilde{V}$ 
et les  diviseurs $\E_{i,j}$, $1\leq i\leq q$, sont les composantes  irr\'eductibles 
de $\D_j\cap \widetilde{V}$, $1\leq j\leq p-1$. 
Les intersections $\E_{0}\cap \E_{i,1}$ et $\E_{i,j}\cap\E_{i,j+1}$, $1\leq i\leq q $, $1\leq j\leq p-2$, sont des courbes lisses et irr\'eductibles.  De plus, les diviseurs $\E_{i,j}$, $1\leq i\leq q$, $1\leq j\leq p-1$ sont des vari\'et\'es  rationnelles.
\end{proposition}
\begin{remarque}
 Dans les r\'esultats suivants, on montre que les diviseurs essentiels sur $V$ sont exactement  les diviseurs exceptionnels de la d\'esingularisation $\pi:\widetilde{V}\rightarrow V$ 
\end{remarque}

\begin{proof}[{\it D\'emonstration de la Proposition }\ref{pr:des_prim_famil_hyp_A4}] 

 Cette proposition peut \^etre obtenue en utilisant les m\'ethodes d\'evelopp\'ees dans l'article 
 \cite{Oka87}. Or on va donner une preuve adapt\'ee \`a notre cas.\\

On rappelle que le morphisme $\pi:X(\Sigma)\rightarrow \AF^{4}_{\KK}$ est une r\'esolution plong\'ee 
de $V$, d'o\`u la fibre exceptionnelle de la d\'esingularisation  $\pi:\widetilde{V} \rightarrow V$ est 
un diviseur \`a croisements normaux. 
Par cons\'equent, les intersections $\E_{0}\cap \E_{i,1}$ et $\E_{i,j}\cap\E_{i,j+1}$, $1\leq i\leq p $, $1\leq j\leq p-2$ sont une r\'eunion de courbes lisses.\\

 On fixe un entier $1\leq j<p-1$ et on consid\`ere le c\^one $\sigma_{2 j 2}\in \Sigma$. On rappelle que   $\sigma_{2 j 2}$ est le c\^one r\'egulier  engendr\'e par les vecteurs $(1,0,0,0), (0,0,0,1)$ ,  $\rho_j=(p-j,p-j,1,1)$, $\rho_{j-1}=(p-j+1,p-j+1,1,1)$. On note $U$ ($\cong \AF^4_{\KK}$) l'ouvert torique qui est en correspondance avec le c\^one $\sigma_{2 j2}$.\\

La restriction \`a $U$ du morphisme torique  $\pi:X(\Sigma)\rightarrow \AF^{4}_{\KK}$ est d\'efinie de la fa\c con suivante:
\begin{center}
 $\pi\mid_U:U\rightarrow \AF^{4}_{\KK}$, $(y_1,y_2,y_3,y_4)\mapsto (x_1,x_2,x_3,x_4):=(y_1y_2^{p-j+1}y_3^{p-j},y_2^{p-j+1}y_3^{p-j},y_2y_3,y_2y_3y_4)$.
\end{center}

On remarque que $U\cap\pi^{-1}(0)=\{y_2=0\}\cup\{y_3=0\}$. La vari\'et\'e $\widetilde{V}\cap U$ est donn\'ee par l'\'equation suivante:
\begin{center}
$\h_{q}(y_1,1)+y_2^{q(j-1)}y_3^{qj}\hk_{pq}(1,y_4)=0$.
\end{center}

Si  $j=1$, l'intersection $\widetilde{V}\cap U\cap\{y_2=0\}$ est une vari\'et\'e  irr\'eductible et  $\widetilde{V}\cap U\cap \{y_3=0\}$  est donn\'ee par les \'equations $\h_{q}(y_1,1)=0$, $y_3=0$. Par cons\'equent, $\widetilde{V}\cap U\cap\{y_3=0\}$ est  la r\'eunion disjointe  de  $q$ composantes irr\'eductibles (chaque composante irr\'eductible  est associ\'ee \`a une racine du  polyn\^ome $\h_q(y,1)$) et  chaque composante  irr\'eductible  est une  vari\'et\'e   rationnelle. Soient $w\in \KK$ une racine du polyn\^ome $\h_q(y,1)$ et $F_w$ la  composante irr\'eductible  de   $\widetilde{V}\cap U\cap \{y_3=0\}$ associ\'ee \`a $w$. On remarque que $F_w\cap \{y_2=0\}$ est une courbe irr\'eductible.\\

Si $j\geq 2$, les intersections  $\widetilde{V}\cap U\cap \{y_2=0\}$, $\widetilde{V}\cap U\cap \{y_3=0\}$ sont  la r\'eunion disjointe  de  $q$ composantes irr\'eductibles  et  chaque composante  est une  vari\'et\'e  rationnelle. Soit  $F_{2,w}$ (resp. $\F_{3,w}$ la  composante irr\'eductible  de   $\widetilde{V}\cap U\cap\{y_2=0\}$ (resp.  $\widetilde{V}\cap U\cap \{y_3=0\}$)  associ\'ee \`a $w$, o\`u $w\in \KK$ une racine du polyn\^ome $\h_q(y,1)$.  Alors l'intersection $F_{2,w}\cap F_{3,w'}$ n'est pas vide si et seulement si $w=w'$. On remarque que $F_{2,w}\cap F_{3,w}$ est une courbe   irr\'eductible.\\

En proc\'edant de la m\^eme mani\`ere sur tous les c\^ones qui contiennent les
vecteurs extr\'emaux $\rho_{j-1}$ et $\rho_j$, on obtient la proposition. 
\end{proof}

\subsubsection{\bf{Preuve de la bijectivit\'e de l'application de Nash}}
Dans cette section, on d\'emontre le Th\'eor\`eme \ref{th:premier-the-hyp-A4} sur la bijectivit\'e de l'application de Nash pour 
l'hypersurface $V$ de $\AF_{\KK}^4$ donn\'ee par une \'equation du type  $\f(x_1,x_2,x_3,x_4):=\h_q(x_1,x_2)+\hk_{pq}(x_3,x_4)$,
ce qui \'equivaut \`a montrer que tous les wedges admissibles se rel\`event \`a une d\'esingularisation de $V$ 
(voir la section \ref{ssec:nash}). Notre but, dans toute la suite de cette section, est de montrer que pour chaque
diviseur essentiel $\E$  tous les $K$-wedges admissibles centr\'es en $\N(\E)$  se rel\`event \`a la d\'esingularisation  $\widetilde{V}$ de 
la Proposition \ref{pr:des_prim_famil_hyp_A4}, o\`u $K$ est une extension du corps  $\KK$.\\

D'abord, on donne quelques r\'esultats techniques.  Ensuite, on d\'emontre le Th\'eor\`eme \ref{th:premier-the-hyp-A4}.\\

 On note $g$ (resp. Par abus de notation, on note $0$) le point  g\'en\'erique  (resp. ferm\'e) de $\spec K[[t]]$. 
 \'Etant donn\'e un $K$-arc  $\alpha:\spec K[[t]]\rightarrow V$, $\alpha(0)=0$,  
 qui n'est pas concentr\'e en un hyperplan $x_i=0$, $1\leq i\leq 4$,  
 on note  $\mu=(\mu_1,\mu_2,\mu_3,\mu_4)\in\ZZ_{> 0}^{4}$ le {\it vecteur principal} du $K$-arc $\alpha$, c'est-\`a-dire 

\begin{center}
$\mu:=(\ord_t\alpha^{\star }(x_1), \ord_t\alpha^{\star }(x_2),\ord_t\alpha^{\star }(x_3),\ord_t\alpha^{\star }(x_4))$.
\end{center}
o\`u  $\com{\alpha}$ est le comorphisme du $K$-arc $\alpha$. On peut donc \'ecrire le comorphisme $\com{\alpha}$ de la fa\c con suivante:
\begin{center}
$\com{\alpha}(x_i)=t^{\mu_{i}}\alpha_i$, $1\leq i\leq 4$,
\end{center}
o\`u  les $\alpha_i$ sont des s\'eries formelles  inversibles dans $K[[t]]$.\\

Dans la proposition suivante, on utilise la notation  ci-dessus et de la Proposition \ref{pr:des_prim_famil_hyp_A4}.\\

On remarque que si un $K$-arc $\alpha:\spec K[[t]]\rightarrow V$ n'est pas concentr\'e en $0$, alors $\alpha$ se rel\`eve \`a $\widetilde{V}$, car le morphisme $\pi$ est une d\'esingularisation de $V$. De plus, si le point  $\widetilde{\alpha}(0)$  est le point g\'en\'erique d'une composante irr\'eductible de la fibre exceptionnelle de $\pi$, alors le $K$-arc $\alpha$ n'est pas concentr\'e en une hypersurface de $V$.

\begin{proposition} 
\label{pr:vect-prin} Soient  $\pi:\widetilde{V}\rightarrow V$ la d\'esingularisation de la 
Proposition \ref{pr:des_prim_famil_hyp_A4}, $\E$  
une composante irr\'eductible de la fibre exceptionnelle de $\pi$ et $\alpha$  un $K$-arc qui n'est pas concentr\'e en $0$ 
tel que  le rel\`evement $\widehat{\alpha}$ de $\alpha$ \`a $\widetilde{V}$ est transverse \`a $\E$. 
De plus, on suppose que   $\widehat{\alpha}(0)$ est le point g\'en\'erique de $\E$. Alors, on a:

\begin{itemize}
 \item[i)] si le diviseur $\E$ est le diviseur $\E_0$, alors le vecteur principal $\mu$ du $K$-arc $\alpha$ est le vecteur $\rho_0=(p,p,1,1)$;
\item[ii)] si le diviseur $\E$ est le diviseur $\E_{i,j}$, $1\leq i\leq q$, $1\leq j\leq p-1$, alors le vecteur principal $\mu$ du $K$-arc $\alpha$ est le vecteur $\rho_j=(p-j,p-j,1,1)$;
\end{itemize}
\end{proposition}
 \begin{proof}
 La d\'emonstration est analogue \`a celle de la Proposition $2.12$ de l'article \cite{Ley11a}. Dans le but de fournir
 un texte auto-contenu, nous en donnerons quand m\^eme une preuve.\\
 
 Comme $\pi: X(\Sigma)\rightarrow \AF^{4}$ est une r\'esolution plong\'ee de $V$, il existe un diviseur torique exceptionnel  $\D$ de $\pi$
qui est transverse \`a $\widetilde{V}$ et tel que $\E$ est une composante irr\'eductible  de $\D\cap \widetilde{V}$. 

\'Etant donn\'e un vecteur extremal $\varrho$ de $\Sigma$, on note $\D_{\varrho}$ le diviseur torique 
associ\'e \`a $\varrho$. 
Soit $\varrho_1$ le vecteur extremal de $\Sigma$ tel que $\D=\D_{\varrho_1}$ et  soient $\varrho_i$, $2\leq i\leq 4$ 
des vecteurs extr\'emaux de $\Sigma$  adjacents \`a $\varrho_1$, c'est-\`a-dire il existe un c\^one $\sigma\in \Sigma$ 
de dimension $4$ tel que les vecteurs $\varrho_i$, $1\leq i\leq 4$,  sont vecteurs extr\'emaux de $\sigma$.

 L'entier $i\in \{1,2,3,4\}$ \'etant fix\'e, soit  $m_i\in \ZZ^4$ un vecteur tel que  le caract\`ere  $\chi^{m_i}$ associ\'e \`a 
   $m_i$ soit  une \'equation du diviseur  $\D_{\varrho_i}\cap U_{\sigma}$, 
   (On rappelle que $U_{\sigma}$ est  l'ouvert torique  de $X(\Sigma)$ associ\'e au c\^one  $\sigma$).
 On a donc $\langle m_i,\varrho_j \rangle=\delta_{ij}$, o\`u $\delta_{ij}$ est le symbole de Kronecker. 
 Quitte \`a remplacer les vecteurs $\varrho_i$, $2\leq i\leq 4$, on peut  supposer que  $U_{\sigma}\cap E\neq \emptyset$.

On rappelle que $\widehat{V}$ est  transverse au diviseur $\D_{\rho_1}$. 
Comme le $K$-arc $\widehat{\alpha}$ est transverse \`a $\E$, le $K$-arc $\widehat{\alpha}$ 
est transverse \`a $\D_{\rho_1}$.

Soit $\mu$ le vecteur principal  de l'arc $\alpha$.  
Comme $\widehat{\alpha}(0)$ est le point g\'en\'erique de $\E$, on peut supposer que 
$\widehat{\alpha}(0)\not \in \D_{\varrho_i} $, $i=2,3,4$.
Par cons\'equent, on obtient que  $\langle m_i, \mu \rangle=\delta_{i 1}$,  $1\leq i\leq 4$.  
Ceci implique que $\mu=\varrho_{1}$. En utilisant explicitement la r\'esolution $\pi$, on ach\`eve la d\'emonstration de la proposition.
\end{proof}

\'Etant fix\'e une composante irr\'eductible $\E$ de la fibre exceptionnelle du morphisme 
$\pi:\widetilde{V}\rightarrow V$ 
($\pi$ est la d\'esingularisation de la Proposition \ref{pr:des_prim_famil_hyp_A4}), 
on consid\`ere un $K$-wedge $\omega:\spec K[[s,t]]\rightarrow V$,  tel que son centre (resp. son arc g\'en\'erique)  est le point g\'en\'erique de 
$\N(\E)$, not\'e $\alpha_{\E}$, (resp. est un point qui appartient \`a $V_{\infty}^{s}$),
et on pose 

\begin{center}
$(\eta_1,\eta_2,\eta_3,\eta_4):=(\ord_t\com{\omega}(x_1),\ord_t\com{\omega}(x_2),\ord_t\com{\omega}(x_3),\ord_t\com{\omega}(x_{4}))\in \ZZ^{4}_{>0}$, 
\end{center}
o\`u $\com{\omega}$ est le comorphisme du $K$-wedge $\omega$.  On peut \'ecrire le comorphisme $\com{\omega}$ de la fa\c con suivante:
\begin{center}
$\com{\omega}(x_i)=t^{\eta_{i}}\varphi_i$, $1\leq i\leq 4$,
\end{center}
o\`u  les  $\varphi_i$ sont des s\'eries formelles  dans $K[[s,t]]$ qui ne sont pas divisibles par $t$.\\

On pose 
\begin{center}
 $(\mu_1,\mu_2,\mu_3,\mu_4):=(\ord_t\com{\alpha_{\E}}(x_1),\ord_t\com{\alpha_{\E}}(x_2),\ord_t\com{\alpha_{\E}}(x_3),\ord_t\com{\alpha_{\E}}(x_{4})),$
\end{center}
le vecteur principal du $\KK_{\alpha_{\E}}$-arc $\alpha_{\E}$, o\`u $\KK_{\alpha_{\E}}$ est le corps r\'esiduel du point g\'en\'erique de
$\N(\E)$.\\

 D'apr\`es la Proposition \ref{pr:vect-prin}, il existe  un entier $0\leq j\leq p-1$, tel que le vecteur 
principal du $\KK_{\alpha_{\E}}$-arc $\alpha_{\E}$ est le vecteur $(\mu_1,\mu_2,\mu_3,\mu_4)=(p-j,p-j,1,1)$.
La proposition suivante est le r\'esultat cl\'e de la preuve du Th\'eor\`eme \ref{th:premier-the-hyp-A4}
\begin{proposition}
\label{pr:ser_for_inv_prem_cas_A4}
 Les s\'eries formelles $\varphi_i$, $1\leq i\leq 4$,  sont inversibles.
\end{proposition}

D'abord, finissons la preuve du Th\'eor\`eme \ref{th:premier-the-hyp-A4}.
On remarque que  la Proposition \ref{pr:ser_for_inv_prem_cas_A4}  
est valable pour toute composante irr\'eductible $\E$ de la fibre exceptionnelle de la
d\'esingularisation $\pi:\widetilde{V}\rightarrow V$ 
et pour tout $K$-wedge $\omega:\spec K[[s,t]]\rightarrow V$, 
tel que son centre (resp. son arc g\'en\'erique)  est le point g\'en\'erique de $\N(\E)$
(resp. est un point qui appartient \`a $V_{\infty}^{s}$). 
 En vertu de la Proposition
\ref{pr:ser_for_inv_prem_cas_A4}, le $K$-wedge $\omega$ se rel\`eve \`a $\widetilde{V}$. 
Par cons\'equent, tout diviseur exceptionnel $\E$ de la 
d\'esingularisation $\pi:\widetilde{V}\rightarrow V$ est  un  diviseur essentiel sur $V$ qui appartient \`a l'image de l'application de Nash
(voir section \ref{ssec:nash}).

\begin{proof}[{\it D\'emonstration de la Proposition} \ref{pr:ser_for_inv_prem_cas_A4}]
 
 On d\'efinit l'application suivante:\\
 
  $\FI:K[[s,t]]\backslash\{0\}\rightarrow \ZZ_{\geq 0}$, o\`u $\FI(\phi)$ 
  est le nombre de facteurs irr\'eductibles de $\phi$ compt\'es avec multiplicit\'e.\\

 Alors  on a:

\begin{center}
 $\FI(\varphi_i)\leq \mu_i-\eta_i$, $1\leq i\leq 4$,
\end{center} 
 De plus,  $\varphi_i$ $1\leq i\leq 4$  est inversible si et seulement si $\mu_i-\eta_i=0$.\\

On rappelle qu'il existe  un entier $0\leq j\leq p-1$, tel que le vecteur 
principal du $\KK_{\alpha_{\E}}$-arc $\alpha_{\E}$ est le vecteur $(\mu_1,\mu_2,\mu_3,\mu_4)=(p-j,p-j,1,1)$. Par cons\'equent,  on a  $\eta_3=\eta_4=\mu_3=\mu_4=1$ et les s\'eries  formelles $\varphi_3$, $\varphi_4$ sont inversibles. \\

Si $j=p-1$, alors $\mu_1=\mu_2=\eta_1=\eta_2=1$. Par cons\'equent, les s\'eries $\varphi_1$ et $\varphi_2$ sont inversibles. Ce qui ach\`eve la d\'emonstration dans ce cas. \\

Dans la suite, on suppose que $j<p-1$. \\

Le $K$-wedge $\omega$ satisfait l'\'equation $\h_q(x_1,x_2)+\hk_{pq}(x_3,x_4)=0$, ainsi on obtient:

\begin{center}
 $\h_q(t^{\eta_1}\varphi_1,t^{\eta_2}\varphi_2)+t^{pq}\hk_{pq}(\varphi_3,\varphi_4)=0$.
\end{center}

 On rappelle que $\h_q(x_1,x_2)$ est un polyn\^ome homog\`ene de degr\'e $q$ sans facteur multiple et que 
 $x_1$ et $x_2$ 
ne le divisent pas (voir la Remarque \ref{re:hq-hpq_no-div}).  

On remarque que $\eta_i\leq \mu_i\leq p$.
Si on suppose que $\eta_1\neq \eta_2$, alors $\ord_t \h_q(t^{\eta_1}\varphi_1,t^{\eta_2}\varphi_2)
 <pq$.  Ce qui n'est pas possible car $\ord_t t^{pq}\hk_{pq}(\varphi_3,\varphi_4)\geq pq$. D'o\`u $\eta_1=\eta_2$.
 
 Par cons\'equent, on a: 
\begin{center}
 $\h_q(\varphi_1,\varphi_2)+t^{(p-\eta_1)q}\hk_{pq}(\varphi_3,\varphi_4)=0$.
\end{center}

Si $\eta_1=p$, alors $\mu_1-\eta_1=\mu_2-\eta_2=0$ car $\mu_i\leq p$, $i\in \{1,2\}$, et $\eta_1=\eta_2$.
Comme $\FI(\varphi_1)=\FI(\varphi_2)=0$, on obtient que les s\'eries formelles $\varphi_1$ et $\varphi_2$
sont inversibles.  Ce qui ach\`eve la d\'emonstration dans ce cas.\\ 

Dans  la suite, on suppose que $\eta_1<p$.\\

D'apr\`es la Remarque \ref{re:hq-hpq_no-div}, on peut  \'ecrire le   polyn\^ome $\h_q(x_1,x_2)$ de la fa\c con suivante:

\begin{center}
$\h_q(x_1,x_2) =\prod^{q}_{i=1}(a_ix_1+b_ix_2)$,
\end{center}
o\`u $a_i\neq 0$, $b_i\neq 0$, pour tout $1\leq i\leq q$. Ainsi on obtient la relation suivante:

\begin{center}
 $\prod^{q}_{i=1}(a_i\varphi_1+b_i\varphi_2)=-t^{(p-\eta_1)q}\hk_{pq}(\varphi_3,\varphi_4)$
\end{center}

On remarque que $t$ divise $a_i\varphi_1+b_i\varphi_2$ si est seulement si $t$ ne divise pas $a_{i'}\varphi_1+b_{i'}\varphi_2$, pour tout $i'\neq i$. En effet, $\h_q$ est un polyn\^ome sans facteur multiple. Donc, s'il existe $1\leq i, i'\leq q$ tel que 
$t$ divise  $a_{i}\varphi_1+b_{i}\varphi_2$ et $a_{i'}\varphi_1+b_{i'}\varphi_2$, alors $t$ divise $\varphi_1$ et $\varphi_2$.
Ce  qui est absurde.

Sans perte de g\'en\'eralit\'e, on peut supposer que $a_1\varphi_1+b_1\varphi_2=\lambda t^{(p-\eta_1)q}$, o\`u $\lambda$ est une s\'erie formelle qui appartient \`a $K[[s,t]]$. Ainsi on obtient la relation suivante:

\begin{center}
 $\lambda \prod^{q}_{i=2}(a_i\varphi_1+b_i\varphi_2)=\hk_{pq}(\varphi_3,\varphi_4)$.
\end{center}

Le lemme suivant est le r\'esultat cl\'e pour la preuve de la Proposition \ref{pr:ser_for_inv_prem_cas_A4}
 
\begin{lemme}
\label{le:ser_for_inv_prem_cas_A4}
  La s\'erie formelle  $\hk_{pq}(\varphi_3,\varphi_4)$ est inversible.
\end{lemme}

D'abord, finissons la preuve de la Proposition \ref{pr:ser_for_inv_prem_cas_A4}. 
Comme $\mu_1-\eta_1=\mu_2-\eta_2$, la s\'erie formelle $\varphi_1$ est inversible si et seulement si 
$\varphi_2$ l'est.

D'apr\`es le lemme pr\'ec\'edent, la s\'erie formelle  $\hk_{pq}(\varphi_3,\varphi_4)$ est inversible. 
Ce qui implique que la s\'erie formelle $\lambda \prod^{q}_{i=2}(a_i\varphi_1+b_i\varphi_2)$ est inversible. 
En particulier  $a_2\varphi_1+b_2\varphi_2$ est inversible, d'o\`u les s\'eries formelles 
$\varphi_1$ et $\varphi_2$ sont inversibles.\\

\begin{proof}[{\it D\'emonstrations du Lemme \ref{le:ser_for_inv_prem_cas_A4}}]
On rappelle que $\omega$ est un $K$-wedge admissible centr\'e en $\N(E)$ (voir la Section \ref{ssec:nash}) et 
que  $\alpha_{\E}$ est le  point g\'en\'erique de $\N(\E)$.

Soit $\lambda:\spec K[[t]]\rightarrow \spec K[[s,t]]$ le morphisme induit par 
l'homomorphisme canonique $\lambda^{\star}:K[[s,t]]\rightarrow K[[s,t]]\slash (s)=K[[t]]$. 
Alors, le morphisme $\alpha=\omega\circ \lambda$ est un $K$-arc sur $V$.
On remarque que le $K$-arc $\alpha$ est un $K$-point de $V_{\infty}$ au-dessus du point  $\alpha_{\E}$.

Soit $\pi: \widetilde{V}\rightarrow V$ la d\'esingularisation de la
Proposition \ref{pr:des_prim_famil_hyp_A4}. En utilisant la propri\'et\'e 
fonctorielle d'espace d'arcs $V_{\infty}$ (voir le Th\'eor\`eme \ref{th:prop-fonct-arc}), on obtient que  le rel\`evement $\widehat{\alpha}$ de $\alpha$ \`a $\widetilde{V}$ est transverse au diviseur  $\E$ et  que   $\widehat{\alpha}(0)$ est le point g\'en\'erique de $\E$, car le centre du $K$-wedge $\omega$ est le point g\'en\'erique de 
$\N(\E)$. En particulier, le $K$-arc $\alpha:\spec K[[t]]\rightarrow V$ n'est pas concentr\'e en un hyperplan $x_i=0$, $1\leq i\leq 4$.\\

On peut \'ecrire le comorphisme $\com{\alpha}$ du $K$-arc $\alpha$ de la fa\c con suivante:
\begin{center}
$\com{\alpha}(x_i)=t^{\mu_{i}}\alpha_i$, $1\leq i\leq 4$,
\end{center}
o\`u  les  $\alpha_i$ sont des s\'eries formelles  inversibles dans $K[[t]]$. On note $a_i$ le terme constant de la s\'erie inversible
$\alpha_i$, $1\leq i\leq 4$.

On remarque que pour $i\in\{3,4\}$, on a  $\varphi_i=a_i+\varphi_i'$, o\`u  $\varphi_i'\in K[[s,t]]$, $i\in \{3,4\}$, est une s\'erie formelle  non inversible, car les s\'eries formelles $\varphi_3$ et $\varphi_4$ sont inversibles. Ainsi, on a:
\begin{center}
 $\hk_{pq}(\varphi_3,\varphi_4)=\hk_{pq}(a_3,a_4)+\psi$,
\end{center}
o\`u $\psi\in K[[s,t]]$ est une s\'erie formelle non inversible.  
Par cons\'equent, la s\'erie  $\hk_{pq}(\varphi_3,\varphi_4)$ est inversible si et seulement si $\hk_{pq}(a_3,a_4)\neq 0$. \\

On rappelle que d'apr\`es la Proposition \ref{pr:vect-prin}, il existe un entier  $0\leq j< p-1$, tel que le vecteur 
principal du $K$-arc $\alpha$ est le vecteur $(\mu_1,\mu_2,\mu_3,\mu_4)=(p-j,p-j,1,1)$.
\\

Soit $j'=j+1$ et on consid\`ere le c\^one $\sigma_{2 j' 2}\in \Sigma$ (voir la Proposition \ref{pr:GDEN-prim-ex-A4}). On rappelle que   $\sigma_{2 j' 2}$ est le c\^one r\'egulier  engendr\'e par les vecteurs $(1,0,0,0), (0,0,0,1)$ ,  $\rho_{j+1}=(p-j-1,p-j-1,1,1)$ et $\rho_{j}=(p-j,p-j,1,1)$. On note $U$ ($\cong \AF^4_{\KK}$) l'ouvert torique qui est en correspondance avec le c\^one $\sigma_{2 j'2}$.\\

La restriction \`a $U$ du morphisme torique  $\pi:X(\Sigma)\rightarrow \AF^{4}_{\KK}$ est d\'efinie de la fa\c con suivante:
\begin{center}
 $\pi\mid_U:U\rightarrow \AF^{4}_{\KK}$, $(y_1,y_2,y_3,y_4)\mapsto (x_1,x_2,x_3,x_4):=(y_1y_2^{p-j}y_3^{p-j-1},y_2^{p-j}y_3^{p-j-1},y_2y_3,y_2y_3y_4)$.
\end{center}

On remarque que $U\cap\pi^{-1}(0)=\{y_2=0\}\cup\{y_3=0\}$. La vari\'et\'e $\widetilde{V}\cap U$ est donn\'e par l'\'equation suivante:
\begin{center}
$\h_{q}(y_1,1)+y_2^{qj}y_3^{q(j+1)}\hk_{pq}(1,y_4)=0$.
\end{center}

Si $j=0$,  le diviseur $\E\cap U$ est donn\'e par les \'equations suivantes:

\begin{center}
$\h_{q}(y_1,1)+y_3^{q}\hk_{pq}(1,y_4)=0$, $y_2=0$.
\end{center}

Si $1\leq j\leq p-2$, le diviseur $\E\cap U$, $1\leq i\leq q$, est une composante irr\'eductible de 

\begin{center}
$\h_{q}(y_1,1)=0$, $y_2=0$.
\end{center}

 Soit $w_i\in \KK^{\star}$ la racine de $\h_{q}(y,1)$,  tel que le diviseur  $\E\cap U$ est donn\'e par les \'equations
\begin{center}
$y_1=w_i$,  $y_2=0$.
\end{center}

Comme le $K$-arc $\alpha$ n'est pas concentr\'e  en un hyperplan $x_i$, $1\leq i\leq 4$, et $\widehat{\alpha}(0)\in U\cap \widetilde{V}$, le $K$-arc $\widehat{\alpha}:\spec K[[t]]\rightarrow U\cap \widetilde{V}$ est un morphisme bien d\'efini. On peut donc \'ecrire le comorphisme $\com{\widehat{\alpha}}$ de la fa\c con suivante:
\begin{center}
$\com{\widehat{\alpha}}(y_i)=t^{r_{i}}\widehat{\alpha}_i$, $1\leq i\leq 4$,
\end{center}
o\`u  les  $\widehat{\alpha}_i$ sont des s\'eries formelles  inversibles dans $K[[t]]$ et les  $r_i$ sont positifs, pour tout $1\leq i\leq 4$.\\

Soit $\widehat{a}_i \in K^{\star}$ le terme constant  de $\widehat{\alpha}_i$, $1\leq i\leq 4$.
Si $j=1$ (resp. $j\geq 2$), alors le $K$-arc $\widehat{\alpha}$ est transverse au diviseur $\E_0$ (resp. $\E_{i,j})$) et  le point $\widehat{\alpha}(0)$ est le point g\'en\'erique de $\E_0$ (resp. $\E_{i,j}$). Par cons\'equent, on  a $(r_1,r_2,r_3,r_4)=(0,1,0,0)$ et $\hk_{pq}(1,\widehat{a}_4)\neq 0$.\\

En utilisant le morphisme $\pi$, on obtient que:
 
\begin{center}

$\alpha_3=\widehat{\alpha}_2\widehat{\alpha}_3$ et $\alpha_4=\widehat{\alpha}_2\widehat{\alpha}_3\widehat{\alpha}_4$.
\end{center}
Par cons\'equent, on a 
$\hk_{pq}(a_3,a_4)=(\widehat{a}_2\widehat{a}_3)^{pq}\hk_{pq}(1,\widehat{a}_4)\neq 0$, d'o\`u le lemme.\end{proof}
Avec la d\'emonstration du Lemme \ref{le:ser_for_inv_prem_cas_A4}, on ach\`eve la preuve de la Proposition \ref{pr:ser_for_inv_prem_cas_A4}. \end{proof}
 
\subsection[La deuxi\`eme famille d'exemples]{La deuxi\`eme famille d'exemples} 
\label{se:DFEA4}
On utilise les  notations de la section \ref{se:PFEA4}. 
Soient un entier  $q\geq 3$  et $V$ une hypersurface de $\AF_{\KK}^4$ donn\'ee par une \'equation du type:

\begin{center}
 $\f(x_1,x_2,x_3,x_4):=\h_q(x_1,x_2)+\hk_{q}(x_3,x_4^2)$,
\end{center}
o\`u   $\h_q$ et $\hk_{q}$ sont deux polyn\^omes homog\`enes  de  degr\'e $q$ sans facteur multiple. De plus, 
on suppose que le polyn\^ome $\f$ n'est pas d\'eg\'en\'er\'e par rapport \`a la fronti\`ere de Newton $\Gamma(\f)$ 
et que $x_3$ et $x_4$ ne divisent pas $\hk_{q}(x_3,x_4^2)$.

\begin{example} Le polyn\^ome $\f=x_1^q+x_2^q+x_3^{q}+x_4^{2q}$, $q\geq 3$,  satisfait les hypoth\`eses ci-dessus.  
Dans ce cas la vari\'et\'e $V$ est une hypersurface de Pham-Brieskorn.
\end{example}
   
\begin{remarque}
\label{re:hq_no_div}
\`A un automorphisme  lin\'eaire  de  $\AF_{\KK}^2$ pr\'es,  $x_1$ et  $x_2$  ne divisent pas $\h_q(x_1,x_2)$.\\
\end{remarque}

Le r\'esultat principal de cette section est le th\'eor\`eme suivant:

\begin{theoreme}
\label{th:deux-the-hyp-A4}
 L'application de Nash $\mathcal{N}_V$ associ\'ee \`a l'hypersurface $V$ est bijective et le nombre de diviseurs essentiels sur $V$ est \'egal \`a $2$.
\end{theoreme}
Le r\'esultat suivant est une cons\'equence direct des Th\'eor\`emes  \ref{th:premier-the-hyp-A4}, \ref{th:def-appl-nash} et la
Proposition \ref{pr:des_prim_famil_hyp_A4}.
Soit $\h\in \KK[x_1,x_2,x_3,x_4]$ tel que les exposants des mon\^omes de $\h$ appartiennent \`a l'ensemble $\Gamma_{+}(\f)-\Gamma(\f)$.
On note $V'$ l'hypersurface  de  $\AF_{\KK}^4$ donn\'ee par l'\'equation $\f(x_1,x_2,x_3,x_4)+\h(x_1,x_2,x_3,x_4)=0$.

\begin{corollaire}
 L'application de Nash $\mathcal{N}_{V'}$ associ\'ee \`a l'hypersurface $V'$ 
 est bijective et le nombre de diviseurs essentiels sur $V'$ est \'egal \`a $2$.
\end{corollaire}

D'abord, on utilise la g\'eom\'etrie torique pour r\'esoudre la singularit\'e de $V$. Ensuite, 
on d\'emontre le Th\'eor\`eme \ref{th:deux-the-hyp-A4}. La proposition suivante r\'esulte d'un calcul direct.
\begin{proposition}
\label{pr:deux-pr-sing-iso}
 L'hypersurface $V$ est normale et le point $0$ est son unique point singulier. De plus, $V$ ne contient aucune $T$-orbite de dimension strictement positive.
\end{proposition}

 D'apr\`es les r\'esultats \ref{pr:cr_nor-f} et
\ref{pr:deux-pr-sing-iso}, si $\Sigma$ est une subdivision r\'eguli\`ere admissible de $\EN{\f}$, alors le
morphisme torique $\pi:X(\Sigma)\rightarrow \AF_{\KK}^4$ associ\'e \`a la subdivision r\'eguli\`ere admissible
$\Sigma$ de $\EN{\f}$ est une r\'esolution plong\'ee de l'hypersurface $V$.\\

Maintenant, on montre explicitement une subdivision r\'eguli\`ere admissible de $\EN{\f}$.\\

La Figure \ref{fi:poly_newt_deux_fami_A4} repr\'esente la face compacte du poly\`edre de Newton $\Gamma_{+}(\f)$,
c'est-\`a-dire la fronti\`ere de Newton $\Gamma(\f)$. On rappelle que $\f=\h_q(x_1,x_2)+\hk_{q}(x_3,x_4^2)$ et que $x_1$ et  $x_2$  (resp. $x_3$, $x_4$) ne divisent pas $\h_q(x_1,x_2)$ (resp. $\hk_{q}(x_3,x^2_4)$), voir la Remarque \ref{re:hq_no_div}.

\begin{figure}[!h]

\begin{tikzpicture}[x=0.8cm,y=0.8cm ]
\tikzfading[name=fade1,left color=transparent!80,right color=transparent!80, inner color=transparent!80,outer color=transparent!60]
 \tikzfading[name=fade2,left color=transparent!10,right color=transparent!80, inner color=transparent!50,outer color=transparent!30]

   \draw[thick] (180:3)  -- (270:3) -- (0:3)  -- (90:3) -- (180:3);
   \draw[very thick] (90:3) --  (270:3);

\shade[ball color=white, path fading=fade1] (180:3) -- (90:3) -- (270:3);
\shade[ball color=gray, path fading=fade2] (90:3) -- (270:3)-- (0:3);
\foreach \i in {0,0.2,...,5.9}
\draw[gray] (-3+\i,0)--(-3+\i+0.1,0);

\draw (0:4.2) node[inner sep=-1pt,below=-1pt,rectangle,fill=white] {{\small $(0,0,q,0)$}};
\draw (180:4.2) node[inner sep=-1pt,below=-1pt,rectangle,fill=white] {{\small $(q,0,0,0)$}};
\draw (90:3.4) node[inner sep=-1pt,below=-1pt,rectangle,fill=white] {{\small $(0,0,0,2q)$}};
\draw (270:3.4) node[inner sep=-1pt,below=-1pt,rectangle,fill=white] {{\small $(0,q, 0,0)$}};

    \draw[fill=black]  (0:3) circle(0.7mm) (180:3) circle(0.7mm)  (90:3) circle(0.7mm)  (270:3) circle(0.7mm);

  \end{tikzpicture}

\caption{La fronti\`ere de Newton $\Gamma(\f)$, o\`u $\f=\h_q(x_1,x_2)+\hk_{q}(x_3,x_4^2)$.}
\label{fi:poly_newt_deux_fami_A4}
\end{figure}

Soit $H$ un hyperplan qui ne contient pas 
l'origine de $\RR^4$ et tel que $H\cap \RR^4_{\geq 0}$ soit un ensemble compact.
La Figure \ref{fi:EN_newt_deux_fami_A4}  repr\'esente l'intersection de $H$ avec la 
subdivision $\EN{\f}$ de $\RR^4_{\geq 0}$. Chaque sommet du diagramme est identifi\'e avec 
le  vecteur extr\'emal  correspondant.

\begin{figure}[!h]

\begin{tikzpicture}[x=0.8cm,y=0.8cm ]

   \draw[very thick] (180:3)  -- (270:3) -- (0:3)  -- (90:3) -- (180:3);
   \draw[very thick] (90:3) --  (270:3);

\foreach \i in {0,0.2,...,5.9}
\draw[gray] (-3+\i,0)--(-3+\i+0.1,0);

\foreach \i in {0,0.2,...,2.28}
\draw[thick] (-3+\i,0.31*\i)--(-3+\i+0.1,0.31*\i+0.032);

\foreach \i in {0,0.05,...,0.71}
\draw[thick] (0-\i,-3+5.22*\i)--(-\i-0.025,-3+5.22*\i+5.22*0.025);

\foreach \i in {0,0.05,...,0.70}
\draw[thick] (0-\i, 3-3.22*\i)--(-\i-0.025,3-3.22*\i-3.22*0.025);

\foreach \i in {0,0.2,...,3.70}
\draw[thick] (3-\i, 0.19*\i)--(3-\i-0.1,0.19*\i +0.19*0.1);

\draw (0:4.7) node[inner sep=-1pt,below=-1pt,rectangle,fill=white] {{\small $\e_3:=(0,0,1,0)$}};
\draw (180:4.7) node[inner sep=-1pt,below=-1pt,rectangle,fill=white] {{\small $\e_1:=(1,0,0,0)$}};
\draw (90:3.4) node[inner sep=-1pt,below=-1pt,rectangle,fill=white] {{\small $\e_4:=(0,0,0,1)$}};
\draw (270:3.4) node[inner sep=-1pt,below=-1pt,rectangle,fill=white] {{\small $\e_2:=(0,1, 0,0)$}};
\draw (-1,1) node[inner sep=-1pt,below=-1pt,rectangle,fill=white] {{\small $\rho_0$}};

\draw (4.5,1.5) node[inner sep=-1pt,below=-1pt,rectangle,fill=white] {{\small $\rho_0:=(2,2,2,1)$}};

    \draw[fill=black]  (0:3) circle(0.7mm) (180:3) circle(0.7mm)  (90:3) circle(0.7mm)  (270:3) circle(0.7mm) (-0.71,0.71) circle(0.7mm);

  \end{tikzpicture}
\caption{Section de l'\'eventail  de Newton $\EN{\f}$, o\`u $\f=\h_q(x_1,x_2)+\hk_{q}(x_3,x_4^2)$}
\label{fi:EN_newt_deux_fami_A4}
\end{figure}
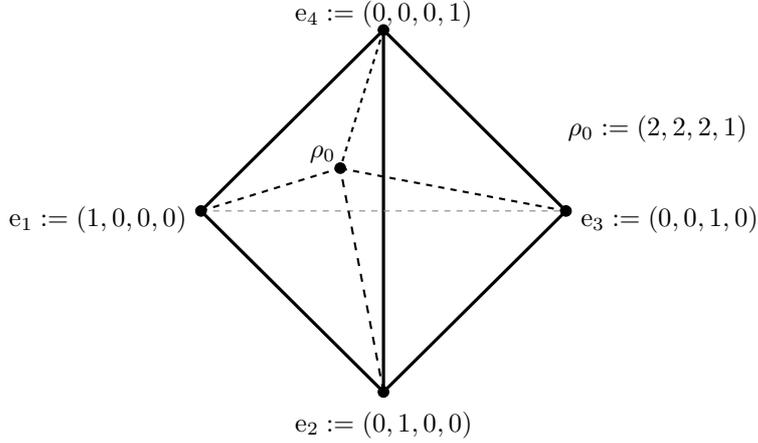

On note $\sigma_j$, $1\leq j\leq 4$, le c\^one de dimension $4$ de $\EN{\f}$ engendr\'e par le vecteur $(2,2,2,1)$ et l'ensemble 
$\{\e_i\mid 1\leq i\leq 4,\;i\neq j\}$.\\

La proposition suivante r\'esulte d'un simple calcul.

\begin{proposition}
 Le c\^one  $\sigma_4$ est r\'egulier.
\end{proposition}

La Figure \ref{fi:sub_sigma12_deux_fami_A4} repr\'esente l'intersection du plan $H$ avec une  subdivision des c\^ones $\sigma_1$, $\sigma_2$ et $\sigma_3$. 
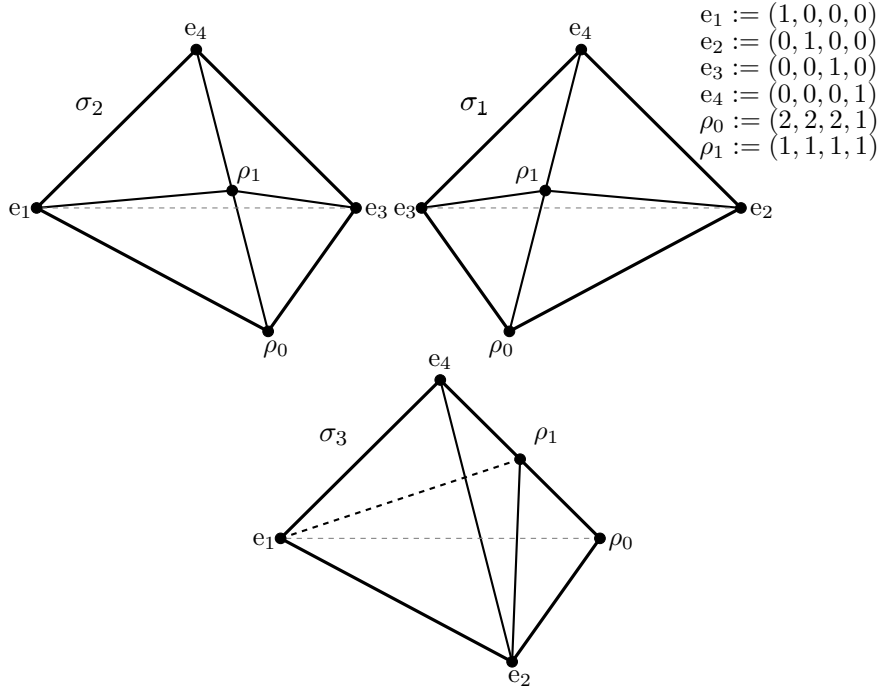
\begin{figure}[!h]
\begin{tikzpicture}[x=0.7cm,y=0.7cm ]

   \draw[very thick] (180:3)  -- (300:2.7) -- (0:3)  -- (90:3) -- (180:3);
   \draw[thick] (90:3) --  (300:2.7);
   \draw[thick] (180:3) --  (0.675,0.3296);
 \draw[thick] (0:3) --  (0.675,0.3296);

\foreach \i in {0,0.2,...,5.9}
\draw[gray] (3-\i,0)--(3-\i-0.1,0);


\draw (0:3.4) node[inner sep=-1pt,below=-1pt,rectangle,fill=white] {{\small $\e_3$}};
\draw (-3.3,0) node[inner sep=-1pt,below=-1pt,rectangle,fill=white] {{\small $\e_1$}};
\draw (90:3.4) node[inner sep=-1pt,below=-1pt,rectangle,fill=white] {{\small $\e_4$}};;
\draw (300:3) node[inner sep=-1pt,below=-1pt,rectangle,fill=white] {{\small $\rho_0$}};
\draw (1,0.7) node[inner sep=-1pt,below=-1pt,rectangle,fill=white] {{\small $\rho_1$}};

\draw (-2,2) node[inner sep=-1pt,below=-1pt,rectangle,fill=white] { $\sigma_2$};
    \draw[fill=black](0:3) circle(0.7mm) (180:3) circle(0.7mm)  (90:3) circle(0.7mm)  (300:2.7) circle(0.7mm) (0.675,0.3296) circle(0.7mm);

  \end{tikzpicture}
\begin{tikzpicture}[x=0.7cm,y=0.7cm ]

   \draw[very thick] (180:3)  -- (240:2.7) -- (0:3)  -- (90:3) -- (180:3);
   \draw[thick] (90:3) --  (240:2.7);
   \draw[ thick] (180:3) --  (-0.675,0.3296);
 \draw[thick] (0:3) --  (-0.675,0.3296);

\foreach \i in {0,0.2,...,5.9}
\draw[gray] (3-\i,0)--(3-\i-0.1,0);


\draw (0:3.4) node[inner sep=-1pt,below=-1pt,rectangle,fill=white] {{\small $\e_2$}};
\draw (-3.3,0) node[inner sep=-1pt,below=-1pt,rectangle,fill=white] {{\small $\e_3$}};
\draw (90:3.4) node[inner sep=-1pt,below=-1pt,rectangle,fill=white] {{\small $\e_4$}};;
\draw (240:3) node[inner sep=-1pt,below=-1pt,rectangle,fill=white] {{\small $\rho_0$}};
\draw (-1,0.7) node[inner sep=-1pt,below=-1pt,rectangle,fill=white] {{\small $\rho_1$}};

\draw (-2,2) node[inner sep=-1pt,below=-1pt,rectangle,fill=white] { $\sigma_2$};
    \draw[fill=black](0:3) circle(0.7mm) (180:3) circle(0.7mm)  (90:3) circle(0.7mm)  (240:2.7) circle(0.7mm) (-0.675,0.3296) circle(0.7mm);

\draw (-2,2) node[inner sep=-1pt,below=-1pt,rectangle,fill=white] {$\sigma_1$};

\draw (3.5+0.4,3+0.8) node[inner sep=-1pt,below=-1pt,rectangle,fill=white] {{\small $\e_1:=(1,0,0,0)$}};
\draw (3.5+0.4,2.5+0.8) node[inner sep=-1pt,below=-1pt,rectangle,fill=white] {{\small $\e_2:=(0,1,0,0)$}};
\draw (3.5+0.4,2+0.8) node[inner sep=-1pt,below=-1pt,rectangle,fill=white] {{\small $\e_3:=(0,0,1,0)$}};
\draw (3.5+0.4,1.5+0.8) node[inner sep=-1pt,below=-1pt,rectangle,fill=white] {{\small $\e_4:=(0,0,0,1)$}};

\draw (3.5+0.4,1+0.8) node[inner sep=-1pt,below=-1pt,rectangle,fill=white] {{\small $\rho_0:=(2,2,2,1)$}};
\draw (3.5+0.4,0.5+0.8) node[inner sep=-1pt,below=-1pt,rectangle,fill=white] {{\small $\rho_1:=(1,1,1,1)$}};
\end{tikzpicture}

\begin{tikzpicture}[x=0.7cm,y=0.7cm ]
   \draw[very thick] (180:3)  -- (300:2.7) -- (0:3)  -- (90:3) -- (180:3);
   \draw[thick] (300:2.7) --  (1.5,1.5);
\draw[thick] (300:2.7) --  (0,3);

\foreach \i in {0,0.2,...,4.4}
\draw[thick] (-3+\i,\i/3)--(-3+\i+0.1,\i/3+0.1/3);

\foreach \i in {0,0.2,...,5.9}
\draw[gray] (3-\i,0)--(3-\i-0.1,0);


\draw (0:3.4) node[inner sep=-1pt,below=-1pt,rectangle,fill=white] {{\small $\rho_0$}};
\draw (-3.3,0) node[inner sep=-1pt,below=-1pt,rectangle,fill=white] {{\small $\e_1$}};
\draw (90:3.4) node[inner sep=-1pt,below=-1pt,rectangle,fill=white] {{\small $\e_4$}};;
\draw (300:3) node[inner sep=-1pt,below=-1pt,rectangle,fill=white] {{\small $\e_2$}};
\draw (2,2) node[inner sep=-1pt,below=-1pt,rectangle,fill=white] {{\small $\rho_1$}};

\draw (-2,2) node[inner sep=-1pt,below=-1pt,rectangle,fill=white] { $\sigma_3$};
    \draw[fill=black](0:3) circle(0.7mm) (180:3) circle(0.7mm)  (90:3) circle(0.7mm)  (300:2.7) circle(0.7mm) (1.5,1.5) circle(0.7mm);
\end{tikzpicture}
\caption{Subdivision des c\^ones $\sigma_1$, $\sigma_2$ et $\sigma_3$}
\label{fi:sub_sigma12_deux_fami_A4}
\end{figure}
Pour chaque entier $1\leq j\leq 3$,  on note $\sigma_{j 1}$ (resp. $\sigma_{j 2}$)  le c\^one de dimension $4$ engendr\'e par le vecteur  $\rho_1$ et l'ensemble $\{\rho_0\}\cup\{\e_i\mid 1\leq i\leq 3,\;  i\neq j\}$ (resp. $\{\e_i\mid  i\neq j\}$).\\

On remarque que l'ensemble $\{ \sigma_{j k}\mid 1\leq j\leq 3,1\leq k\leq 2\}\cup \{\sigma_4\}$ engendre un \'eventail, not\'e  $\Sigma$.\\  

Soient $\Bo:=\{e_1,\e_3,\e_3,\e_4\}$ la base ordonn\'ee canonique de $\N$ et  $\Delta$ le c\^one standard  
engendr\'e par $\Bo$. Maintenant, on consid\`ere la subdivision du c\^one  $\Delta$ suivante:\\

 Soit $\Sigma_{1}$  la subdivision \'el\'ementaire de $\Delta$ centr\'ee en $u=\sum_{i=1}^{4}e_i$. Pour chaque entier $j$, $1\leq j\leq 4$, on note  $\Bo_j$ la base ordonn\'ee de $\N$
obtenue en rempla\c cant $\e_j$ par $u$ dans la base $\Bo$. Soient $\Delta_j:=\langle \Bo_j\rangle$ 
le c\^one r\'egulier engendr\'e par $\Bo_j$, pour $1\leq j \leq 4$, et  $\Sigma_{2}$ 
la subdivision de  $\Sigma_{1}$ obtenue  en rempla\c cant $\Delta_{4}$ en $\Sigma_{1}$ 
par les c\^ones   $\Delta_{4j}:=\langle \Bo_{4j}\rangle$, o\`u $\Bo_{4j}$ est la 
base ordonn\'ee de $\N$ obtenue en   rempla\c cant le $j$-i\`eme vecteur de $\Bo_{4}$ par $\sum_{u\in \Bo_{4}}u$.

On remarque que la suite de morphismes $X(\Sigma_2)\rightarrow X(\Sigma_1)\rightarrow \AF^{4}_{\KK}$ est obtenue en \'eclatant deux points ferm\'es. \\

Par une simple inspection des c\^ones des  \'eventails $\Sigma$ et $\Sigma_{2}$, 
on obtient la proposition suivante:

\begin{proposition}
 \label{pr:GDEN-deux-ex-A4} 

L'\'eventail $\Sigma$ engendr\'e par l'ensemble   $\{ \sigma_{j k}\mid 1\leq j\leq 3,1\leq k\leq 2\}\cup \{\sigma_4\}$ est une $G$-subdivision r\'eguli\`ere de l'\'eventail $\EN{\f}$. De plus, $\Sigma$ et $\Sigma_2$ sont les m\^emes \'eventails.
\end{proposition}

Soit $\pi: X(\Sigma)\rightarrow \AF^4_{\KK}$ le morphisme torique induit par l'\'eventail $\Sigma$ de la 
Proposition \ref{pr:GDEN-deux-ex-A4}. On remarque que  le morphisme   $\pi$ est une r\'esolution plong\'ee de $V$. Soit  $\widetilde{V}$ le transform\'e strict de $V$ dans $X(\Sigma)$. Par abus de notation, on note $\pi:\widetilde{V}\rightarrow V$ la restriction de $\pi$ \`a $\widetilde{V}$.

On note $\Sigma_0$ l'\'eventail engendr\'e par le c\^one $\Delta$ et $\Bl_i$ le diviseur exceptionnel de l'\'eclatement 
$X(\Sigma_i)\rightarrow X(\Sigma_{i-1})$, $1\leq i\leq 2$. Par abus de notation, on note $\Bl_1$ le transform\'e strict de 
$\Bl_1$ dans $X(\Sigma_2)$.  On pose $\E_i:=\Bl_i\cap \widetilde{V}$, $1\leq i\leq 2$.\\
 
Par abus de notation, on note $0$ l'origine de $\AF_{\KK}^4$. 
Pour d\'emontrer la proposition suivante,  on peut appliquer 
la m\^eme m\'ethode que celle utilis\'ee pour la d\'emonstration  de la Proposition  \ref{pr:des_prim_famil_hyp_A4}.

\begin{proposition}
\label{pr:des_deux_famil_hyp_A4}
Les $\E_i$ sont des vari\'et\'es lisses et irr\'eductibles et l'intersection $\E_1\cap\E_2$ est une vari\'et\'e lisse et irr\'eductible.
De plus, la  fibre exceptionnelle  $\pi^{-1}(0)$ de la d\'esingularisation 
$\pi:\widetilde{V}\rightarrow V$ est  la r\'eunion de $\E_1$ et $\E_2$. 
En particulier $\pi^{-1}(0)$ est  un diviseur \`a croisements normaux.

\end{proposition}
\begin{remarque}
 Dans les r\'esultats suivants, on montre que les diviseurs essentiels sur $V$ sont exactement  les diviseurs exceptionnels de la d\'esingularisation $\pi:\widetilde{V}\rightarrow V$ 
\end{remarque}

\subsubsection{\bf{Preuve de la bijectivit\'e de l'application de Nash}}
Dans cette section, on d\'emontre le Th\'eor\`eme \ref{th:deux-the-hyp-A4} sur la  bijectivit\'e de
l'application de Nash pour  l'hypersurface $V$ de $\AF_{\KK}^4$ donn\'ee par une \'equation du type 
$\f(x_1,x_2,x_3,x_4):=\h_q(x_1,x_2)+\hk_{q}(x_3,x_4^2)$, ce qui \'equivaut \`a montrer que tous les 
wedges admissibles se rel\`event \`a une d\'esingularisation de $V$. 
Notre but, dans toute la suite de cette section, est de montrer que pour chaque diviseur essentiel $\E$  tous les $K$-wedges admissibles centr\'es en 
$\N(\E)$  se rel\`event \`a la d\'esingularisation  $\widetilde{V}$ de 
la Proposition \ref{pr:des_deux_famil_hyp_A4}, o\`u $K$ est une extension du corps  $\KK$.\\

D'abord, on donne quelques r\'esultats techniques. Ensuite, on d\'emontre le Th\'eor\`eme \ref{th:deux-the-hyp-A4}.\\

On note $0$ (resp $g$) le point ferm\'e (resp. g\'en\'erique) de $\spec K[[t]]$.
\'Etant donn\'e un $K$-arc  $\alpha:\spec K[[t]]\rightarrow V$, $\alpha(0)=0$, 
qui n'est pas concentr\'e en un hyperplan $x_i=0$, $1\leq i\leq 4$,  on note  $\mu=(\mu_1,\mu_2,\mu_3,\mu_4)\in\ZZ_{> 0}^{4}$ le {\it vecteur principal} du $K$-arc $\alpha$, c'est-\`a-dire 

\begin{center}
$\mu:=(\ord_t\alpha^{\star }(x_1), \ord_t\alpha^{\star }(x_2),\ord_t\alpha^{\star }(x_3),\ord_t\alpha^{\star }(x_4))$.
\end{center}
o\`u  $\com{\alpha}$ est le comorphisme du $K$-arc $\alpha$. On peut donc \'ecrire le comorphisme $\com{\alpha}$ de la fa\c con suivante:
\begin{center}
$\com{\alpha}(x_i)=t^{\mu_{i}}\alpha_i$, $1\leq i\leq 4$,
\end{center}
o\`u  les $\alpha_i$ sont des s\'eries formelles  inversibles dans $K[[t]]$.\\

Dans les r\'esultats suivants, on utilise la notation  ci-dessus et celle de la Proposition \ref{pr:des_deux_famil_hyp_A4}.\\

On remarque que si un $K$-arc $\alpha:\spec K[[t]]\rightarrow V$ n'est pas concentr\'e en $0\in V$, alors $\alpha$ 
se rel\`eve \`a $\widetilde{V}$, car le morphisme $\pi$ est une d\'esingularisation de $V$. De plus, si le point 
$\widetilde{\alpha}(0)$  est le point g\'en\'erique d'une composante irr\'eductible de la fibre exceptionnelle de 
$\pi$, alors le $K$-arc $\alpha$ n'est pas concentr\'e en une hypersurface de $V$.\\

Pour d\'emontrer la proposition suivante,  on peut appliquer 
la m\^eme m\'ethode que celle de la d\'emonstration  de la Proposition  
\ref{pr:vect-prin}.

\begin{proposition} 
\label{pr:vect-prin-deux} Soient  $\pi:\widetilde{V}\rightarrow V$ la d\'esingularisation de la Proposition \ref{pr:des_deux_famil_hyp_A4}, $\E$  une composante irr\'eductible de la fibre exceptionnelle de $\pi$ et $\alpha$  un $K$-arc qui n'est pas concentr\'e en $0$   tel que  le rel\`evement $\widehat{\alpha}$ de $\alpha$ \`a $\widetilde{V}$ est transverse \`a $\E$. De plus, on suppose que   $\widehat{\alpha}(0)$ est le point g\'en\'erique de $\E$. Alors,  si  $\E$ est le diviseur $\E_1$ (resp. $\E_{2}$), alors le vecteur principal $\mu$ du $K$-arc $\alpha$ est le vecteur $\rho_1=(1,1,1,1)$ (resp. $\rho_0=(2,2,2,1)$).

\end{proposition}

\'Etant fix\'e une composante irr\'eductible $\E$ de la fibre exceptionnelle du morphisme $\pi:\widetilde{V}\rightarrow V$ ($\pi$ est la d\'esingularisation de la Proposition \ref{pr:des_deux_famil_hyp_A4}), on consid\`ere un $K$-wedge $\omega:\spec K[[s,t]]\rightarrow V$,  tel que son centre (resp. son arc g\'en\'erique)  est le point g\'en\'erique de 
$\N(\E)$, not\'e $\alpha_{\E}$, (resp. est un point qui appartient \`a $V_{\infty}^{s}$), et on pose \\

\begin{center}
$(\eta_1,\eta_2,\eta_3,\eta_4):=(\ord_t\com{\omega}(x_1),\ord_t\com{\omega}(x_2),\ord_t\com{\omega}(x_3),\ord_t\com{\omega}(x_{4}))\in \ZZ^{4}_{>0}$, 
\end{center}
o\`u $\com{\omega}$ est le comorphisme du $K$-wedge $\omega$.
On peut donc \'ecrire le comorphisme $\com{\omega}$ de la fa\c con suivante:
\begin{center}
$\com{\omega}(x_i)=t^{\eta_{i}}\varphi_i$, $1\leq i\leq 4$,
\end{center}
o\`u  les  $\varphi_i$ sont des s\'eries formelles  dans $K[[s,t]]$ qui ne sont pas divisibles par $t$.\\

On pose 
\begin{center}
 $(\mu_1,\mu_2,\mu_3,\mu_4):=(\ord_t\com{\alpha_{\E}}(x_1),\ord_t\com{\alpha_{\E}}(x_2),\ord_t\com{\alpha_{\E}}(x_3),\ord_t\com{\alpha_{\E}}(x_{4}))$,
\end{center}
le vecteur principal du $\KK_{\alpha_{\E}}$-arc $\alpha_{\E}$, o\`u $\KK_{\alpha_{\E}}$ est le corps r\'esiduel du point g\'en\'erique de
$\N(\E)$.\\

 D'apr\`es la Proposition \ref{pr:vect-prin-deux}, si $\E$ est le diviseur $\E_1$ (resp. $\E_2$), alors le vecteur 
principal du $\KK_{\alpha_{\E}}$-arc $\alpha_{\E}$ est le vecteur $(\mu_1,\mu_2,\mu_3,\mu_4)=(1,1,1,1)$ 
(resp. $(\mu_1,\mu_2,\mu_3,\mu_4)=(2,2,2,1)$). On remarque que  $1\leq \eta_i\leq \mu_i\leq 2$, 
pour tout $1\leq i\leq 4$, en particulier $\eta_4=\mu_4=1$.\\

Soient  $\phi\in K[[s,t]]$  une s\'erie formelle non nulle et $v\in \RR^{2}_{>0}$. 
Le vecteur $v$ induit une graduation positive sur l'anneau $K[[s,t]]$, 
on note  $\nu_v\phi$ (resp.  $\phi_v$)  le $v$-ordre (resp. la $v$-partie principale) de $\phi$ .\\

 Le lemme suivant est un r\'esultat technique qu'on utilise dans la preuve du Th\'eor\`eme \ref{th:deux-the-hyp-A4}.  
 
\begin{lemme}
\label{le:v-ord-hq_h'q}
 Si le $K$-wedge $\omega$ est centr\'e en $\N(\E_2)$, alors il existe un vecteur $v=(u,1)\in \QQ^{2}_{>0}$ tel que

\begin{center}
 $\nu_v\h_q(t^{\eta_1}\varphi_1,t^{\eta_2}\varphi_2)=\degt{\h_q(t^{\eta_1}(\varphi_1)_v,t^{\eta_2}(\varphi_2)_v)}=2q$,
\vspace{0.5cm}

$\nu_v\hk_q(t^{\eta_3}\varphi_3,t^{2}\varphi_4^2)=\degt{\hk_q(t^{\eta_3}(\varphi_3)_v,t^{2}(\varphi_4^2)_v)}=2q$.
\end{center}

\end{lemme}
\begin{proof}
On rappelle que  $\alpha_{\E_2}$ est le  point g\'en\'erique de $\N(\E)$.

 Soit $\lambda:\spec K[[t]]\rightarrow \spec K[[s,t]]$ le morphisme induit par
 l'homomorphisme canonique $\lambda^{\star}:K[[s,t]]\rightarrow K[[s,t]]\slash (s)=K[[t]]$. 
 Alors, le morphisme $\alpha=\omega\circ \lambda$ est un $K$-arc sur $V$.
 On remarque que le $K$-arc $\alpha$ est un $K$-point de $V_{\infty}$ au-dessus du point  $\alpha_{\E_2}$.

En utilisant la propri\'et\'e fonctorielle de l'espace d'arcs $V_{\infty}$,
on obtient que  le rel\`evement $\widehat{\alpha}$ de $\alpha$ \`a $\widetilde{V}$ est transverse au diviseur  $\E_2$ et  que   $\widehat{\alpha}(0)$ est le point g\'en\'erique de $\E_2$, car le centre du $K$-wedge $\omega$ est le point g\'en\'erique de 
$\N(\E_2)$.\\

D'apr\`es la Proposition \ref{pr:vect-prin-deux}, on peut \'ecrire le comorphisme $\com{\alpha}$ de $\alpha$ de la fa\c con suivante:
\begin{center}
$\com{\alpha}(x_i)=t^{\mu_{i}}\alpha_i$, $1\leq i\leq 4$, $(\mu_1,\mu_2,\mu_3,\mu_4)=(2,2,2,1)$,
\end{center}
o\`u  les $\alpha_i$ sont des s\'eries formelles  inversibles dans $K[[t]]$. On remarque qu'on a:

\begin{center}
 $t^{\mu_i-\eta_i}\alpha_i=\lambda^{\star}\circ \varphi_i$, $1\leq i\leq 3$, et $\alpha_4=\lambda^{\star}\circ \varphi_4$.
\end{center}

 Soit $a_i\in K^{\star}$ le terme constant de la s\'erie formelle $\alpha_i$. Alors, il existe un vecteur $v=(u,1)\in \QQ^2_{>0}$ tel que 
\begin{center}
 $(\varphi_i)_v=a_it^{\mu_i-\eta_i}$, $1\leq i\leq 3$ et $(\varphi_4)_v=a_4$
\end{center}
En effet, il suffit de choisir $v\in \QQ^{2}_{>0}$ tel que le nombre  $u$ soit ``assez grand".\\

On rappelle que $\sigma_{11}\in \Sigma$ est le c\^one  engendr\'e par les vecteurs $\rho_0=(2,2,2,1)$, $\rho_1=(1,1,1,1)$, $\e_2=(0,1,0,0)$ et $\e_3=(0,0,1,0)$ et que la restriction du morphisme $\pi:X(\Sigma)\rightarrow \AF^4_{\KK}$ \`a l'ouvert $U:=U_{\sigma_{11}}$ est donn\'ee de la fa\c con suivante:

\begin{center}
 $\pi:U\rightarrow \AF^{4}_{\KK}$, $(y_1,y_2,y_3,y_4)\mapsto (y_1^2y_4,y_1^2y_2y_4,y_1^2y_3y_4,y_1y_4)$
\end{center}
 
 L'intersection de $\widetilde{V}$ (resp. $\E_2$) et $U$  est donn\'ee par l'\'equation suivante (resp. les \'equations suivantes):
\begin{center}
 $\h_q(1,y_2)+\hk_{q}(y_3,y_4)=0$  (resp. $\h_q(1,y_2)+\hk_{q}(y_3,y_4)=0$  et $y_1=0$)
\end{center}

Comme le $K$-arc $\alpha$ n'est pas concentr\'e en un hyperplan $x_i=0$, $1\leq i\leq 4$, et $\widehat{\alpha}(0)\in U\cap \widetilde{V}$, $\widehat{\alpha}: \spec K[[t]]\rightarrow U\cap \widetilde{V}$ est un morphisme bien d\'efini. On peut donc \'ecrire le comorphisme $\com{\widehat{\alpha}}$ de la fa\c con suivante:
\begin{center}
$\com{\widehat{\alpha}}(y_i)=t^{r_{i}}\widehat{\alpha}_i$, $1\leq i\leq 4$,
\end{center}
o\`u  les  $\widehat{\alpha}_i$ sont des s\'eries formelles  inversibles dans $K[[t]]$ et les  $r_i$ sont positifs, pour tout $1\leq i\leq 4$.\\

Soit $\widehat{a}_i \in K^{\star}$ le terme constant  de $\widehat{\alpha}_i$, $1\leq i\leq 4$. On rappelle que le $K$-arc $\widehat{\alpha}$ est transverse au diviseur $\E_2$  et  que $\widehat{\alpha}(0)$ est le point g\'en\'erique de $\E_2$.  Par cons\'equent, on obtient que $(r_1,r_2,r_3,r_4)=(1,0,0,0)$, car le $K$-arc $\widehat{\alpha}$ est transverse au diviseur $\E_2$,   et  que $\h_{q}(1,\widehat{a}_2)\neq 0$ et $\hk_{q}(\widehat{a}_3,\widehat{a}^2_4)\neq 0$, car le point $(0,\widehat{a}_2,\widehat{a}_3,\widehat{a}_4)$ est le point g\'en\'erique de $\E_2$.\\

En utilisant le morphisme $\pi$, on obtient que:
 \begin{center}
$\begin{array}{ccccccc}
\alpha_1 & = &\widehat{\alpha}_1^2\widehat{\alpha}_4,& \; & \alpha_2& =&\widehat{\alpha}_1^2\widehat{\alpha}_2\widehat{\alpha}_4,\\

\alpha_3 & = & \widehat{\alpha}_1^2\widehat{\alpha}_3\widehat{\alpha}_4,&\;& \alpha_4& =&\widehat{\alpha}_1\widehat{\alpha}_4.
\end{array}$
\end{center}
Par cons\'equent, on a:
\begin{center}
$\h_{q}(a_1,a_2)=(\widehat{a}_1^2\widehat{a_4})^q\h_{q}(1,\widehat{a}_2)\neq 0$,
\vspace{0.5cm}

$\hk_{q}(a_3,a^2_4)=(\widehat{a}_1^2\widehat{a}_4)^q\hk_{q}(\widehat{a}_3,\widehat{a}^2_4)\neq 0$.
\end{center}

Comme  il existe un vecteur $v=(u,1)\in \QQ^2_{>0}$ tel que 
 $(\varphi_i)_v=a_it^{\mu_i-\eta_i}$, $1\leq i\leq 3$ et $(\varphi_4)_v=a_4$, on obtient que
\begin{center}
 $(\h_q(t^{\eta_1}\varphi_1,t^{\eta_2}\varphi_2))_v=\h_q(t^{\eta_1}(\varphi_1)_v,t^{\eta_2}(\varphi_2)_v)=t^{2q}\h_{q}(a_1,a_2)$,
\vspace{0.5cm}

$(\hk_q(t^{\eta_3}\varphi_3,t^{2}\varphi_4^2))_v=\hk_q(t^{\eta_3}(\varphi_3)_v,t^{2}(\varphi_4^2)_v)=t^{2q}\hk_{q}(a_3,a^2_4)$.
\end{center}
Par cons\'equent, on a.
\begin{center}
 $\nu_v\h_q(t^{\eta_1}\varphi_1,t^{\eta_2}\varphi_2)=\degt{\h_q(t^{\eta_1}(\varphi_1)_v,t^{\eta_2}(\varphi_2)_v)}=2q$,
\vspace{0.5cm}

$\nu_v\hk_q(t^{\eta_3}\varphi_3,t^{2}\varphi_4^2)=\degt{\hk_q(t^{\eta_3}(\varphi_3)_v,t^{2}(\varphi_4^2)_v)}=2q$.
\end{center}
Ce qui ach\`eve la preuve du lemme. \end{proof}

La proposition suivante est le r\'esultat cl\'e de la preuve du Th\'eor\`eme \ref{th:deux-the-hyp-A4}
\begin{proposition}
\label{pr:ser_for_inv_deux_cas_A4}
 Les s\'eries formelles $\varphi_i$, $1\leq i\leq 4$,  sont inversibles.
\end{proposition}

D'abord, finissons la preuve du Th\'eor\`eme \ref{th:deux-the-hyp-A4}. 
Comme les s\'eries formelles $\varphi_i$, $1\leq i\leq 4$,  sont inversibles, le $K$-wedge $\omega$ se rel\`eve \`a $\widetilde{V}$. 
On remarque que  la Proposition \ref{pr:ser_for_inv_deux_cas_A4}  est valable pour toute composante irr\'eductible 
$\E$ de la fibre exceptionnelle de la d\'esingularisation $\pi:\widetilde{V}\rightarrow V$  et pour tout 
$K$-wedge $\omega:\spec K[[s,t]]\rightarrow V$,  tel que son centre (resp. son arc g\'en\'erique)  
est le point g\'en\'erique de $\N(\E)$ (resp. est un point qui appartient \`a $V_{\infty}^{s}$). 
Par cons\'equent  on obtient que  tout diviseur exceptionnel $\E$ de la 
d\'esingularisation $\pi:\widetilde{V}\rightarrow V$ est  un  diviseur essentiel sur $V$.

\begin{proof}[{\it D\'emonstration de la Proposition \ref{pr:ser_for_inv_deux_cas_A4}}]

On rappelle que

\begin{center}
 $\FI(\varphi_i)\leq \mu_i-\eta_i$, $1\leq i\leq 4$,
\end{center}
o\`u $\FI(\varphi_i)$ est le nombre de facteurs irr\'eductibles de $\varphi_i$ compt\'es avec multiplicit\'e. De plus,  $\varphi_i$ $1\leq i\leq 4$  est inversible si et seulement si $\mu_i-\eta_i=0$.\\

Si $\E$ est le diviseur $\E_1$, alors le vecteur 
principal du $\KK_{\alpha_{\E}}$-arc $\alpha_{\E}$ est le vecteur $(\mu_1,\mu_2,\mu_3,\mu_4)=(1,1,1,1)$. Par cons\'equent, $\FI(\varphi_i)=0$ (c'est-\`a-dire $\varphi_i$ est inversible)  pour tout  $1\leq i\leq 4$.

Alors on peut supposer que $\E$ est le diviseur $\E_2$. Ainsi on obtient que   $(\mu_1,\mu_2,\mu_3,\mu_4)=(2,2,2,1)$.\\

 Comme   $\eta_4=\mu_4=1$,  la s\'erie  formelle  $\varphi_4$ est inversible.\\

Raisonnons par l'absurde. On suppose que les s\'eries formelles $\varphi_1$, $\varphi_2$ et $\varphi_3$ 
ne sont pas simultan\'ement inversibles. Par cons\'equent, il existe un entier $i_0\in\{1,2,3\}$ tel que $\mu_{i_0}-\eta_{i_0}\neq 0$.\\

Le  $3$-squelette $S_3\EN{\f}$ de l'\'eventail de Newton $\EN{\f}$ est la r\'eunion des  c\^ones de $\EN{\f}$
de dimension inf\'erieure ou \'egale \`a trois 
 (voir la figure \ref{fi:EN_newt_deux_fami_A4}) et on d\'enote $\Delta^0$ l'int\'erieur du c\^one standard. 
Un vecteur $v:=(v_1,v_2,v_3,v_4)\in \Delta^0$ d\'efinit naturellement  une graduation strictement positive sur
l'anneau  de polyn\^omes $\KK[x_1,x_2,x_3,x_3]$. Soit $\f_v$ le terme principal du polyn\^ome 
$\f(x_1,x_2,x_3,x_4)=\h_q(x_1,x_2)+\hk_q(x_3,x_4^2)$ obtenu au moyen de cette graduation.
Par d\'efinition de l'\'eventail de Newton $\EN{\f}$, le vecteur  $v$ appartient \`a  $S_3\EN{\f}\cap \Delta^0$
si et seulement si $\f_v$ n'est pas un mon\^ome (voir la Section \ref{ssec:torique}).  Soit $\beta$ l'arc g\'en\'erique du $K$-wedge $\omega$. 
Comme $\beta^{\star}\f=0$ (on rappelle que $\beta^{\star}$ est le comorphisme 
de $\beta$), on obtient que $\f_{\eta}$, $\eta:=(\eta_1,\eta_2,\eta_3,1)$, n'est pas un m\^onome. Ce qui implique que $(\eta_1,\eta_2,\eta_3,1)\in
S_3\EN{\f}\cap \Delta^0$.

Comme $(\eta_1,\eta_2,\eta_3,1)\in
S_3\EN{\f}\cap \Delta^0$,  $(\mu_1,\mu_2,\mu_3,\mu_4)=(2,2,2,1)$ et $1\leq \eta_i\leq \mu_i$, $1\leq i\leq 4$, 
 les entiers $\eta_1$, $\eta_2$, $\eta_3$ satisfont une des relations suivantes:\\

\begin{itemize}
 \item[1)] $\eta_1=\eta_2=1$ et $\eta_3= 2$;
\item[2)] $\eta_1=\eta_3=1$ et $1\leq \eta_2\leq 2$;
\item[3)] $\eta_2=\eta_3=1$ et $1\leq \eta_1 \leq  2$.\\
\end{itemize}

Le $K$-wedge $\omega$ doit satisfaire l'\'equation  $\h_q(x_1,x_2)+\hk_{q}(x_3,x_4^2)=0$, d'o\`u la relation suivante:

\begin{equation}\label{eq:deux_eq_hyp_A4}
 \h_q(t^{\eta_1}\varphi_1,t^{\eta_2}\varphi_2)+\hk_{q}(t^{\eta_3}\varphi_3,t^2\varphi_4^2)=0.
\end{equation}

On remarque que le cas $2$) \'equivaut au cas $3$), quitte \`a permuter les variables $x_1$ et $x_2$. Il suffit donc de trouver une contradiction dans les Cas $1$) et $2$) pour achever la preuve de la Proposition \ref{pr:des_deux_famil_hyp_A4}.\\

Cas 1).  On suppose que  $\eta_1=\eta_2=1$ et $\eta_3= 2$.\\

 Comme $\mu_3=\eta_3=2$, on obtient que la s\'erie formelle $\varphi_3$ est inversible 
  et  que $\hk_{q}(t^{\eta_3}\varphi_3,t^{2}\varphi_4^2)=t^{2q}\hk_{q}(\varphi_3,\varphi_4^2)$.
  En vertu du lemme \ref{le:v-ord-hq_h'q}, on obtient que $\nu_v \hk_{q}(\varphi_3,\varphi_4^2)=0$. 
  Par cons\'equent, la s\'erie formelle  $\hk_{q}(\varphi_3,\varphi_4^2)$ est inversible.\\

D'apr\`es la Relation \ref{eq:deux_eq_hyp_A4}, on obtient la relation suivante:
\begin{center}
$ \h_q(\varphi_1,\varphi_2)+t^q\hk_{q}(\varphi_3,\varphi_4^2)=0$.
\end{center}

On rappelle que $\h_q$ est un polyn\^ome homog\`ene de degr\'e $q\geq 3$ sans facteur multiple et que $x_1$ et $x_2$ ne divisent pas $\h_q(x_1,x_2)$. On peut donc supposer que $\h_{q}(x_1,x_2)=\prod_{i=1}^{q}(b_ix_1+c_ix_2)$, o\`u $b_i,\;c_i\in \KK^{\star}$. Par cons\'equent, on a:
\begin{center}
 $\prod_{i=1}^{q}(b_i\varphi_1+c_i\varphi_2)=-t^q\hk_{q}(\varphi_3,\varphi_4^2)$.
\end{center}

Comme le polyn\^ome $\h_q$ n'a pas de facteur multiple, les s\'eries formelles $\varphi_1$ et $\varphi_2$ sont inversibles ou divisibles par $t$, ce qui est absurde.\\

Cas 2). On suppose que $\eta_1=\eta_3=1$ et $1\leq \eta_2\leq 2$.\\

 D'apr\`es la Relation \ref{eq:deux_eq_hyp_A4}, on obtient la relation suivante:
\begin{center}
$ \h_q(\varphi_1,t^{\eta_2-1}\varphi_2)+\hk_{q}(\varphi_3,t\varphi_4^2)=0$.
\end{center}

On rappelle que  $x_1$ et $x_2$ (resp. $x_3$ et $x_4$) ne divisent pas $\h_q(x_1,x_2)$ (resp. $\hk_q(x_3,x^2_4)$). On peut donc supposer que 
\begin{center}
$\h_{q}(x_1,x_2)=\prod_{i=1}^{q}(b_ix_1+c_ix_2)$, o\`u $b_i,\;c_i\in \KK^{\star}$.
\vspace{0.5cm}

 $\hk_{q}(x_3,x^2_4)=\prod_{i=1}^{q}(d_ix_3+f_ix^2_4)$, o\`u $d_i,\;f_i\in \KK^{\star}$.\\
\end{center}

On pose  $\gamma_i:=b_i\varphi_1+c_i\varphi_2t^{\eta_2-1}$
et $\gamma'_i:=d_i\varphi_2+f_i\varphi^2_4t$, $1\leq i\leq q$. Ainsi, on obtient la relation suivante:

\begin{equation}
\label{eq:gamma-gamma'}
\prod^{q}_{i=1}\gamma_i=-\prod^{q}_{i=1}\gamma'_i
 \end{equation}

Le lemme suivant nous permettra de construire un syst\`eme de relations entre  les s\'eries formelles $\varphi_i$, $1\leq i\leq 4$ 

\begin{lemme} \label{le:gamma-gamma'_irred} Les s\'eries formelles $\gamma_i$ et $\gamma'_i$, $1\leq i\leq q$, sont irr\'eductibles.
 \end{lemme}
\begin{proof} D'apr\`es le Lemme \ref{le:v-ord-hq_h'q}, il existe un vecteur $v=(u,1)\in \QQ^2_{>0}$ tel que  $\nu_v(\prod^{q}_{i=1}\gamma_i)=\degt{(\prod^{q}_{i=1}(\gamma_i)_v)}=q$ et $\nu_v(\prod^{q}_{i=1}\gamma'_i)=\degt{(\prod^{q}_{i=1}(\gamma'_i)_v)}=q$.\\

 Pour tout $i\in {1,2,...,q}$, les s\'eries formelles $\gamma_i$ et $\gamma'_i$ ne sont pas inversibles, car les s\'eries $\varphi_1$, $t^{\eta_2-1}\varphi_2$, $\varphi_3$ et $t\varphi_4$ ne le sont pas. Par cons\'equent, on a:

\begin{center}
 $\nu_v\gamma_i=\degt{((\gamma_i)_v)}=\degt{(b_1(\varphi_1)_v+t^{\eta_2-1}c_i(\varphi_2)_v)}=1$ et $\nu_v\gamma'_i=\degt{((\gamma'_i)_v)}=\degt{(d_i(\varphi_3)_v+tf_i(\varphi_4^2)_v)}=1$, 
\end{center}
pour tout $1\leq i\leq q$. Par cons\'equence,  on obtient que les s\'eries $\gamma_i$ et $\gamma'_i$ sont irr\'eductibles, pour tout $1\leq i\leq q$. \end{proof}

D'apr\`es le Lemme \ref{le:gamma-gamma'_irred} et la Relation \ref{eq:gamma-gamma'}, on peut supposer,  sans perte de g\'en\'eralit\'e, que les $\gamma_i$ et les $\gamma'_i$ satisfont les relations suivantes:

\begin{center}
 $\gamma_i=I_i\gamma'_i$, pour tout $1\leq i\leq q$,
\end{center}
o\`u les $I_i\in K[[s,t]]$ sont des s\'eries formelles inversibles telles que $\prod^{q}_{i=1}I_i=-1$. Autrement dit, on a le syst\`eme suivant:

\begin{center}
$b_i\varphi_1+c_i\varphi_2t^{\eta_2-1}-d_iI_i\varphi_3=f_iI_i\varphi^2_4t$, pour tout $1\leq i\leq q$, et $\prod^{q}_{i=1}I_i=-1$.
\end{center}

 Soit $M$ la matrice du syst\`eme ci-dessus, c'est-\`a-dire
\begin{center}
$M:= \begin{pmatrix} b_1& c_1 & -d_1I_1 \\  b_2& c_2 & -d_2I_2\\  \vdots & \vdots &\vdots \\ b_q& c_q & -d_qI_q \end{pmatrix}$
 \end{center}

On pose 

\begin{center}
$M_i:=\begin{pmatrix} b_1& c_1 & -d_1I_1 \\  b_2& c_2 & -d_2I_2 \\ b_i& c_i & -d_iI_i \end{pmatrix}$,
 \end{center}
pour $3\leq i\leq q$. Le lemme suivant ach\`eve  la d\'emonstration.

\begin{lemme} \label{le:rgM}  Il existe $3\leq i'\leq q$, tel que le d\'eterminant $\det M_{i'}$ est une s\'erie inversible.  
\end{lemme}
D'abord, en utilisant le Lemme \ref{le:rgM},  on montre qu'on arrive aussi  \`a une contradiction dans le Cas $2$). Ce qui ach\`eve la preuve de la Proposition \ref{pr:ser_for_inv_deux_cas_A4}.\\

 En vertu du Lemme \ref{le:rgM}, on peut supposer,  quitte \`a permuter les lignes de la matrice $M$, que la matrice  $M_3$ est inversible.  Comme on a  le syst\`eme 
\begin{center}
$b_i\varphi_1+c_i\varphi_2t^{\eta_2-1}-d_iI_i\varphi_3=f_iI_i\varphi^2_4t$, $1\leq i\leq 3$, 
\end{center}
 et que la matrice $M_3$ est inversible, on obtient que les s\'eries formelles $\varphi_1$, $\varphi_3$  sont divisibles par $t$. Ce qui est absurde.\\

\begin{proof}[{\it D\'emonstration du Lemme \ref{le:rgM}}]  Comme le polyn\^ome $\h_{q}(x_1,x_2)=\prod_{i=1}^{q}(b_ix_1+c_ix_2)$, $b_i,c_i\in \KK^{\star}$, n'a pas de facteur multiple, on obtient que 
$A:=\begin{pmatrix} b_1 & c_1 \\ b_2 & c_2 \end{pmatrix}$
 est de rang $2$. \\ 

Raisonnons par l'absurde, on suppose que pour tout $3\leq i\leq q$ le d\'eterminant $\det M_i$ n'est pas une s\'erie inversible. Comme la matrice $A$ est de rang $2$,  il existe  $\lambda_i, \lambda'_i\in \KK^{\star}$ , $1\leq i\leq q$, tels que $(b_i,c_i)=\lambda_i(b_1,c_1)+\lambda'_i(b_2,c_2)$. Ainsi on obtient:

\begin{center}
 $\det M_i=(b_1c_2-b_2c_1)(d_iI_i-\lambda_id_1I_1+\lambda'_id_2I_2)$,
\end{center}
pour tout $3\leq i\leq q$. En utilisant le syst\`eme

\begin{center}
$b_i\varphi_1+c_i\varphi_2t^{\eta_2-1}-d_iI_i\varphi_3=f_iI_i\varphi^2_4t$, pour tout $1\leq i\leq q$,
\end{center}
on obtient les relations suivantes 

\begin{center}
 $(d_iI_i-\lambda_id_1I_1+\lambda'_id_2I_2)\varphi_3=(f_iI_i-\lambda_if_1I_1+\lambda'_if_2I_2)\varphi^4t$,  pour tout $1\leq i\leq q$.
\end{center}

 Comme  les s\'eries $\det  M_i$ ne sont  pas inversibles, on obtient 
  les relations suivantes:

\begin{center}
$\lambda_id_1j_1+\lambda'_id_2j_2=d_ij_i$, 

\vspace{0.5cm}

$\lambda_if_1j_1+\lambda'_if_2j_2=f_ij_i$,
 \end{center}
o\`u $j_i$ est le terme constant de la s\'erie inversible $I_i$, $1\leq i\leq q$. Comme le polyn\^ome $\hk_{q}(x,y)=\prod_{i=1}^{q}(d_ix+f_iy)$, $b_i,f_i\in \KK^{\star}$, n'a pas de facteur multiple, on obtient que la matrice  $\begin{pmatrix} d_1 & f_1 \\ d_2 & f_2 \end{pmatrix}$
est inversible. Par cons\'equent, pour tout $i\in \{1,2,..,q\}$, il existe $l_i\in \KK^{\star}$ tel que 

\begin{center}
 $j_i=l_ij_1$.
\end{center}

Comme $\prod^{q}_{i=1}I_i=-1$, on obtient que $(\prod_{i=1}^{q}l_i)j_{1}^q=-1$. Par cons\'equent, on a  $j_1\in \KK^{\star}$.\\

Soient $\alpha_i$, $\widehat{\alpha}_i$, $a_i$, $\widehat{a}_i$, $i\leq i\leq 4$, et $\lambda^{\star}$ comme dans la d\'emonstration du Lemme \ref{le:v-ord-hq_h'q}.

On rappelle que $\sigma_{11}\in \Sigma$ est le c\^one  engendr\'e par les vecteurs $\rho_0=(2,2,2,1)$, $\rho_1=(1,1,1,1)$, $\e_2=(0,1,0,0)$ et $\e_3=(0,0,1,0)$ et que la restriction du morphisme $\pi:X(\Sigma)\rightarrow \AF^4_{\KK}$ \`a l'ouvert $U:=U_{\sigma_{11}}$ est donn\'ee de la fa\c con suivante:

\begin{center}
 $\pi:U\rightarrow \AF^{4}_{\KK}$, $(y_1,y_2,y_3,y_4)\mapsto (y_1^2y_4,y_1^2y_2y_4,y_1^2y_3y_4,y_1y_4)$
\end{center}
 
 L'intersection de $\widetilde{V}$ (resp. $\E_2$) et $U$  est donn\'ee par l'\'equation suivante (resp. les \'equations suivantes):
\begin{center}
 $\h_q(1,y_2)+\hk_{q}(y_3,y_4)=0$  (resp. $\h_q(1,y_2)+\hk_{q}(y_3,y_4)=0$  et $y_1=0$)
\end{center} En utilisant le morphisme $\pi$, on obtient que:
 \begin{center}
$\begin{array}{ccccccc}
\alpha_1 & = &\widehat{\alpha}_1^2\widehat{\alpha}_4,& \; & \alpha_2& =&\widehat{\alpha}_1^2\widehat{\alpha}_2\widehat{\alpha}_4,\\

\alpha_3 & = & \widehat{\alpha}_1^2\widehat{\alpha}_3\widehat{\alpha}_4,&\;& \alpha_4& =&\widehat{\alpha}_1\widehat{\alpha}_4.
\end{array}$
\end{center}

Maintenant on consid\`ere la relation 
\begin{center}
$b_i\varphi_1+c_i\varphi_2t^{\eta_2-1}-d_1I_1\varphi_3=f_1I_1\varphi^2_4t$
\end{center}
 En utilisant le homomorphisme $\lambda^{\star}$ on obtient la relation suivante

\begin{equation}
\label{eq:relS}
 b_1 +c_1\widehat{a}_2-d_1j_1\widehat{a}_3-f_1j_1\widehat{a}_4^2=0.
\end{equation}
 Ce qui est absurde. En effet, on consid\`ere le polyn\^ome 
\begin{center}
$\g(x_1,x_2,x_3,x_4)=b_1+c_1y_2-d_1j_1y_3-f_1j_1y_4$.
\end{center}
Comme $j_1\in \KK^{\star}$, on a $\g\in \KK[y_1,y_2,y_3,y_4]$. \\

Soit $S$ l'adh\'erence dans $\widetilde{V}$ de l'hypersurface de $U$ donn\'ee l'\'equation $g=0$. En vertu de la relation \ref{eq:relS}, on obtient que  $\widehat{\alpha}(0)$ appartient \`a l'hypersurface $S$, o\`u $\widehat{\alpha}$ est le rel\`evement de $\alpha$ \`a $\widetilde{V}$ et $0$ est le point ferm\'e de $K[[t]]$.

 On remarque que les diviseurs $S$ et $\E_2$  sont transverses.  Alors $\widehat{\alpha}(0)$ appartient \`a l'intersection de $S$ et $\E_2$ ce qui est une contradiction car $\widehat{\alpha}(0)$ est le point g\'en\'erique de $\E_2$.
\end{proof}
On a montr\'e que les Cas $1$) et $2$) conduisent \`a une contradiction. Par cons\'equent, on en d\'eduit que les 
s\'eries formelles $\varphi_i$, $1\leq i \leq 4$, sont inversibles.
\end{proof}

\end{document}